\documentclass{article}
\usepackage[12pt]{extsizes}
\usepackage [utf8]{inputenc}
\usepackage[T1]{fontenc}
\usepackage{amsmath,amssymb,amsthm}
\usepackage[all]{xy}
\usepackage{enumitem}
\usepackage{xcolor,mathrsfs} 
\definecolor{vertfonce}{rgb}{0.20, 0.46, 0.25}
\definecolor{rougefonce}{rgb}{0.64, 0.09, 0.20}
\usepackage[breaklinks=true,
            colorlinks=true,
            linkcolor=rougefonce,
            citecolor=vertfonce,
            hypertexnames=false]{hyperref}

\usepackage{autonum} 

\newtheorem{theo}{Theorem}[section]
\newtheorem{lemm}[theo]{Lemma}
\newtheorem{prop}[theo]{Proposition}
\newtheorem{coro}{Corollary}[theo]
\newtheorem{defi}[theo]{Definition}

\newtheorem{remark}[theo]{Remark}

\newcommand{\Vois}{\mathrm{Neigh}}
\newcommand{\Op}{\mathrm{Op}}
\newcommand{\Hol}{\mathrm{Hol}}
\newcommand{\loc}{\mathrm{loc}}
\newcommand{\Brg}{\mathrm{Brg}}
\newcommand{\Top}{\mathop{\textup{Top}}}

\newcommand{\w}{\mathrm{w}}
\newcommand{\abs}[1]{\left|#1\right|}
\newcommand{\norm}[1]{\left\|#1\right\|}
\newcommand{\tnorm}[1]{\|#1\|}
\newcommand{\Opbrg}{\Op^{\Brg}}
\newcommand{\lphi}{L^2_\Phi}
\newcommand{\dbar}{\overline\partial}
\newcommand{\DD}{\:\!\mathrm{d}}
\newcommand{\trsp}{\raisebox{.6ex}{${\scriptstyle t}$}}

\newcommand{\Cinf}{\mathscr{C}^\infty}
\newcommand{\spt}{\mathop{\textup{supp}}}
\newcommand{\adiag}{\textup{adiag}}
\newcommand{\phy}{\varphi}
\newcommand{\intint}{\int\!\!\!\!\int}

\newcommand{\RM}{\mathbb{R}}

\newcommand{\NM}{\mathbb{N}}
\newcommand{\CM}{\mathbb{C}}
\renewcommand{\O}{\mathcal{O}}
\newcommand{\h}{\hbar}
\newcommand{\N}{\mathcal{N}}
\newcommand{\pscal}[2]{\left(#1 | #2\right)}
\newcommand{\ham}[1]{\mathcal{X}_{#1}}
\newcommand{\hnegl}{\equiv_\Phi}

\title{Analytic Bergman operators in the semiclassical limit}

\author{Ophélie \textsc{Rouby}\footnote{Lycée Edouard Branly, 2, rue
    Porte Gayole 62321 Boulogne-sur-mer, France}, Johannes
  \textsc{Sjöstrand}\footnote{Université de Bourgogne, CNRS, IMB - UMR
    5584, BP 47870, F-21078 Dijon, France}~ and \textsc{V\~u Ng\d oc}
  San\footnote{Univ Rennes, CNRS, IRMAR - UMR 6625, F-35000 Rennes,
    France}}

\date{}

\begin{document}
\numberwithin{equation}{section}

\maketitle

\begin{abstract}
  Transposing the Berezin quantization into the setting of analytic
  microlocal analysis, we construct approximate semiclassical Bergman
  projections on weighted $L^2$ spaces with analytic weights, and show
  that their kernel functions admit an asymptotic expansion in the
  class of analytic symbols. As a corollary, we obtain new estimates
  for asymptotic expansions of the Bergman kernel on $\CM^n$ and for
  high powers of ample holomorphic line bundles over compact complex
  manifolds.
\end{abstract}
\maketitle

\section{Introduction}

Let $L\to X$ be a holomorphic line bundle over a closed complex
manifold $X$, and assume that $L$ is equipped with a positive
Hermitian metric.  The corresponding Chern form induces a Riemannian
metric on $X$, and the integrated scalar product on $L$ gives a
natural Hilbert space structure on sections of $L$. In this work we
will be interested in the so-called \emph{semiclassical limit} $L^k$
of high tensor powers of the line bundle $L$, where $\frac{1}{k}$
plays the role of the small Planck constant $\h$. The line bundle
$L^k$ is naturally equipped with the product Hermitian metric, and we
may consider the Hilbert space $L^2(X;L^k)$ of square-integrable
sections of $L^k$.  The orthogonal projection onto holomorphic
sections:
\[
\Pi_k : L^2(X;L^k) \to H^0(X; L^k)
\]
is called the Bergman projection. A central question in complex
geometry is to understand the asymptotic behavior, as $k\to+\infty$,
of the distributional kernel $K(x,y;k)$ of $\Pi_k$. The same problem
arises in the sister theory of the Szegö projection, for which one
considers the boundary of a strictly pseudo-convex domain in $\CM^n$.
After the pioneer works of Fefferman~\cite{fefferman74}, Boutet de
Monvel--Sjöstrand~\cite{boutet-sjostrand76} and
Kashiwara~\cite{kashiwara77}, their techniques have been transposed
over to compact complex manifolds.  In particular, thanks to the works
by Bouche~\cite{bouche90}, Tian~\cite{tian90}, and then
Catlin~\cite{catlin99} and Zelditch~\cite{zelditch98}, a complete
asymptotic expansion of the Bergman function (the norm of the Bergman
kernel on the diagonal) was given:
\[
\abs{K(x,x;k)}_{L^k_x} \sim {k^n}\left(b_0(x) + \frac{b_1(x)}{k} +
  \frac{b_2(x)}{k^2} + \cdots\right)
\]
where the coefficients $b_j$ are smooth functions on $X$, and the
symbol $\sim$ stands for asymptotic expansion with respect to the
powers of the small parameter $\frac{1}{k}$, in the $\Cinf$
topology. See for instance the expository works~\cite{berndtsson03},
\cite{ma-marinescu}, and the references therein.
In~\cite{zelditch98}, this expansion was used to obtain approximations
of arbitrary Kähler metrics by Fubini-Study metrics obtained via the
Kodaira embeddings; this result, which can be seen as a semiclassical
interpretation of some constructions by Donaldson~\cite{Donaldson96},
is now called the Tian-Yau-Zelditch theorem, see~\cite{Phong-Sturm07}.

Since then, there has been an intense activity on getting a better
understanding of the expansion of $K(x,y;k)$ --- including the
analysis away from the diagonal $x=y$, applying it to new more general
settings, and exploring new applications. For instance,
in~\cite{Song-Zelditch10}, the authors show that the local Bergman
expansion can yield approximation results on global geometric objects
like geodesics in the space of Kähler metrics. Recently, a very
interesting program about ``partial'' Bergman kernels has been
started, where one considers the projections on specific subspaces of
holomorphic functions, see~\cite{Zelditch-Zhou19}. The robustness of
the Bergman expansion manifests itself in the way that several authors
have been able to generalise it to the non-Kähler symplectic manifolds
(see~\cite{Charles2016} for a very general treatment,
and~\cite{kordyukov18} for the use of the Bochner Laplacian for
obtaining analogues of the spaces $H^0(X; L^k)$).

In this paper, we wish to consider the issue of the relationships
between the analyticity of the metric on $L$ and optimal estimates for
$b_k$ on or off the diagonal. The question has raised recent interest,
see for instance the articles~\cite{christ-13}, \cite{christ-13b},
\cite{hezari-lu-xu17}. In this last paper, the authors prove that, if
the metric is analytic (in which case the functions $b_k$ are analytic
on $X$), the estimate $\abs{\tilde b_k(x,y)} \leq C^k k!^2$ holds
locally uniformly near the diagonal in $X\times X$, where $\tilde b_k$
is the holomorphic extension of $b_k$.  But they also present the
conjecture formulated by Zelditch in 2014 that a stronger, more
natural estimate
\begin{equation}
  \label{equ:natural-analytic-estimate}
  \abs{\tilde b_k(x,y)} \leq C^k k!
\end{equation}
could hold, and relate various debates on this issue. This question
was also asked to us by S. Zelditch and L. Charles in 2014 and 2016.
One of our main results, Theorem~\ref{theo:line-bundle} below, settles
the conjecture in a positive way, showing that the more natural
version~\eqref{equ:natural-analytic-estimate} is correct and hence
that we have exponentially accurate approximations. As a consequence
of the existence of such asymptotics, the Bergman kernel admit lower
and upper bounds of the form $k^ne^{-k d^2(x,y)/C}$.  Recently, the
upper bound has been obtained independently by Hezari and
Xu~\cite{hezari-xu18}, and an alternative approach to Theorem
\ref{theo:line-bundle} was proposed
in~\cite{deleporte18,charles2019analytic,hezari2019property}.

But this result was not our unique goal.  In fact, our initial
motivation for undertaking this research has its roots in the spectral
theory of Berezin-Toeplitz operators. Starting from a prequantizable
Kähler manifold $X$, one can construct a Hermitian line bundle as
above, and define an algebra of operators extending the usual
geometric quantization scheme, see ~\cite{berezin75,BG}, and
also~\cite{charles-toeplitz,lefloch-book} and references therein.
Such operators are defined by a `symbol' $f$, which is a function on
$X$, through the formula
\[
T_f : u \mapsto \Pi_k (f u) : H^0(X;L^k) \to H^0(X;L^k)\,.
\]
In~\cite{rouby-17} it was conjectured that, for Berezin-Toeplitz
operators with analytic symbols on a Riemann surface, one has a very
accurate asymptotic description, in the semiclassical limit
$k\to+\infty$, of all individual eigenvalues, provided that the
operator is `nearly selfadjoint'. The conjecture was supported by the
proof of this result in the case of analytic pseudo-differential
operators acting on $L^2(\RM)$ or $L^2(S^1)$~\cite{rouby-17}, and,
under some additional geometric assumption (related to complete
integrability), on $L^2(\RM^2)$~\cite{melin-sjostrand-non,
  h-sj-04,h-sj-vn-07}.

Although the idea of transposing these results to the Berezin-Toeplitz
case is natural, analytic microlocal analysis was never applied to
general Berezin-Toeplitz operators, and there was a fundamental
obstacle to this. Namely, one should prove that the Bergman projection
$\Pi_k$, viewed as a complex Fourier integral operator, has an
analytic symbol with a suitable asymptotic expansion. Building the
necessary theory and finally proving this result constitutes the core
of this article, see Theorems~\ref{theom}, \ref{bcn4}
and~\ref{theo:line-bundle}. We hope that these results will allow to
improve other semiclassical estimates involving analytic Hamiltonians
like~\cite{hitr-mant-sjoe17}, and to reach phenomena that are
inherently exponentially small, such as tunneling; see for
instance~\cite{bonth-raym-vns19}.

\medskip

To conclude this introduction, we would like to give an informal
overview of the method. As mentioned above, it had been realized for a
long time that the asymptotic study of the Bergman kernel is tightly
related to semiclassical analysis,
see~\cite{boutet-sjostrand76,zelditch98}. More recently, other
approaches have been proposed that derive asymptotic expansions in a
more `elementary' (but still semiclassical) way, see for
instance~\cite{b-b-sj-08}. The techniques that we use in this article
also go back to the initial ideas, but in a more systematic and
natural way: we combine a fully microlocal approach with $L^2$
estimates, in order to obtain a transparent 'local-to-global'
principle. The first step is to construct an approximate Bergman
projection by means of analytic microlocal analysis, via a
quantization scheme which was initially advocated by
Charles~\cite{charles-toeplitz}, that we call Bargmann-Bergman (or
$\Brg$ for short) quantization. By its local nature, this step does
not require any geometric assumption on the underlying phase
space. The second step uses an $L^2$-analysis of these operators
(combined with the usual Hörmander $\dbar$ estimates) to show that, up
to an exponentially small error in terms of the semiclassical
parameter, the exact Bergman projection coincides with the
microlocally constructed one. Once this is established, the estimates
for the asymptotics of the Bergman kernel are a consequence of the
pseudo-differential calculus in analytic classes of symbols developed
in~\cite{sj-asterisque-82}. We have applied this second step to two
cases, namely $\CM^n$ (Section~\ref{sec:bcn}) and compact complex
manifolds (Section~\ref{sec:line-bundle}). It would be interesting to
apply the same program to more general non-compact geometries.
% cotangent?

\subsection*{Organization of the article}

In section~\ref{sec:brg} we introduce the `$\Brg$' quantization on
weighted spaces of germs of holomorphic functions $H_{\Phi,x_0}$,
where $x_0\in\CM^n$ and $\Phi$ is an analytic, strictly
plurisubharmonic function defined near $x_0$
(Definition~\ref{defi:brg-quantization}). This particular form of
quantization is directly inspired by the well-known exact formula for
the Bergman projection in the setting of weighted $L^2$ spaces on
$\CM^n$, when the weight is quadratic
(Proposition~\ref{prop_defi_Brg_projection}); it was introduced by
Charles~\cite{charles-toeplitz} in the setting of $\Cinf$
Berezin-Toeplitz operators.

Section \ref{sec:proof} contains the main microlocal result of the
paper, namely: using an analytic Fourier integral operator, one
obtains a microlocal equivalence between the $\Brg$ quantization and
the `usual' (but complex) Weyl quantization (Theorem~\ref{theom}). The
proof consists in expressing this equivalence as a product of analytic
Fourier integral operators, and proving transverse intersection of the
underlying canonical relations.

In Section~\ref{sec:appr-bergm-proj}, we cast the microlocal result in
terms of approximate weighted $L^2$ spaces on nested domains. This
allows the construction of approximate Bergman projections
(Proposition~\ref{prop:approx_proj}), which are unique modulo an
exponentially small error (Proposition~\ref{prop:uniqueness}).

Sections~\ref{sec:bcn} and~\ref{sec:line-bundle} contain the main
applications to the asymptotics of the Bergman kernel. In
Section~\ref{sec:bcn} we treat the case of weighted $L^2$ spaces on
$\CM^n$, while Section~\ref{sec:line-bundle} deals with the Bergman
projection associated with a holomorphic Hermitian line bundle.

\vskip 1em

\paragraph{Acknowledgements.}
The question of proving the analytic behaviour of the Bergman kernel
was raised to one of us (J.S.) by Steve Zelditch and Maciej Zworski,
and independently to the three of us by Laurent Charles. We are
grateful to them for giving us the impulse to write this paper.  We
also thank Laurent Charles for explaining the anzatz he used
in~\cite{charles-toeplitz} which is at the origin of our $\Brg$
quantization. Finally, we are grateful to the referees for their
careful reading and suggestions.  Funding for O.R. was provided in
part by the Funda\c{c}\~ao para a Ci\^encia e a Tecnologia (FCT)
through project PTDC/MAT-CAL/4334/2014.

\tableofcontents
% https://tex.stackexchange.com/questions/53513/hyperref-token-not-allowed

\section{\texorpdfstring{$\Brg$}{Brg} quantization}
\label{sec:brg}

\subsection{FBI-Bargmann transforms and Bergman kernels}

For the sake of completeness, and in order to introduce the relevant
notation, we discuss here the FBI-Bargmann transform, which is in fact
a generalization of the original Segal-Bargmann and FBI transforms
from~\cite{bargmann} and~\cite{bros-iagolnitzer} to the case of a
general strictly plurisubharmonic quadratic form $\Phi : \CM^n\to \RM$
(see Definition \ref{defi_plurisubharmonic_fonctions}). The original
case investigated by Bargmann corresponds to
$\Phi(z)= \frac{1}{2} \abs{z}^2$; the corresponding transform has been
used in various settings under different names: Bargmann-Segal, Gabor,
or wavepacket transforms. The general case was studied by several
authors; one can find a good account of the theory in the
book~\cite[Chapter 13]{zworski-book-12}. In~\cite{sj-asterisque-82},
\cite{sj-96} these transformations are treated as Fourier integral
operators and integrated into microlocal (semiclassical) analysis.

We present here the semiclassical version.  Let $ 0 < \hbar \leq 1$ be
the semiclassical parameter.  Without explicit notice, all constants
in this text are implicitly independent of $\h$.

The following definition was introduced in~\cite{sj-asterisque-82} and
specifically in~\cite[(1.3)]{sj-96} in the quadratic case, where the
constant $c_\phy$ was computed.  In the wording
of~\cite{zworski-book-12}:
\begin{defi} \label{defi_FBI_transform} Let $ \phi(z, x)$ be a
  holomorphic quadratic function on $ \CM^n \times \CM^n$ such that:
  \begin{enumerate}[label=\roman*)]
  \item $ \Im \left( \dfrac{\partial^2 \phi}{\partial x^2} \right)$ is
    a positive definite matrix;
  \item
    $ \det \left( \dfrac{\partial^2 \phi}{\partial x \partial z}
    \right) \neq 0$.
  \end{enumerate}
  The \textbf{FBI-Bargmann transform} associated with the function
  $\phi$ is the operator, denoted by $T_{ \phi}$, defined on the
  Schwartz space $ \mathscr{S}( \mathbb{R}^n)$ by:
  \[
  T_{ \phi} u(z) = c_{ \phi} \hbar^{-3n/4} \int_{ \mathbb{R}^n}
  e^{(i/\hbar) \phi(z,x)} u(x)\DD x, \quad z\in\CM
  \]
  where:
  \begin{equation} \label{formule_constante_c_phi} c_{ \phi} =
    \dfrac{1}{2^{n/2} \pi^{3n/4}} \dfrac{| \det \partial_x \partial_z
      \phi|}{(\det \Im \partial^2_x \phi )^{1/4}}.
  \end{equation}
\end{defi}
Let us define the complex linear canonical transformation
$\kappa_\phi$ by
\begin{align}
  \kappa_{ \phi}: \CM^n \times \CM^n 
  & \longrightarrow \CM^n \times \CM^n \\
  (x, - \partial_x \phi(z, x)) & \longmapsto (z, \partial_z \phi(z, x)).
\end{align}

\begin{defi} \label{defi_plurisubharmonic_fonctions} A function
  $ \Phi \in \mathscr{C}^2( \CM^n; \mathbb{R})$ is called
  \textbf{plurisubharmonic} (respectively \textbf{strictly
    plurisubharmonic}) if, for all $x \in \CM^n$, the matrix
  $ ( \partial^2_{x_j, \bar{x}_k} \Phi )_{j, k=1}^n$ is positive
  semidefinite (respectively positive definite).
\end{defi}
We shall often identify the matrix
$ ( \partial^2_{x_j, \bar{x}_k} \Phi )_{j, k=1}^n$ (where $j$ is the
line index and $k$ the column index) with the $(1,1)$-form
$\partial_{\bar x}\partial_x \Phi = \sum_{j,k} \partial^2_{x_j,
  \bar{x}_k} \Phi\, \DD \bar x_k \wedge \DD x_j $.
\begin{prop}[\cite{sj-96}]
  With the notation of Definition \ref{defi_FBI_transform}, define for
  $z \in \CM^n$:
  \begin{equation} \label{eq_defi_Phi} \Phi(z) := \max_{x \in
      \mathbb{R}^n} - \Im \phi(z, x) .
  \end{equation}
  Then $\Phi$ is a strictly plurisubharmonic quadratic function and
  $\kappa_{\phi}$ is a bijection from $ \mathbb{R}^{2n}$ to
  \begin{equation}
    \label{equ:Lambda_Phi}
    \Lambda_{ \Phi} = \left\lbrace \left(z, \dfrac{2}{i} \dfrac{\partial
          \Phi}{\partial z}(z) \right); z \in \CM^n
    \right\rbrace.
  \end{equation}
\end{prop}

Throughout the paper, we shall use the following notation.
\begin{itemize}
\item $L(\DD z)$ is the Lebesgue measure on $ \CM^n$, \textit{i.e.}
  \begin{equation}
    \label{equ:lebesgue}
    L(\DD z)= \prod_{j=1}^n \left( 
      \dfrac{i}{2}\DD z_j \wedge\DD \bar{z}_j \right) =: 
    \left( \frac{i}{2} \right)^n\DD z \wedge\DD \bar{z}. 
  \end{equation}
\item
  $L^2_{\Phi}( \CM^n) := L^2( \CM^n, e^{-2 \Phi(z)/ \hbar} L(\DD z))$
  is the set of measurable functions $f:\CM^n\to\CM$ such that:
  \[
  \int_{ \CM^n} |f(z)|^2 e^{-2 \Phi(z)/ \hbar} L(\DD z) < + \infty.
  \]
\item $H_{\Phi}( \CM^n) := \Hol( \CM^n) \cap L^2_{ \Phi}( \CM^n)$ is
  the closed subspace of holomorphic functions in
  $L^2_{\Phi}( \CM^n)$.
\item If $z,w\in\CM^n$, $z=(z_1,\dots,z_n)$ and $w=(w_1,\dots,w_n)$,
  then we denote by $z\cdot w$ the `complex scalar product',
  \emph{i.e.}
  \[
  z \cdot w := \sum_{j=1}^n z_j w_j \; .
  \]
\end{itemize}

\begin{prop}[{\cite[Formula~(1.12)]{sj-96}}] \label{prop_defi_Brg_projection}
  Let $ \Phi$ be the strictly plurisubharmonic quadratic function
  defined by Equation \eqref{eq_defi_Phi}, and let $\psi$ be the
  unique holomorphic quadratic form on $ \CM^n \times \CM^n$ such
  that, for all $z \in \CM^n$:
  \begin{equation}
    \psi(z, \bar{z}) = \Phi(z) \,.
    \label{equ:psi}  
  \end{equation}
  The following properties hold.
  \begin{enumerate}[label=\roman*)]
  \item The orthogonal projection
    $ \Pi_{\Phi}: L^2_{ \Phi}( \CM^n) \to H_{\Phi}( \CM^n)$ is given
    by:
    \begin{equation}
      \Pi_{ \Phi} u(z) = \dfrac{2^n \det ( \partial_{z\bar z}^2 \Phi )}{( \pi \hbar)^n} 
      \int_{ \CM^n} 
      e^{\frac{2}{\h}( \psi(z, \bar{w}) - \Phi(w))} u(w) L(\DD w) .
      \label{equ:quadratic-projection}
    \end{equation}
    
  \item $T_{ \phi}: L^2( \mathbb{R}^n) \to H_{ \Phi}( \CM^n)$ is a
    unitary transformation and if
    $T_{\phi}^*: L^2_{\Phi}( \CM^n) \to L^2( \mathbb{R}^n)$ is the
    adjoint of $T_{\phi}$, then $\Pi_{ \Phi} = T_{\phi} T_{\phi}^*$.
  \end{enumerate}
  The operator $ \Pi_{\Phi}$ is called the Bergman projection onto
  $H_{ \Phi}$.
\end{prop}

The FBI transform allows to obtain a correspondence between Weyl
operators acting on $L^2( \mathbb{R}^n)$ and Weyl operators acting on
$H_{\Phi}( \CM^n)$, introduced in~\cite{sj-96}.  Let
$S( \mathbb{R}^{2n})$ denote the following symbol class:
\[
S( \mathbb{R}^{2n}) = \lbrace a \in \Cinf( \mathbb{R}^{2n}); \quad
\forall \alpha \in \mathbb{N}^{2n}, \exists C_{ \alpha} > 0,
| \partial^{ \alpha} a | \leq C_{ \alpha} \rbrace.
\]
Using the parametrization~\eqref{equ:Lambda_Phi} of
$ \Lambda_{ \Phi} \simeq \CM^n$, we will also use the class of symbols
$S( \Lambda_{ \Phi})$ that we identify with
$S( \CM^n) \simeq S( \mathbb{R}^{2n})$.

\begin{defi}
  Let $a_{ \hbar} \in S( \mathbb{R}^{2n})$. Define the \textbf{Weyl
    quantization} of $a_{ \hbar}$, denoted by
  $ \Op^{\w}_\h(a_{ \hbar})$, by the following formula, for
  $u \in \mathscr{S}( \mathbb{R}^{n})$:
  \[
  \left[\Op^{\w}_\h(a_{ \hbar}) u \right] (x) = \dfrac{1}{(2 \pi
    \hbar)^n} \intint_{\mathbb{R}^{2n}} e^{\frac{i}{\h}(x-y) \cdot
    \xi} a_{ \hbar} \left( \tfrac{x+y}{2}, \xi \right) u(y)\DD y\DD
  \xi.
  \]
  Then $a_{ \hbar}$ is called the (Weyl) symbol of the
  pseudo-differential operator $ \Op^{\w}_\h(a_{ \hbar})$.
\end{defi}
By the Calderón-Vaillancourt theorem (see for instance~\cite[Theorem
7.11]{dimassi-sjostrand}, \cite[Theorem 2.8.1]{martinez-book} or
\cite[Theorem 4.23]{zworski-book-12}), such an operator
$ \Op^{\w}_\h(a_{ \hbar})$ extends to a bounded operator on
$L^2(\RM^n)$ whose operator norm is bounded by a constant independent
of $\h$.

\begin{defi}
  \label{defi:complex-weyl}
  Let $b_{ \hbar} \in S( \Lambda_{ \Phi})$.  The \textbf{complex Weyl
    quantization} of the symbol $b_{ \hbar}$ is the operator given by
  the contour integral:
  \[
  [\Op^{\w}_{ \Phi}(b_{ \hbar}) u] (z) = \dfrac{1}{(2 \pi \hbar)^n}
  \intint_{\Gamma(z)} e^{(i/\hbar) (z-w) \cdot \zeta} b_{ \hbar}
  \left( \tfrac{z+w}{2}, \zeta \right) u(w)\DD w\DD \zeta,
  \]
  where
  $ \Gamma(z) = \left\lbrace (w, \zeta) \in \CM^{2n}; \zeta =
    \dfrac{2}{i} \dfrac{\partial \Phi}{\partial z} \left(
      \tfrac{z+w}{2} \right) \right\rbrace$.
\end{defi}

The complex Weyl quantization can be related to the usual Weyl
quantization be the following result, which is an instance of the
``exact Egorov theorem'':
\begin{prop}[\cite{h-sj-04}]\label{prop_transfor_Bargmann_et_quantif_weyl}
  Let $a_{ \hbar} \in S( \mathbb{R}^{2n})$. We have
  \begin{equation}
    T_{ \phi} \Op^{\w}_\h(a_{ \hbar}) T_\phi^*  = \Op^{\w}_{ \Phi}(b_{ \hbar})
  \end{equation}
  where the symbol $b_{ \hbar}$ is given by
  $ b_{ \hbar} = a_{ \hbar} \circ \kappa_{ \phi}^{-1}$, thus
  $b_{ \hbar} \in S( \Lambda_{ \Phi})$.
\end{prop}
In particular,
$ \Op^{\w}_{ \Phi}(b_{ \hbar}) : H_{ \Phi}( \CM^n) \to H_{ \Phi}(
\CM^n)$ is uniformly bounded with respect to $\hbar$.

There exists also a connection between the Berezin-Toeplitz
quantization and the complex Weyl quantization $\Op^{\w}_{ \Phi}$ of
Definition~\ref{defi:complex-weyl} above. Let us first recall the
definition of the Berezin-Toeplitz quantization of $ \CM^n$.

\begin{defi}
  Let $f_{\hbar} \in S( \CM^n)$. Define the \textbf{Berezin-Toeplitz
    quantization} of $f_{\hbar}$ as:
  \[
  T_{f_{\hbar}} := \Pi_{ \Phi} M_{f_{_\hbar}} \Pi_{ \Phi} ,
  \]
  where $M_{f_{\hbar}}: L^2_{\Phi}( \CM^n) \to L^2_{\Phi}( \CM^n)$ is
  the operator of multiplication by the function $f_{\hbar}$. We call
  $f_{\hbar}$ the {symbol} of the {Berezin-Toeplitz operator}
  $T_{f_{\hbar}}$.
\end{defi}

\begin{remark}
  In \cite{sj-96}, Berezin-Toeplitz operators on $ \CM^n$ were denoted
  by $ \widetilde{\Op}_{ \hbar, 0}$.
\end{remark}

The relation between the Berezin-Toeplitz and the complex Weyl
quantizations of $ \CM^n$ is given in the next proposition, where we
identify $\Lambda_\Phi$ with $\CM^n$.
\begin{prop}[{\cite[(1.23)]{sj-96}}, {\cite[Theorem
    13.10]{zworski-book-12}}] \label{prop_Toeplitz=pseudo_H(Phi)} $ $
  \begin{enumerate}[label=\roman*)]
  \item \label{item:f-to-b} Let $f_{\hbar} \in S( \CM^n)$ admit an
    asymptotic expansion in powers of $\hbar$ (in the topology of
    $S(\CM^n)$). Let $T_{f_{\hbar}}$ be the Berezin-Toeplitz operator
    of symbol $f_{\hbar}$. Then, we have:
    \[
      T_{f_{\hbar}} = \Op^{\w}_{ \Phi}( b_{ \hbar}) \quad \text{ in }
      \quad \mathcal{L}(H_{\Phi}( \CM^n)),
    \]
    where $b_{ \hbar} \in S( \Lambda_{ \Phi})$ admits an asymptotic
    expansion in powers of $\hbar$ given, for all
    $z \in \Lambda_{ \Phi} \simeq \CM^n$, by
    \begin{equation}
      \label{equ:f-to-b}
      b_{ \hbar}(z) = \exp \left( \dfrac{\hbar}{4} \left\langle
          \left( \partial^2_{z \bar{z}} \Phi
          \right)^{-1} \partial_z, \partial_{ \bar{z}}
        \right\rangle \right) ( f_{\hbar}(z)) \quad \text{ in } \quad
      S(\Lambda_{\Phi}).
    \end{equation}

  \item \label{item:b-to-f} Let $b_{ \hbar} \in S( \Lambda_{ \Phi})$
    admit an asymptotic expansion in powers of $ \hbar$. Then, there
    exists a function $f_{\hbar} \in S( \CM^n)$ such that:
    \[
      \Op^{\w}_{ \Phi}(b_{ \hbar}) = T_{f_{\hbar}} +
      \mathcal{O}(\hbar^{\infty}) \quad \text{ in } \quad
      \mathcal{L}(H_{\Phi}( \CM^n)),
    \]
    where $T_{f_{\hbar}}$ is the Berezin-Toeplitz operator of symbol
    $f_{\hbar}$, and $f_\h$ admits the following asymptotic expansion
    in powers of $\hbar$
    \begin{equation}
      \label{equ:b-to-f}
      f_{ \hbar}(z) \sim \exp \left( \dfrac{-\hbar}{4} \left\langle
          \left( \partial^2_{z \bar{z}} \Phi
          \right)^{-1} \partial_z, \partial_{ \bar{z}} \right\rangle
      \right) ( b_{\hbar}(z)) \quad \text{ in } \quad S(\CM^n).
    \end{equation}
  \end{enumerate}
\end{prop}

\begin{remark}
  In Item~\ref{item:f-to-b}, we actually don't need $f_\h$ to admit an
  asymptotic expansion in powers of $\h$. Then $b_\h$ is given
  by~\eqref{equ:f-to-b}, which is an exact formula, corresponding to
  solving a heat equation in positive time. On the other hand, the
  reverse formula~\eqref{equ:b-to-f} is only formal.
\end{remark}
\begin{remark}
  If $f_{\hbar} = 1$, then $T_{f_{\hbar}} = \Pi_{\Phi}$.  Hence
  Proposition \ref{prop_Toeplitz=pseudo_H(Phi)} implies that the
  Bergman projection $\Pi_{\Phi}$ can be written as a complex
  pseudo-differential operator. This remark was essential in the
  work~\cite{rouby-17}; our aim in this paper is to obtain a analogue
  of this remark for more general, non-quadratic $\Phi$.
\end{remark}

\subsection{Analytic symbols}

In view of extending the representation
formula~\eqref{equ:quadratic-projection} for the Bergman projection to
the case of a general phase function $\Phi$, we first need to discuss
the microlocal classes of analytic symbols, as introduced
in~\cite{sj-asterisque-82}, following~\cite{boutet-kree-67}.

\begin{defi}[Space $H^{\loc}_{ \Phi}$]
  \label{defi:h-loc}
  Let $ \Omega$ be an open subset of $\CM^n$. Let
  $ \Phi \in \mathscr{C}^0( \Omega; \mathbb{R})$. Let $u_{ \hbar}$ be
  a function defined on $ \Omega$. We say that $u_{ \hbar}$ belongs to
  the space $H_{ \Phi}^{\loc}( \Omega)$ if:
  \begin{enumerate}
  \item $ u_{ \hbar} \in \Hol( \Omega)$;
  \item $ \forall K \Subset \Omega$, $ \forall \epsilon > 0$,
    $ \exists C > 0$, such that
    $ |u_{ \hbar}(z)| \leq C e^{(\Phi(z) + \epsilon)/ \hbar}$ for all
    $z \in K$.
  \end{enumerate}
  If $u_{ \hbar} \in H_0^{\loc}( \Omega)$ (meaning that $ \Phi = 0$),
  we say that $u_{ \hbar}$ is an \textbf{analytic symbol}.
\end{defi}

Notice that analytic symbols may have a sub-exponential growth as
$\h\to 0$: for instance the constants $\h^{-m}$, for any $m\geq 0$,
are analytic symbols. In this work it will be important to control the
polynomial growth in $\h^{-1}$, and hence we introduce a finer
definition, as follows.
\begin{defi}
  \label{defi:symbol-finite-order}
  Let $m \in \mathbb{R}$. We say that $a_{ \hbar}\in\Hol(\Omega)$ is
  an \textbf{analytic symbol of finite order} $m$ if
  $a_{ \hbar} = \mathcal{O}(\hbar^{-m})$ locally uniformly in
  $ \Omega$, \textit{i.e.}  $ \forall K \Subset \Omega$,
  $ \exists C>0$ such that for all $z \in K$:
  \[
  |a_{ \hbar}(z)| \leq C \hbar^{-m}.
  \]
\end{defi}
Naturally, analytic symbols of finite order are also analytic symbols
in the sense of Definition~\ref{defi:h-loc}. Let $S^0(\Omega)$ be the
space of analytic symbols of order zero in $\Omega$.
\begin{defi}
  \label{defi:formal}
  A \textbf{formal classical analytic symbol} $\hat{a}_\h$ in $\Omega$
  is a formal series $\hat{a}_\h=\sum_{j=0}^\infty a_j \h^j$, where
  $a_j\in\Hol(\Omega)$ satisfies:
  \[
  \forall K \Subset \Omega, \exists C>0,\forall j\geq 0, \quad
  \sup_K\abs{a_j}\leq C^{j+1}j^j.
  \]
  We denote by $\hat S^0(\Omega)$ the space of such power series.
\end{defi}
The next definition is similar to the one used in~\cite[Definition
2.1]{boutet-kree-67}.
\begin{defi}
  \label{defi:cas}
  We say that $a_{ \hbar}\in S^0(\Omega)$ is a \textbf{classical
    analytic symbol} if there exists $\hat a_\h\in \hat S^0(\Omega)$
  such that $a_{ \hbar}$ admits the asymptotic expansion:
  $a_{ \hbar} \sim \hat{a}_{\hbar} = \sum_{j=0}^{ \infty} a_j
  \hbar^j$, in the following sense:
  \begin{equation}
    \label{eq_classical_symbol} \forall K \Subset \Omega, \exists C>0,
    \forall N\geq 0, \quad \quad \sup_K\abs{a_\h - \sum_{j=0}^{N-1}
      a_j \hbar^j}\leq \h^NC^{N+1}N^N.
  \end{equation}
\end{defi}

It is useful to introduce spaces of germs at a given point
$x_0\in \CM^n$; we let $H_{\Phi,x_0}$ be the space of germs of the
presheaf $H^{\loc}_{ \Phi}$ at $x_0$, \emph{i.e.}  $H_{\Phi,x_0}$ is
the inductive limit:
\[
H_{\Phi,x_0} := \lim_{\overrightarrow{\Omega\ni x_0}} H^{\loc}_{
  \Phi}(\Omega),
\]
where $\Omega$ varies in the set $\mathcal{V}(x_0)$ of open
neighbourhoods of $x_0$.  The local space $\widetilde H_{\Phi,x_0}$
consists of the germs $H_{\Phi,x_0}$ modulo exponentially small terms,
as follows.
\begin{defi}[Negligible germs and $\widetilde H_{\Phi,x_0}$]
  \label{defi:negligible-germ}
  An element $u_\h\in H_{ \Phi,x_0}$ will be called
  \textbf{negligible} if it belongs to the space
  \[
  \N := \{ u_\h \in H_{ \Phi, x_0} \quad ; \quad \exists c>0, \exists
  \Omega\in \mathcal{V}(x_0), \quad u_\h\in
  H_{\Phi-c}^{\loc}(\Omega)\}.
  \]
  The {$x_0$-localized space} is the quotient:
  \[
  \widetilde H_{\Phi,x_0} := H_{\Phi,x_0} / \N.
  \]
\end{defi}
Thus, two germs $u_\h$ and $v_\h$ at $x_0$ are equivalent if
$e^{-\Phi/\h}(u_\h-v_\h)$ is exponentially small near $x_0$ as
$\h\to 0$. We shall use the notation $u_\h \sim_a v_\h$ (where ``a''
stands for ``analytic'') to indicate that $u_\h = v_\h \mod \N$. Since
$S^0(\Omega)\subset H_{0}^{\loc}(\Omega)$, the space
$\widetilde H_{0,x_0}$ contains the subspace $(S^0_{x_0} \mod \N)$ of
symbols of order zero localized at $x_0$.

If $\Omega\in \mathcal{V}(x_0)$ and $\hat a_\h\in \hat S^0(\Omega)$,
then there exists a unique element $a_\h \in \widetilde H_{0,x_0}$
that admits, in some $B\in \mathcal{V}(x_0)$, the asymptotic expansion
given by $\hat a_\h$. Indeed, if the asymptotic expansion of a
classical analytic symbol $a_\h$ is zero, it follows
from~\eqref{eq_classical_symbol} (taking $N=1/e C\h$) that
$\sup_K\abs{a_\h}\leq Ce^{-1/e C\h}$; in particular,
$a_\h \in H_{-c}^{\loc}$ with $c = 1/eC$. Conversely, let $B$ be an
open ball centred at $x_0$ and such that $B\Subset\Omega$. Let $C$ be
the constant of Definition~\ref{defi:formal} with $K=\overline{B}$,
and let, for $z\in B$,
\[
a_\h(z) = \sum_{j=0}^{[1/ (e C \hbar)]} a_j(z) \hbar^j.
\]
Then, one can check that $a_\h\in S^0(B)$ and admits the asymptotic
expansion $\hat a_\h$ in the sense of~\eqref{eq_classical_symbol}.
With a slight abuse of notation, $a_\h$ will be called a classical
analytic symbol at $x_0$.

In order to discuss exponential decay, it is useful to introduce
variations of the weight function $\Phi$.
\begin{defi}\label{defi:H-negligible}
  A linear operator
  $R:H_{\Phi}^\loc(\Omega) \to H_{\Phi}^\loc(\Omega)$ will be called
  \textbf{negligible at $x_0$} if for any
  $\Omega_1\in\mathcal{V}(x_0)$ with $\Omega_1\Subset\Omega$, there
  exists $\Omega_2\in\mathcal{V}(x_0)$ with $\Omega_2\Subset\Omega$,
  and a continuous function $\Phi_2<\Phi$ on $\Omega_2$, such that
  \[
  R : H_{\Phi}(\Omega_1) \to H_{\Phi_2}(\Omega_2)
  \]
  is uniformly bounded as $\h\to 0$, where $H_{\Phi}(\Omega_1)$ and
  $H_{\Phi_2}(\Omega_2)$ are equipped with the corresponding
  $\lphi$-norm (Definition~\ref{defi:lphi}). When $R_1$ and $R_2$ are
  two operators such that $R_1-R_2$ is negligible, we will write
  $R_1\hnegl R_2$.
\end{defi}
In particular, if $R$ is negligible at $x_0$ and
$u_\h\in H_{\Phi,x_0}$, then $Ru_\h \sim_a 0$.

\paragraph{Notation. } In the rest of this text, when dealing with
germs, we will sometimes use the notation $\Vois(x_0,E)$, where
$x_0\in E$, to denote ``a sufficiently small neighbourhood of $x_0$ in
$E$''.

\subsection{Analytic pseudo-differential operators}
\label{sec:analyt-pseudo-diff}

Classical analytic symbols give rise to a well-behaved
pseudo-differential calculus, as shown in the
book~\cite{sj-asterisque-82}. We briefly recall here the necessary
definitions and properties, referring to~\cite{sj-asterisque-82}
and~\cite{h-sj-minicours} for details.

Let $x_0\in\CM^n$, let $\Phi$ be a $\mathscr{C}^2$ real-valued
function defined in a small neighbourhood $\Omega$ of $x_0$.  For
$x\in\Omega$, $r>0$ sufficiently small and $R>0$, we define the
contour in $\CM^{2n}$:
\begin{equation} \label{equ:Gamma(x)} \Gamma(x) :=
  \left\{(y,\theta)\in \Omega\times \CM^{n}; \quad \theta =
    \dfrac{2}{i} \dfrac{\partial \Phi}{\partial x}(x) + i R
    \overline{(x-y)}; \quad |x-y| \leq r \right\}.
\end{equation}
By Taylor's formula
$\Phi(y) = \Phi(x) + 2\Re ((y-x)\cdot \partial_x\Phi(x) +
\O(\abs{x-y}^2))$,
we obtain the following estimate when $\abs{x-y}$ is small enough:
\begin{equation}
  e^{- \Phi(x)/ \hbar} \left| e^{i(x-y)\cdot \theta/ \hbar} \right| e^{
    \Phi(y)/ \hbar} \leq e^{-(1/ \hbar)(R - C) |x-y|^2},
  \label{equ:kernel-pseudo}
\end{equation}
where $C$ is controlled by the $\mathscr{C}^2$-norm of $\Phi$ near
$x_0$.  Let $R>C$, and let $a_{ \hbar}(x, y, \theta)$ be an analytic
symbol defined in a neighbourhood of
$ \left( x_0, x_0, \theta_0 \right)\in\CM^{3n}$, with
$\theta_0:=\frac{2}{i} \frac{\partial \Phi}{\partial x} (x_0)$ (in
other terms, $a_\h\in H_{0,(x_0, x_0, \theta_0)}$); let us consider,
for $u_\h\in H_{\Phi,x_0}$, the contour integral:
\begin{equation}
  A_\Gamma u(x) = \dfrac{1}{( 2 \pi \hbar)^n} \intint_{
    \Gamma(x)} \!\!  e^{(i/ \hbar)(x-y)\cdot \theta} a_{ \hbar}(x, y,
  \theta) u(y) \DD y \DD \theta.
  \label{equ:pseudo-diff-complexe}
\end{equation}
Using a deformation variant of Stokes' formula (see for
instance~\cite[Lemma 12.2]{sj-asterisque-82}), one can show that
$\dbar(A_\Gamma u)$ is $\mathcal{O}(e^{-c/\h})$, for some $c>0$,
uniformly near $x_0$. Hence, by solving a $\dbar$ problem, one can
find a holomorphic function $v$ near $x_0$ such that
$A_\Gamma u = v + \mathcal{O}(e^{-c/\h})$. Such a $v$ is unique modulo
$\mathcal{N}$. Hence we will slightly abuse notation and write
$v=A_\Gamma u$.

Note that the size $r>0$ depends on the domain of definition of
$u_\h$. However, if $u_\h\in H_{\Phi,x_0}$, the choice of $r$ and $R$
only modifies $A_\Gamma u(x) $ by a negligible term in $\mathcal{N}$.
Moreover, if $u_\h = v_\h$ in $\widetilde H_{\Phi, x_0}$, then
$A_\Gamma u_\h = A_\Gamma v_\h$ in $\widetilde H_{\Phi, x_0}$. This gives the following statement:
\begin{prop}[{\cite[section 4]{sj-asterisque-82}, \cite[section 2.5]{h-sj-minicours}}]
  Formula~\eqref{equ:pseudo-diff-complexe} defines an operator
  $A:\widetilde H_{\Phi, x_0} \to \widetilde H_{\Phi, x_0}$. It is
  called a complex pseudo-differential operator.
\end{prop}
If $a_\h=1$, then $Au_\h = u_\h$ in $\widetilde H_{\Phi,x_0}$, which
can be viewed as a version of the Fourier inversion formula.

The symbol of $A$ is the function $\sigma_A$ defined for
$(x, \theta) \in \Vois \left( (x_0, \theta_0); \CM^{2n}\right)$ by the
formula:
\[
\sigma_A(x, \theta) = e^{-i x\cdot \theta/ \hbar} A \left( e^{ i (
    \cdot)\cdot \theta/ \hbar} \right) .
\]
Then $\sigma_A\in H_{0,(x_0,\theta_0)}$. If $a_\h$ does not depend on
the variable $y$, then $\sigma_A(x,\theta) \sim_a a_\h$ in
$H_{0,(x_0,\theta_0)}$. If $\Phi\in C^\infty$ and $a_\h$ is a
classical analytic symbol of order zero, then by the stationary phase
lemma and the use of good contours~\cite[Lemme 3.2]{sj-asterisque-82},
\cite[Theorem 2.3.3, Lemma 2.4.2]{h-sj-minicours}, $\sigma_A$ is also
a classical analytic symbol of order zero. Moreover, if the formal
series associated with $a_\h$ by~\eqref{eq_classical_symbol} is zero
(which means that all $a_j$'s are identically zero) then the formal
series associated with $\sigma_A$ is also zero, in the same sense. We
will see in Section~\ref{sec:from-hbar-pseudo} that the converse
statement holds as well.

An important particular class of analytic pseudo-differential
operators concern the case where the symbol $a_\h$ has the form
$a_\h(x,y,\theta)=b^{\w}_\h(\frac{x+y}{2},\theta)$, for a classical
analytic symbol
$b^{\w}_\h\in S^0(\Vois\left( x_0, \theta_0 := \frac{2}{i}
  \frac{\partial \Phi}{\partial x} (x_0) \right))$.
As in Proposition~\ref{prop_transfor_Bargmann_et_quantif_weyl}, we
obtain the so-called complex Weyl quantization, namely:
\begin{equation}
  \label{equ:complex-weyl} 
  \Op^{\w}_\h(b^{\w}_{ \hbar}) u(x) = 
  \dfrac{1}{(2 \pi \hbar)^n} \intint_{\Gamma(x)} e^{\frac{i}{\h}(x-y) \cdot \theta} 
  b^{\w}_{ \hbar} \left( \dfrac{x+y}{2}, \theta \right) u(y) \DD y \DD \theta .
\end{equation}

\subsection{Brg-quantization}
\label{ssec:brg-quantization}

We explain here how formula~\eqref{equ:quadratic-projection}, extended
to non-quadratic phase functions, naturally leads to a general local
quantization scheme. For achieving this, we describe our new class of
operators via a general ansatz for their Schwartz kernel, which was
first introduced in the setting of smooth symbols on compact Kähler
manifolds by Charles~\cite{charles-toeplitz}. This ansatz was also
crucial for the generalization of Berezin-Toeplitz operators on
symplectic manifolds~\cite{Charles2016}. It improves on initial ideas
by Berezin, Lieb and Simon, who used particular forms of it in order
to obtain operator bounds, see~\cite[Definitions (2.2), (2.4)]{Simon80}.

Let $x_0\in\CM^n$, and let $\Phi$ be a real-analytic function defined
in a neighbourhood of $x_0$.  We view $\CM^n$ as a totally real
subspace of $\CM^{2n}$ via the embedding in the anti-diagonal
$ \Lambda = \lbrace (x, \overline{x}); x \in \CM^n \rbrace$.  The map
$(x,\bar x) \mapsto \Phi(x)$ admits a holomorphic extension to a
neighbourhood of $(x_0,\bar x_0 )$ in $\CM^{2n}$. We denote this
extension by $\psi(x,w)$; thus $\psi(x,\bar x)=\Phi(x)$, and we have
\begin{equation}
  \psi(x, w) = \overline{\psi(\bar w, \bar{x})}\,,
  \label{equ:bar-psi}
\end{equation}
provided of course that the neighbourhood is invariant under the
involution $(x,w)\mapsto(\bar w, \bar x)$.  This follows from the fact
that the identity holds on the maximally totally real subspace
$\Lambda$, on which $\psi$ takes real values. Finally, let us assume
that $\Phi$ is strictly plurisubharmonic: there exists $m>0$ such that
\begin{equation}
  \qquad m\textup{Id} \leq ( \partial^2_{x_i,
    \bar{x}_j} \Phi(x_0) )_{i, j=1}^n.
  \label{equ:spsh}
\end{equation}

\begin{lemm}\label{lemm:phase-brg}
  For $x,y$ near $x_0$ we have
  \begin{equation}
    \Phi(x) + \Phi(y)  - 2 \Re \left( \psi(x, \bar{y}) \right) 
    \asymp |x-y|^2 ,
    \label{equ:estim-phi1}
  \end{equation} 
\end{lemm}
\begin{proof}
  This follows from a Taylor expansion of $\psi$ around
  $(x_0, \bar x_0)$.
\end{proof}

\begin{defi}
  \label{defi:brg-quantization}
  Let $\tilde r > r>0$. Let
  $a_\h\in L^\infty(B((x_0,\bar x_0),\tilde r))$. We define the
  operator $\Opbrg_r(a_\h)$ locally near $x_0$ by the following
  integral representation. For $x\in\CM^n$ with
  $\abs{x-x_0}<\tilde r - r$, and $u\in L^1(B(x_0,\tilde r))$,
  \begin{equation}
    [\Opbrg_{r}(a_\h) u] (x) = \int_{B(x,r)} k_{ \hbar}(x, y) u(y) L(\DD y),
    \label{equ:kernel-brg}
  \end{equation}
  where the kernel $k_{ \hbar}$ is defined as follows, for $(x, y)$
  such that $|x-y| < r$:
  \[
  k_{ \h}(x, y) = \frac{2^n}{(\pi\h)^n} e^{\frac{2}{\h}( \psi(x,
    \overline{y}) - \Phi(y))} a_{ \hbar}(x,\overline{y}) \det
  \left(\partial_{\tilde w}\partial_x \psi \right)(x, \overline{y})
  \]
\end{defi}
Note that
$(x,y)\mapsto \det \left(\partial_{\tilde w}\partial_x \psi \right)
(x, y)$
is the holomorphic extension to a neighbourhood of $(x_0,\bar x_0)$ of
the real-analytic map
$(x,\bar x)\mapsto \det ( \partial^2_{x_i, \bar{x}_j} \Phi(x) )_{i,
  j=1}^n$.
Under the assumptions of Lemma~\ref{lemm:phase-brg}, and choosing a
smaller $r$ if necessary, there exists $\epsilon>0$ such that
\[
- \Phi(x) + 2 \Re \left( \psi(x, \overline{y}) \right) - \Phi(y) \leq
- \left(m - \epsilon \right) |x-y|^2 ,
\]
and hence
\begin{equation}
  e^{- \Phi(x)/ \hbar} \left| k_{ \h}(x, y) \right| e^{ \Phi(y)/
    \hbar}\leq \frac{2^n}{(\pi\h)^n} \abs{a_\h(x,\bar y)} e^{-(1/
    \hbar)(m - \epsilon) |x-y|^2},
  \label{equ:brg-bon-contour}
\end{equation}
which is similar to~\eqref{equ:kernel-pseudo}. By the same arguments
as the ones used there, we see that $\Opbrg(a_\h)= \Opbrg_r(a_\h)$
defines an operator on $\widetilde H_{\Phi, x_0}$, which does not
depend on $r$ small enough.

This `$\Brg$-quantization' is a natural generalization of
Formula~\eqref{equ:quadratic-projection} when the weight is quadratic:
in this special case, we get formally $\Pi_\Phi = \Opbrg_\infty(1)$,
and Berezin-Toeplitz operators can be obtained when $a_\h$ only
depends on $y$.

\subsection{Analytic Fourier integral operators}
\label{sec:fio}

In this section, we recall the definition of semiclassical Fourier
integral operators in the complex domain, and prove that, under a
transversality condition, they act on spaces of germs holomorphic
functions $H_{\Phi,y_0}\to H_{\widetilde \Phi, x_0}$, modulo
exponentially small remainders.

We want to give a meaning to the formal expression
\begin{equation}
  \label{equ:formal-fio}
  Au(x) = \intint e^{\frac{i}{\h}\phy(x,y,\theta)}a(x,y,\theta;\h) u(y)
  \DD y \DD \theta,
\end{equation}
where $a$ is an analytic symbol defined near $(x_0,y_0,\theta_0)$, and
$\phy$ is a non-degenerate holomorphic phase function, as follows. Let
$\phy(x,y,\theta)$ be holomorphic in a neighbourhood of
$(x_0,y_0,\theta_0)\in \CM^m\times \CM^n\times \CM^N$. Assume that
$\phy'_\theta(x_0,y_0,\theta_0)=0$. Recall that $\phy$ is a
\emph{non-degenerate phase function} (in the sense of
Hörmander~\cite{FIO1}) if the map $\phy'_\theta$ is a local
submersion, \emph{i.e.}
\begin{equation}
  \label{equ:ND-hormander}
  \DD \partial_{\theta_1}\phy(x_0,y_0,\theta_0),\dots,\DD \partial_{\theta_N}\phy(x_0,y_0,\theta_0)
  \text{  are linearly independent.}
\end{equation}
Then
\[
C_\phy := \{ (x,y,\theta)\in \Vois((x_0,y_0,\theta_0), \CM^{m+n+N});
\quad \phy'_\theta(x,y,\theta)=0)\}
\]
is a complex manifold of codimension $N$. Moreover, the map
\begin{equation}
  C_\phy \ni (x,y,\theta)\mapsto (x,\partial_x\phy(x,y,\theta); y,
  -\partial_y\phy(x,y,\theta)) \in T^*\CM^m \times T^*\CM^n
  \label{equ:Cphy}
\end{equation}
has injective differential and hence the image $\Lambda'_\phy$ is a
complex manifold (defined near $(x_0,\xi_0; y_0,\eta_0)$, with
$\xi_0:=\partial_x\phy(x_0,y_0,\theta_0)$ and
$\eta_0:=-\partial_y\phy(x_0,y_0,\theta_0)$) of dimension $m+n$; thus,
in view of~\eqref{equ:Cphy}, $\Lambda'_\phy$ is a holomorphic
canonical relation. This general setting for Fourier integral
operators is well known, since~\cite{FIO1} (see
also~\cite{duistermaat-oif}), at least in the $\Cinf$ setting. The
adaptation to analytic semiclassical analysis was done
in~\cite{sj-asterisque-82}, and necessitates substantial
modifications. It is not our goal here to recall the general theory,
but instead we want to point out a application of transversality that
we could not find elsewhere in the literature, as follows.

In order to have a well defined operator $A$ acting on holomorphic
functions, we don't require the relation $\Lambda'_\phy$ to be a
diffeomorphism, but we strengthen the
assumption~\eqref{equ:ND-hormander} to
\begin{equation}
  \label{equ:ND-strong}
  (y,\theta)\mapsto \phy(x_0,y,\theta) \text{ is a non-degenerate phase function near } (y_0,\theta_0).
\end{equation}
Equivalently, the map
\[
\Lambda'_\phy \ni (x,\xi;y,\eta) \mapsto x
\]
is a local submersion (which implies that $\Lambda'_\phy\cap\{x=x_0\}$
is a complex manifold of dimension $n$) and the map
\begin{equation}
  \label{equ:immersion}
  \Lambda'_\phy\cap \{x=x_0\} \ni (x_0,\xi;y,\eta) \mapsto (y,\eta)
\end{equation}
is a local immersion. The image of~\eqref{equ:immersion}, namely
$(\Lambda'_\phy)^{-1}(T^*_{x_0}\CM^n)$, is a complex Lagrangian
manifold in $T^*\CM^n$.

\begin{prop}
  \label{prop:fio}
  Let $\Phi$ be a pluriharmonic function defined near
  $y_0\in\CM^n$. Let
  $\Lambda_\Phi:=\{(y,\frac{2}{i}\partial_y\Phi(y));
  y\in\Vois(y_0,\CM^n)\}$,
  and $\eta_0:=\frac{2}{i}\partial_y \Phi(y_0)$. Assume that $\phy$
  satisfies~\eqref{equ:ND-strong}, so that
  $(\Lambda'_\phy)^{-1}(T^*_{x_0}\CM^m)$ and $\Lambda_\Phi$ both are
  complex Lagrangian manifolds passing through $(y_0,\eta_0)$. Assume
  \begin{equation}
    \label{equ:transv-inter}
    (\Lambda'_\phy)^{-1}(T^*_{x_0}\CM^m) \text{ and } \Lambda_\Phi
    \text{ intersect transversally at } (y_0,\eta_0).
  \end{equation}
  Then $A$ is a well-defined operator
  $\widetilde H_{\Phi,y_0}\to \widetilde H_{\widetilde\Phi, x_0}$,
  where $\widetilde\Phi$ is a pluriharmonic function defined near
  $x_0$ with the property
  \begin{equation}
    \label{equ:fio-lagrangian}
    \Lambda_{\widetilde\Phi} = \Lambda'_\phy(\Lambda_\Phi)\,.
  \end{equation}
\end{prop}
\begin{proof}
  For $x$ close to $x_0$, $ (\Lambda'_\phy)^{-1}(T^*_{x}\CM^m)$ and
  $\Lambda_\Phi$ intersect transversally at a unique point
  $(y(x),\eta(x))$, and because of~\eqref{equ:immersion}, there is a
  corresponding unique point
  $(x,\xi(x); y(x),\eta(x))\in \Lambda'_\phy$. Here $\xi(x)$, $y(x)$,
  $\eta(x)$ are holomorphic functions of $x$. Thus
  $\Lambda'_\phy(\Lambda_\Phi)$ is a complex manifold of dimension
  $m$, given by
  \[
  \Lambda'_\phy(\Lambda_\Phi) = \{(x,\xi(x)); \quad
  x\in\Vois(x_0;\CM^m)\}.
  \]
  The assumptions~\eqref{equ:ND-strong} and~\eqref{equ:immersion}
  imply that the pluri-harmonic function
  \begin{equation}
    \label{equ:fio-phase}
    (y,\theta) \mapsto -\Im \phy(x,y,\theta) + \Phi(y)
  \end{equation}
  has a unique non-degenerate critical point $(y(x), \theta(x))$,
  necessarily of signature $(n+N,n+N)$, near $(y_0,\theta_0)$,
  depending holomorphically on $x$ near $x_0$. Here $y(x)$ is the same
  as before and if we denote by $\widetilde \Phi(x)$ the corresponding
  critical value:
  \[
  \widetilde\Phi(x) := \textup{vc}_{(y,\theta)} ( -\Im
  \phy(x,y,\theta) + \Phi(y)),
  \]
  then we see that $\frac{2}{i}\partial_x\widetilde\Phi(x) = \xi(x) $
  is the point defined above. Thus~\eqref{equ:fio-lagrangian} holds.

  Next consider formally $Au$ in~\eqref{equ:formal-fio} for
  $u\in H_{\Phi,y_0}$. Then
  \[
  \abs{e^{\frac{i}{\h}\phy(x,y,\theta)}a(x,y,\theta;\h) u(x)} \leq
  C_\epsilon e^{\frac{1}{\h}(\epsilon-\Im\phy(x,y,\theta) + \Phi(y))},
  \quad \forall \epsilon>0
  \]
  and since~\eqref{equ:fio-phase} has a non-degenerate critical point
  $(y(x),\theta(x))$, of signature $(n+N,n+N)$, we know that we can
  find a good contour $\Gamma(x)$, \emph{i.e.} a real submanifold of
  dimension $n+N$, passing through $(y(x),\theta(x))$ along which
  \[
  -\Im \phy(x,y,\theta) + \Phi(y) - \widetilde\Phi(x) \asymp
  -\abs{y-y_0}^2 - \abs{\theta-\theta_0}^2\,.
  \]
  It then suffices to define
  \begin{equation}
    \label{equ:fio-contour}
    Au(x) = \intint_{\Gamma(x)}
    e^{\frac{i}{\h}\phy(x,y,\theta)}a(x,y,\theta;\h) u(y)
    \DD y \DD \theta\,,
  \end{equation}
  and argue as we did for analytic pseudo-differential operators
  (Section~\ref{sec:analyt-pseudo-diff}) to obtain that
  $A:\widetilde H_{\Phi,y_0}\to \widetilde H_{\widetilde\Phi, x_0}$ is
  well-defined.
\end{proof}

\begin{remark}
  The more general case where $\Phi$ is plurisubharmonic can certainly
  be treated with some additional arguments.
\end{remark}

As an application of this proposition, we can compose Fourier integral
operators. If
$A:\widetilde H_{\Phi,y_0}\to \widetilde H_{\widetilde\Phi, x_0}$ is
as in Proposition~\ref{prop:fio}, let $K_A$ be the corresponding
canonical relation, so far denoted $\Lambda'_\phy$. Let
$z_0\in\CM^\ell$ and let
$B:\widetilde H_{\widetilde \Phi,x_0}\to \widetilde H_{\hat\Phi, z_0}$
be a Fourier integral operator which satisfies the same assumption as
$A$ with $\widetilde\Phi, \Phi$ replaced by
$\hat\Phi, \widetilde\Phi$. Let $K_B$ be the canonical relation of
$B$, and let $(z_0,\zeta_0;x_0,\xi_0)\in K_B$. By
Proposition~\ref{prop:fio}, the composition
$B\circ A:\widetilde H_{\Phi,y_0}\to \widetilde H_{\hat\Phi, z_0}$ is
well-defined. The condition~\eqref{equ:transv-inter} for $B$ says that
\begin{equation}
  \label{equ:KA-KB}
  K_B^{-1}(T^*_{z_0}\CM^\ell) \text{ and } \Lambda_{\widetilde\Phi}
  \text{ intersect transversally at } (x_0,\xi_0),
\end{equation}
and in view of~\eqref{equ:fio-lagrangian} we have
$\Lambda_{\widetilde\Phi} =
K_A(\Lambda_\Phi)$. Hence~\eqref{equ:KA-KB} is equivalent to
\begin{gather}
  \label{equ:KA-KB2}
  \left(K_B\cap(T_{z_0}^*\CM^\ell \times T^* \CM^m)\right) \times
  \left(K_A\cap(T^*\CM^m\times \Lambda_{\Phi})\right) \text{ and }\\
  T_{z_0}^*\CM^\ell \times \textup{diag}(T^* \CM^m \times T^* \CM^m)
  \times \Lambda_\Phi\\
  \text{ intersect transversally in } \quad T^*\CM^\ell \times T^*
  \CM^m \times T^* \CM^m \times \Lambda_\Phi\,.
\end{gather}
This implies the classical transversality condition for the
composition $B\circ A$:
$T^*\CM^\ell \times \textup{diag}(T^* \CM^m \times T^* \CM^m) \times
T^*\CM^n$
and $K_B\times K_A$ intersect transversally in
$T^*\CM^\ell \times T^* \CM^m \times T^* \CM^m \times T^*\CM^n$.
Therefore, if in addition to~\eqref{equ:formal-fio} we write
\begin{equation}
  Bv(z) = \intint e^{\frac{i}{\h}\psi(z,x,\omega)}b(z,x,\omega;\h) v(x)
  \DD x \DD \omega\,,
\end{equation}
where $\psi$ is a non-degenerate phase function defined near
$(z_0,x_0,\omega_0)$, then we know that $B\circ A$ is an analytic
Fourier integral operator for which
$\psi(z,x,\omega)+\phy(x,y,\theta)$ is a non-degenerate phase function
with $z,y$ as base variables and $x,\omega,\theta$ as fibre variables
and that the canonical relation $K_{B\circ A}$ is equal to
$K_B\circ K_A$.

\section{Equivalence of quantizations}
\label{sec:proof}
One of the main results of this work is to show that, in the
semiclassical limit, operators of the form $ \Op^{\Brg}(a_{ \hbar})$
with an analytic weight $\Phi$ can in fact be written, up to
exponentially small terms, as analytic pseudo-differential operators.

\begin{theo} \label{theom} Let
  $ \Phi: \Vois(x_0; \CM^n) \to \mathbb{R}$ be a real-analytic and
  strictly plurisubharmonic function.
  \begin{enumerate}
  \item \label{item:1} Let $a_{ \hbar}(x, w)$ be a classical analytic
    symbol of order zero defined in a neighbourhood of
    $(x_0, \bar{x}_0)$. Then there exists a classical analytic symbol
    $b^{\w}_{ \hbar}(x, \theta)$ of order zero defined in a
    neighbourhood of
    $ \left( x_0, \theta_0:= \frac{2}{i} \frac{\partial \Phi}{\partial
        x}(x_0) \right)$ such that
    \[
    \Op^{\Brg}(a_{ \hbar}) u(x) \hnegl \Op^{\w}_\h(b^{\w}_{ \hbar})
    \quad : H_{ \Phi, x_0} \to H_{ \Phi, x_0}\,.
    \]
  \item \label{item:2} Let $b^{\w}_{ \hbar}(x, \theta)$ be a classical
    analytic symbol of order zero defined in a neighbourhood of
    $ \left( x_0, \theta := \frac{2}{i} \frac{\partial \Phi}{\partial
        x}(x_0) \right)$.
    Then there exists a classical analytic symbol $a_{ \hbar}(x, w)$
    of order zero defined in a neighbourhood of $(x_0, \bar{x}_0)$
    such that
    \[
    \Op^{\w}_\h(b^{\w}_{ \hbar}) \hnegl \Op^{\Brg}(a_{ \hbar}) \quad :
    H_{ \Phi, x_0} \to H_{ \Phi, x_0}\,.
    \]
  \item In case (\ref{item:1}) (resp. case (\ref{item:2})), the formal
    symbol associated with $b^{\w}_{ \hbar}(x, \theta)$ (resp.
    $a_{ \hbar}(x, w)$) is uniquely determined by the formal symbol
    associated with $a_{ \hbar}(x, w)$ (resp.
    $b^{\w}_{ \hbar}(x, \theta)$).
  \end{enumerate}
\end{theo}

The proof of the first assertion of Theorem \ref{theom} is divided
into two parts (Sections~\ref{sec:from-brg-operators} and
\ref{sec:from-hbar-pseudo} below): first, we relate a $\Brg$-operator
to a complex pseudo-differential operator in the sense of Equation
\eqref{equ:pseudo-diff-complexe} and then we relate this last operator
to a complex Weyl pseudo-differential operator.

The second and third assertions of Theorem \ref{theom} are obtained by
showing that the operator $a_\h\mapsto b_\h$ in the first assertion is
in fact an elliptic Fourier Integral Operator and hence can be
microlocally inverted in the analytic category; see
Sections~\ref{sec:FIO} and~\ref{sec:transco}.

\subsection{From Brg-operators to complex
  \texorpdfstring{$\h$}{h}-pseudo-differential operators}
\label{sec:from-brg-operators}

Let $\tilde a_\h(x,w)=a_\h(x,\bar w, w)$, where $a_{ \hbar}(x, y, w)$
is a classical analytic symbol of order zero defined on a
neighbourhood of $ (x_0, x_0, \bar{x}_0)$.  Recall
from~\eqref{equ:kernel-brg} and~\eqref{equ:lebesgue} that we have the
formula, for $u \in H_{ \Phi, x_0}$:
\begin{gather}
  \Op^{\Brg}(\tilde a_{ \hbar}) u(x) = \\
  \dfrac{1}{( 2 \pi \hbar)^{n}} \int_{\Vois(x_0)}
  e^{\frac{2}{\h}\psi(x, \bar{y})} a_{ \hbar}(x, y, \bar{y}) u(y)
  e^{-\frac{2}{\h}\Phi(y)} J(x, \bar{y}) \; (\DD y \wedge \DD
  \bar{y}),
\end{gather}
where
$J(x, \bar{y}) = \det \left(\frac{2}{i}\partial_{\tilde w} \partial_x
  \psi \right)(x, \bar{y})$
(see Section~\ref{ssec:brg-quantization}).  We can rewrite this
formula as follows, for $u \in H_{ \Phi, x_0}$:
\begin{equation}
  \Op^{\Brg}(\tilde a_{ \hbar}) u(x)
  = \dfrac{1}{( 2 \pi \hbar)^{n}} \intint_{\tilde{\Gamma}(x_0)} 
  e^{\frac{2}{\h}(\psi(x, w) - \psi(y, w))} a_{ \hbar}(x, y, w) u(y) J(x, w) \DD y \DD w,
\end{equation}
where $ \tilde{\Gamma}(x_0)\subset \CM^{2n}$ is the integration
contour $\{(y,w) = (y,\bar{y})\}$ for $y$ near $x_0$, and using the
fact that $ \Phi(y) = \psi(y, \bar{y})$.  We perform Kuranishi's trick
and write for
$(x, y, w) \in \Vois(x_0; \CM^{n}) \times \tilde{\Gamma}(x_0)$:
\[
2 \left( \psi(x, w) - \psi(y, w) \right) = i (x-y)\cdot \theta(x, y,
w) ,
\]
where $ \theta$ is holomorphic on
$ \Vois(x_0) \times \tilde{\Gamma}(x_0)$ and satisfies the following
equality, for $(x, y, w) \in \Vois(x_0) \times \tilde{\Gamma}(x_0)$:
\begin{equation} \label{eq_Kuranishi_trick} \theta(x, y, w) =
  \dfrac{2}{i} \partial_x \psi(x, w) + \mathcal{O}( |x-y|).
\end{equation}
Although we don't use this here, it is often important to see that,
writing $ \theta$ as
\[
\theta(x, y, w) = \int_0^1 \dfrac{2}{i} \partial_x \psi((1-t)y+ tx, w)
\DD t ,
\]
then \eqref{eq_Kuranishi_trick} improves into:
\[
\theta(x, y, w) = \dfrac{2}{i} \partial_x \psi \left( \dfrac{x+y}{2},
  w \right) + \mathcal{O}( |x-y|^2).
\]
Therefore, we can rewrite the operator
$ \Op^{\Brg}(\tilde a_{ \hbar})$ as follows, for
$u \in H_{ \Phi, x_0}$:
\[
\Op^{\Brg}(\tilde a_{ \hbar}) u(x) = \dfrac{1}{( 2 \pi \hbar)^{n}}
\intint_{\tilde{\Gamma}(x_0)} e^{\frac{i}{\h} (x-y)\cdot\theta(x, y,
  w)} a_{ \hbar}(x, y, w) u(y) J(x, w) \DD y \DD w.
\]
We deduce from \eqref{eq_Kuranishi_trick} that for
$(x, y, w) \in \Vois(x_0) \times \tilde{\Gamma}(x_0)$:
\begin{equation}
  \label{equ:kuranishi-derive}
  \partial_w \theta(x, y, w) =
  \dfrac{2}{i} \partial_w \partial_x \psi(x, w) + \mathcal{O}(|x-y|),
\end{equation}
whose determinant is non-vanishing because $ \Phi$ is strictly
plurisubharmonic. Thus, according to the holomorphic implicit function
theorem, the function $w \mapsto \theta(x, y, w)$ admits a holomorphic
inverse in $ \Vois(\bar{x}_0)$. We denote this inverse for
$(x, y, \theta) \in \Vois(x_0, x_0, \theta_0)$ by:
\begin{equation} \label{eqref:w(x,y,theta)} w = w(x, y, \theta) ,
\end{equation}
where
\[
\theta_0 := \tfrac{2}{i} \partial_x \psi(x_0, \bar{x}_0) =
\dfrac{2}{i} \dfrac{\partial \Phi}{\partial x}(x_0)\,.
\]
We want to rewrite the operator $ \Op^{\Brg}(\tilde a_{ \hbar})$ in
terms of the $ \theta$-variable. We have, as holomorphic $2n$-forms,
\begin{align}
  \DD y \wedge \DD \theta 
  & = \det  \left( \partial_w \theta(x, y, w) \right) \DD y \wedge \DD w, \\
  & = \det \left( \dfrac{2}{i} \partial_w \partial_x \psi(x, w) + \mathcal{O}(|x-y|) \right) \DD y \wedge \DD w .
\end{align}
Here, as always in this paper, we use the notation
\[
\DD y \wedge \DD \theta := \bigwedge_{j=1}^n (\DD y_j \wedge \DD
\theta_j)\,, \quad \DD y \wedge \DD w := \bigwedge_{j=1}^n (\DD y_j
\wedge \DD w_j)\,.
\]

Let $\tilde{J}(x, y, \theta)$ be the following quantity, for
$(x, y, \theta) \in \Vois(x_0, x_0, \theta_0)$:
\begin{equation} \label{defi_J} \tilde{J}(x, y, \theta) := \dfrac{J(x,
    w(x, y, \theta))}{\det \left( \partial_w \theta(x, y, w) \right) }
  = (1+\mathcal{O}(x-y))\,,
\end{equation}
so that:
\[
J (x, w) \DD y \wedge \DD w = \tilde{J}(x, y, \theta) \DD y \wedge \DD
\theta \,.
\]
Using~\eqref{eq_Kuranishi_trick} and \eqref{equ:kuranishi-derive}, the
image of $\tilde \Gamma(x_0)$ under $w\mapsto \theta=\theta(x,y,w)$
can be deformed into $\Gamma(x)$ (see~\eqref{equ:Gamma(x)}) in such a
way that 
\[
-\Phi(x) - \Phi(y) + \Re \left( i(x-y)\cdot\theta \right) \asymp
-|x-y|^2 ,
\]
uniformly on all the deformed contours.  Hence we can rewrite the
operator $ \Op^{\Brg}(\tilde a_{ \hbar})$ as follows, for
$u \in H_{ \Phi, x_0}$:
\[
\Op^{\Brg}(a_{ \hbar}) u(x) \sim_a \dfrac{1}{( 2 \pi \hbar)^{n}}
\intint_{\Gamma(x)} \! e^{\frac{i}{\h} (x-y)\cdot\theta} a_{ \hbar}(x,
y, w(x, y, \theta)) u(y) \tilde{J}(x, y, \theta) \DD y \DD \theta,
\]
which is a complex pseudo-differential operator (in the sense of
Equation \eqref{equ:pseudo-diff-complexe}) with symbol, for
$(x, y, \theta) \in \Vois(x_0, x_0, \theta_0= \frac{2}{i} \partial_x
\psi(x_0, \bar{x}_0))$:
\begin{equation} \label{equation_symbol_b} b_{ \hbar}(x, y, \theta) =
  a_{ \hbar}(x, y, w(x,y, \theta)) \tilde{J}(x, y, \theta) \,.
\end{equation}
Let:
\begin{align}
  \label{equ:defi-W}
  W : \Vois(x_0, x_0, \theta_0) 
  & \longrightarrow \Vois \left( x_0, x_0,  \bar{x}_0 \right) \\
  (x, y, \theta) & \longmapsto (x, y, w(x,y,\theta)).
\end{align}
Then, for $(x, y, \theta) \in \Vois \left( x_0, x_0, \theta_0\right)$:
\[
b_{ \hbar}(x, y, \theta) = \tilde{J}(x, y, \theta) a_{ \hbar}(x, y,
w(x, y, \theta)) = \tilde{J}(x, y, \theta) (W^* a_{ \hbar})(x, y,
\theta).
\]
Here $W^*$ denotes the pull-back by $W$, \emph{i.e.}
$W^* a_\h = a_\h \circ W$.  To conclude, we have
\begin{align}
  \Op^{\Brg}(\tilde a_{ \hbar}) u(x) \hnegl B_{\Gamma} u(x) \quad \text{ in } H_{ \Phi, x_0}\,,
\end{align}
where $B_{\Gamma}$ means the quantization (in the sense of Equation
\eqref{equ:pseudo-diff-complexe}) of the classical analytic symbol
$b_{ \hbar} $ defined for
$(x, y, \theta) \in \Vois \left( x_0, x_0, \theta_0 =\frac{2}{i}
  \frac{\partial \Phi}{\partial x}(x_0) \right)$ by:
\[
b_{ \hbar}(x, y, \theta) = \left( \tilde{J} W^* a_{ \hbar} \right) (x,
y, \theta).
\]

\subsection{From complex \texorpdfstring{$\h$}{h}-pseudo-differential
  operators to complex Weyl pseudo-differential operators}
\label{sec:from-hbar-pseudo}

Our goal is now to replace the symbol $b_{ \hbar}(x, y, \theta)$
defined in a neighbourhood of
$\left( x_0, x_0, \theta_0 := \frac{2}{i} \frac{\partial
    \Phi}{\partial x}(x_0) \right)$
by a symbol of the form
$b^{\w}_{ \hbar} \left( \tfrac{x+y}{2}, \theta \right)$ defined in a
neighbourhood of $\left( x_0, \theta_0\right)$ in order to obtain the
complex Weyl quantization.

We first recall how to relate the various quantizations of a symbol
depending on $(y,\theta)$. Let $a_{\hbar, t}(y, \theta)$ be a
classical analytic symbol defined in a neighbourhood of
$\left(x_0, \theta_0 \right)$.  For $t\in|0,1]$, the quantization
$\Op_t$ is defined, for $u \in H_{\Phi, x_0}$, by
\[
\Op_t(a_{\hbar, t}) u(x) = \dfrac{1}{(2 \pi \hbar)^n}
\intint_{\Gamma(x)} e^{\frac{i}{\h}(x-y) \cdot \theta} a_{ \hbar,
  t}(tx+(1-t)y, \theta) u(y) \DD y \DD \theta,
\]
where $ \Gamma(x)$ is defined in Equation \eqref{equ:Gamma(x)}.  When
$t = \frac{1}{2}$, we recover the complex Weyl quantization
(Proposition~\ref{prop_transfor_Bargmann_et_quantif_weyl}).  We now
look for a symbol $a_{ \hbar, t}(y, \theta)$ defined in a
neighbourhood of
$\left(x_0, \theta_0:= \frac{2}{i} \frac{\partial \Phi}{\partial
    x}(x_0) \right)$
such that the operator $ \Op_t(a_{ \hbar, t})$ does not depend on
$t$. Let $u \in H_{ \Phi, x_0}$ and denote $y_t(x,y)=tx+(1-t)y$; we
have:
\[
(2 \pi \hbar)^n \hbar D_t \Op_t(a_{ \hbar, t}) u(x) = \hbar D_t \left(
  \intint_{\Gamma(x)} e^{\frac{i}{\h}(x-y) \cdot \theta} a_{ \hbar,
    t}(y_t(x,y), \theta) u(y) \DD y \DD \theta \right),
\]
\begin{align}
  & =  \intint_{\Gamma(x)}  e^{\frac{i}{\h}(x-y)\cdot \theta}
    \hbar D_t \left( a_{ \hbar, t}(y_t(x,y), \theta) \right) u(y) \DD y \DD \theta, \\
  & =  \intint_{\Gamma(x)} \left( e^{\frac{i}{\h}(x-y)\cdot \theta} 
    \hbar D_t a_{ \hbar, t}   + e^{\frac{i}{\h}(x-y)\cdot \theta} (x-y) 
    \cdot \hbar D_y a_{ \hbar, t}  \right) (y_t(x,y), \theta) u(y) \DD y \DD \theta, \\
  & = \intint_{\Gamma(x)} \left( e^{\frac{i}{\h}(x-y)\cdot \theta}
    \hbar D_t a_{ \hbar, t}  + \hbar D_{ \theta} 
    (e^{\frac{i}{\h}(x-y)\cdot \theta})
    \cdot
    \hbar D_y a_{ \hbar, t} \right)(y_t(x,y), \theta) u(y) \DD y \DD \theta, \\
  & \sim_a  \intint_{\Gamma(x)} e^{\frac{i}{\h}(x-y)\cdot \theta} 
    \left(  \hbar D_t a_{ \hbar, t}  - \hbar D_{ \theta}\cdot  \hbar D_y a_{ \hbar, t} \right) 
    (y_t(x,y), \theta) 
    u(y) \DD y \DD \theta,
\end{align}
where the last equality holds modulo a negligible term, see
Definition~\ref{defi:negligible-germ}, and follows from Stokes'
formula and~\eqref{equ:kernel-pseudo}.  Consequently, the operator
$ \Op_t(a_{ \hbar, t})$ will be independent of the parameter $t$ if
the symbol $a_{\hbar, t}$ satisfies the following condition for
$(y, \theta) \in \Vois \left( x_0, \theta_0 \right)$:
\[
\left( \hbar D_t - \hbar D_{ \theta} \cdot \hbar D_y \right) a_{\hbar,
  t}(y, \theta) = 0.
\]
This will hold if the symbol $a_{ \hbar, t}$ satisfies the following
equality for $(y, \theta) \in \Vois \left( x_0, \theta_0\right)$:
\[
a_{ \hbar, t}(y, \theta) = e^{\frac{i}{\h}(t-s) \hbar D_{ \theta}
  \cdot \hbar D_y} a_{ \hbar, s} (y, \theta)\,.
\]
Similarly to the more general case treated below, the propagator
$e^{-\frac{it}{\h} (-\hbar D_{ \theta} \cdot \hbar D_y)}$ is an
analytic Fourier integral operator, with canonical relation:
\[
\kappa_t : (y, \theta;\, y^*, \theta^*) \mapsto (y-t \theta^*, \theta
- t y^*;\, y^*, \theta^*)\, .
\]
Because $ \kappa_t$ sends the zero section $\theta^* = 0$, $y^* = 0$
on itself, we may apply Proposition~\ref{prop:fio} with $\Phi=0$,
which gives that this propagator sends analytic symbols to analytic
symbols.

\medskip

We now wish to generalize this procedure to a symbol of the form
$b_{ \hbar, t}(x, y, \theta)$, defined for $0 \leq t \leq 1$ in a
neighbourhood of
$\left( x_0, x_0, \theta_0 = \frac{2}{i} \frac{\partial \Phi}{\partial
    x}(x_0) \right)$, and such that
\[
b_{ \hbar, 0}(x, y, \theta) := b_{ \hbar}(x, y, \theta) \quad
\text{defined by Equation \eqref{equation_symbol_b}}.
\]

Let $ \Op_t(b_{ \hbar, t})$ be the following operator for
$u \in H_{ \Phi, x_0}$:
\begin{align}
  & \Op_t(b_{ \hbar, t}) u(x) \\
  &= \dfrac{1}{(2 \pi \hbar)^n}
    \intint_{\Gamma(x)} e^{\frac{i}{\h}(x-y)\cdot \theta} b_{ \hbar, t}
    ((1-t)x+ty, tx+(1-t)y, \theta) u(y) \DD y \DD \theta 
\end{align}
Remark that, when $t = \frac{1}{2}$, we obtain the complex Weyl
quantization (see~\eqref{equ:complex-weyl}) of the symbol
$b^{\w}_{\hbar}$ defined in $\Vois(x_0, \theta_0)$ by
\[ b_{ \hbar, 1/2} \left( \frac{x+y}{2}, \frac{x+y}{2}, \theta \right)
=: b^{\w}_{\hbar} \left( \frac{x+y}{2}, \theta\right).
\]
In order to lighten notation, let
$X_t:=((1-t)x+ty, tx+(1-t)y, \theta)$. Then, for
$u \in H_{\Phi, x_0}$, we have
\[
(2 \pi \hbar)^n \hbar D_t \Op_t(b_{ \hbar, t}) u(x) =
\intint_{\Gamma(x)} e^{\frac{i}{\h}(x-y)\cdot \theta} \hbar D_t \left(
  b_{ \hbar, t} (X_t) \right) u(y) \DD y \DD \theta,
\]
\begin{align}
  & = \intint_{\Gamma(x)} e^{\frac{i}{\h}(x-y)\cdot \theta}
    \Big( \hbar D_t b_{ \hbar, t}(X_t) - (x-y) \hbar D_x b_{ \hbar, t}(X_t) \\
  & \hspace{0.4\linewidth}  + (x-y) \hbar D_y b_{ \hbar, t}(X_t) \Big) u(y) \DD y
    \DD \theta ,\\
  & =  \intint_{\Gamma(x)} \left(e^{\frac{i}{\h}(x-y)\cdot \theta} \hbar D_t b_{ \hbar, t} 
    - \hbar D_{ \theta} ( e^{\frac{i}{\h}(x-y)\cdot \theta} ) \left(\hbar D_x b_{ \hbar, t} 
    - \hbar D_y b_{ \hbar, t}\right)\right) (X_t) u(y) \DD y \DD \theta, \\
  & \sim_a  \intint_{\Gamma(x)} e^{\frac{i}{\h}(x-y)\cdot \theta} \left( \hbar D_t b_{ \hbar, t} + \hbar D_{ \theta}\cdot  \left(\hbar D_x - \hbar D_y \right) b_{ \hbar, t} \right)  (X_t) u(y) \DD y \DD \theta,
\end{align}
using Stokes' formula.  Thus, the operator $ \Op_t(b_{ \hbar, t})$,
acting on $\widetilde H_{\Phi,x_0}$, is independent of the parameter
$t$ if the symbol $b_{ \hbar, t}(x, y, \theta)$ satisfies the
following equality for
$(x, y, \theta) \in \Vois \left(x_0, x_0, \theta_0 \right)$:
\[
\left( \hbar D_t + \hbar D_{ \theta}\cdot \left( \hbar D_x - \hbar D_y
  \right) \right) b_{ \hbar, t}(x, y, \theta) = 0 .
\]
This leads to
\begin{equation}
  b_{ \hbar, t}(x, y, \theta) = U_{t-s} \; b_{ \hbar, s}(x, y, \theta) ,
\end{equation}
where
\[
U_t := \exp \left( -\dfrac{i}{\hbar} t \left( \hbar D_{ \theta} \cdot
    \left( \hbar D_x - \hbar D_y \right)\right) \right).
\]
By taking $ t= \frac{1}{2}$ and $s=0$, we obtain:
\[
b_{ \hbar, 1/2} \left( x, y, \theta \right) = U_{1/2} b_{ \hbar}(x, y,
\theta),
\]
and we recall that the Weyl symbol is defined as
$b_\h^{\w}(x,\theta) = b_{ \hbar, 1/2}(x, x, \theta)$.  Writing $U_t$
as a Fourier multiplier, \emph{i.e.}
\begin{equation}
  \label{equ:Ut}
  U_{t}=\mathcal{F}_\h^{-1}\circ \exp \left(\frac{it}{\h}\theta^* \cdot
    (y^*-x^*) \right)\circ \mathcal{ F}_\h ,
\end{equation}
where $\mathcal{F}_\h$ denotes the usual semiclassical Fourier
transform, we see that it is formally semiclassical analytic Fourier
integral operator (see Section~\ref{sec:fio}) whose principal symbol
is the constant $1$; it is the exponential of the differential
operator $P= \hbar D_{ \theta}\cdot ( \hbar D_x - \hbar D_y)$, acting
on formal analytic symbols.  Its canonical relation is actually the
graph of a symplectic diffeomorphism defined by
\begin{align}
  \kappa_t: T^*  \Vois(x_0, x_0, \theta_0) & \longrightarrow T^*  \Vois(x_0, x_0, \theta_0) \\
  (x, y, \theta;\, x^*, y^*, \theta^*) & \longmapsto (x, y, \theta;\, x^*, y^*, \theta^*) + t \ham{p} 
\end{align}
where $\ham{p}$ is the Hamiltonian vector field associated with the
symbol $p$ of the differential operator $P$, namely
$p(x, y, \theta;\, x^*, y^*, \theta^*) = \theta^*\cdot (x^* - y^*)$
and
$\ham{p} = \theta^*\cdot \partial_x - \theta^* \cdot \partial_y + (x^*
- y^*)\cdot \partial_{ \theta}$. Thus:
\begin{equation} \label{equa_transfo_canonique_kappa_t} \kappa_t: (x,
  y, \theta;\, x^*, y^*, \theta^*) \mapsto (x+ t \theta^*, y-t
  \theta^*, \theta + t(x^*-y^*);\, x^*, y^*, \theta^* ).
\end{equation}
Since $\kappa_t$ is a diffeomorphism, its phase function is strongly
non-degenerate in the sense of~\eqref{equ:ND-strong}.  Because
$ \kappa_t$ sends the zero section $x^*=0$, $y^*=0$, $\theta^*=0$ on
itself, we may apply Proposition~\ref{prop:fio} with $\Phi=0$, which
gives that the Fourier integral operator $U_t$ sends analytic symbols
to analytic symbols of the same order.  Besides, using analytic
stationary phase lemma, we obtain that $U_t$ sends classical analytic
symbols to classical analytic symbols (see also~\cite{sj-96}).  Let:
\begin{align}
  \label{equ:defi-gamma}
  \gamma : \Vois \left( x_0, \theta_0 \right) 
  & \longrightarrow \Vois \left(x_0, x_0, \theta_0  \right) \\
  (x, \theta) & \longmapsto (x, x, \theta).
\end{align}
Then, with $\gamma^*$ denoting pullback, we have:
\[
\gamma^* \left( U_{1/2} b_{ \hbar}(x, y, \theta) \right) = \gamma^*
\left( b_{ \hbar, 1/2} \left( x,y, \theta \right) \right) = b^{\w}_{
  \hbar} \left( x, \theta \right) .
\]
$ \gamma^* U_{1/2} b_{ \hbar}$ is a classical analytic symbol of order
zero that we denote by $b^{\w}_{ \hbar}$. To conclude, taking
$b_{ \hbar, t}$ such that $ \Op_t (b_{ \hbar, t})$ is independent of
$t$, gives us:
\[
\Op_{1/2}(b_{ \hbar, 1/2}) = \Op^{\w}_\h(b_{ \hbar}^{\w}) =
\Op^{\w}_\h(\gamma^* U_{1/2} b_{ \hbar}) \hnegl \Op_0(b_{ \hbar, 0}) =
B_{ \Gamma} .
\]

To summarize, we have the following proposition.

\begin{prop} \label{prop_lien_op_brg_op_weyl} Let
  $a_{ \hbar}(x, y, w)$ be a classical analytic symbol of order zero
  defined on a neighbourhood of $(x_0, x_0, \bar{x}_0)$. Then, on
  $H_{\Phi, x_0}$, we have:
  \[
  \Op^{\Brg}(a_{ \hbar}) \hnegl \Op^{\w}_\h(b^{\w}_{ \hbar}), \quad
  \text{where} \quad b^{\w}_{ \hbar} = \gamma^* U_{1/2} \tilde{J} W^*
  a_{ \hbar},
  \]
  where:
  \[
  \left\lbrace
    \begin{split}
      & W: (x, y, \theta) \mapsto (x, y, w(x, y, \theta)), \\
      & \qquad \text{with $w$ defined in
        Equation~\eqref{eqref:w(x,y,theta)}},\\
      & \gamma: (x, \theta) \mapsto (x, x, \theta), \\
      & U_{1/2} = \exp \left( \tfrac{i}{2 \hbar} \hbar D_{ \theta} \cdot  \left( \hbar D_y - \hbar D_x \right) \right), \\
      & \text{$\tilde{J}$ is defined by Equation \eqref{defi_J}.}
    \end{split}
  \right.
  \]
  Besides, $b^{\w}_{ \hbar}$ is a classical analytic symbol of order
  zero defined on a neighbourhood of
  $\left( x_0, \theta_0 := \frac{2}{i} \frac{\partial \Phi}{\partial
      x} (x_0) \right)$.
  Finally, if $a_\h\sim_a 0$, then $b^{\w}_{ \hbar}\sim_a 0$.
\end{prop}

The proof of the converse statement, namely that $a_\h$ is analytic
whenever $b^{\w}_\h$ is analytic, is the goal of the following section
(Proposition~\ref{prop:S}).

\subsection{Composition of Fourier integral operators}
\label{sec:FIO}

Let $b^{\w}_{ \hbar}(x, \theta)$ be a classical analytic symbol of
order zero defined on a neighbourhood of
$ \left( x_0, \theta_0 := \frac{2}{i} \frac{\partial \Phi}{\partial x}
  (x_0) \right)$.
We want to prove that there exists a classical analytic symbol of
order zero $a_{ \hbar}(x, w)$ defined on a neighbourhood of
$(x_0, \bar{x}_0)$ (and which does not depend on the $y$-variable)
such that the Brg-quantization of $a_{ \hbar}$ coincides with the
complex Weyl quantization of $b_{ \hbar}^{\w}$
(see~\eqref{equ:complex-weyl}).  Instead of doing this directly, let
us consider the map
\begin{align}
  \mathbf{S}: \hat S^0(\Vois(x_0, x_0, \bar{x}_0)) 
  & 
    \longrightarrow \hat S^0(\Vois \left(x_0, \theta_0 \right)) \\
  a_{ \hbar} 
  & \longmapsto b^{\w}_{ \hbar} = \gamma^* U_{1/2} \tilde{J} W^* a_{ \hbar},
\end{align}
restricted to the subset of classical analytic symbols of order zero
which do not depend on the $y$-variable.  We already proved in the
previous subsection that this map is well-defined in the sense that it
sends a formal classical analytic symbol of order zero to a formal
classical analytic symbol of order zero. Consequently, it suffices to
prove the following proposition in order to conclude the proof of
Theorem~\ref{theom}.

\begin{prop}
  \label{prop:S}
  The map $\mathbf{S}$ restricted to the set of classical analytic
  symbols which do not depend on the $y$-variable is an analytic
  Fourier integral operator associated with a canonical transformation
  which sends the zero section on itself. Moreover, this Fourier
  integral operator is elliptic.
\end{prop}

This proposition implies that the map $\mathbf{S}$ is a bijection from
the space of classical analytic symbols of order zero defined in a
neighbourhood of $(x_0, \bar{x}_0)$ to the space of classical analytic
symbols of order zero defined in a neighbourhood of
$ \left( x_0, \theta_0\right)$.

Let
\begin{align}
  \label{equ:def-pi}
  \pi : \Vois \left( x_0, x_0, \bar{x}_0 \right) 
  & \longrightarrow \Vois (x_0, \bar{x}_0) \\
  (x, y, w) & \longmapsto (x, w).
\end{align}
Let $a_{ \hbar}(x, w)$ be a classical analytic symbol of order zero,
which we view as a function of $(x, y, w)$ by identifying it with
$ \pi^* a_{ \hbar}(x, y, w) = a_{ \hbar} \circ \pi (x, y, w)$. We use
the maps $W,\gamma,U_{\frac{1}{2}}$ and $\tilde J$ from
Proposition~\ref{prop_lien_op_brg_op_weyl}.  According to this
proposition, we introduce
\begin{equation}
  \label{equ:defA} 
  b^{\w}_{ \hbar} = \gamma^* U_{1/2}
  \tilde{J} W^* \pi^* a_{ \hbar} =: A a_{ \hbar}.
\end{equation}
The operator $A$ acting on symbols
$a_\h\in S^0(\Vois (x_0, \bar{x}_0))$ is the composition of the five
operators ($ \gamma^*$, $U_{1/2}$, $\tilde J$, $ W^*$ and $ \pi^*$).
We shall give two independent proofs that this composition is an
analytic Fourier integral operator: first by proving that all these
operators are good analytic Fourier integral operators and applying
Proposition~\ref{prop:fio}; in the second proof (Appendix~\ref{dsa})
we give an explicit computation with stationary phase arguments in
order to obtain a simple formula for $A$ (Equation~\eqref{dsa.10}).

\subsection{Proof of Proposition~\ref{prop:S}}
\label{sec:transco}
\paragraph{The operator $\pi^*$.}
Recall from~\eqref{equ:def-pi} that $(\pi^*u)(x,y,w)=u(x,w)$. We have
\[
(\pi^*u)(x,y,w) = \frac{1}{(2\pi\h)^{2n}}\intint\!\!\!\!\intint
e^{\frac{i}{\h}[(x-\tilde x)\cdot\theta + (w-\tilde w)\cdot\omega]}
u(\tilde x,\tilde w)\DD x \DD \theta \DD \tilde w \DD \omega\,.
\]
The operator $\pi^*$ is an elliptic analytic Fourier integral operator
and
\[
\phy(x,y,w;\tilde x, \tilde w; \theta, \omega) = (x-\tilde
x)\cdot\theta + (w-\tilde w)\cdot\omega
\]
 is a strongly non-degenerate
phase function in the sense of~\eqref{equ:ND-strong} (\emph{i.e.},
when $(x,y,w)$ is fixed), with critical variety
\[
C_\phy = \{(x,y,w;\tilde x, \tilde w; \theta, \omega); \quad x =
\tilde x, w = \tilde w\}\,.
\]
From this we get the canonical relation $K_{\pi^*}$:
\[
K_{\pi^*} = \{(x,y,w;\theta,0,\omega), (x,w;\theta,\omega)\} =
\{\left( (a; \trsp\DD\pi_a b^*),\; (\pi(a); b^*) \right)\}\,,
\]
where $a=(x,y,w)\in\CM^{3n}$ and
$b^*=(\theta,\omega)\in(\CM^{2n})^*$. It maps the zero section
$\{b^*=0\}\subset T^*\CM^{2n}$ to the zero section
$\{a^*=0\}\subset T^*\CM^{3n}$, and the inverse image of
$T^*_a\CM^{3n}$ is
\[
\{(\pi(a);b^*); \quad b^*\in(\CM^{2n})^*\} = T^*_{\pi(a)}\CM^{2n}\,,
\]
which intersects the zero section $\{b^*=0\}$ transversally.
Therefore, we may apply Proposition~\ref{prop:fio}, and
$\pi^*: \widetilde{H}_{0,b}\to \widetilde H_{0,a}$ is an analytic
Fourier integral operator.

\paragraph{The operator $W^*$.} Recall that $W:\CM^{3n}\to \CM^{3n}$
is a locally defined diffeomorphism and $W^*u(a) = u(W(a))$,
$a\in\CM^{3n}$. We can write $W^*$ as an elliptic analytic Fourier
integral operator by the formula
\[
W^* u (a) = \frac{1}{(2\pi\h)^{3n}}\intint e^{\frac{i}{\h}(W(a) -
  c)\cdot c^*} u(c) \DD c \DD c^*\,.
\]
The phase $(W(a)-c)\cdot c^*$ is non-degenerate as a function of
$(c,c^*)$ with critical manifold $\{(a,c,c^*); W(a)=c\}$. The
canonical relation is the graph of the lifted symplectic
transformation, \emph{i.e.}
\[
K_{W^*} = \{\left( (a; \trsp W'(a) c^*),\; (W(a); c^*) \right)\}\,.
\]
It maps the zero section to the zero section, and
$K_{W^*}^{-1}(T^*_a\CM^{3n}) = T^*_{W(a)}\CM^{3n}$, which is
transversal to the zero section. Thus we may apply
Proposition~\ref{prop:fio}.

\paragraph{The operator $\tilde J$} is a multiplication operator,
$K_{\tilde J}=\textup{Id}$.

\paragraph{The operator $U_{\frac{1}{2}}$.} We have seen
in~\eqref{equ:Ut} (and below that) that $U_{\frac{1}{2}}$ is an
analytic Fourier integral operator with associated canonical
transformation $\kappa_{\frac{1}{2}}$ given by
\[
\kappa_{\frac{1}{2}}(a,a^*) = (a + h(a^*), \; a^*),
\]
where $a\in\CM^{3n}$ and $h$ is the block-matrix
$h= \frac{1}{2}\begin{pmatrix}
  0 & 0 & 1\\
  0 & 0 & -1\\
  1 & -1 & 0
\end{pmatrix}$.
It follows that, when $c\in\CM^{3n}$ is fixed,
$\kappa^{-1}_\frac{1}{2}(T^*_c\CM^{3n}) = \{(c-h(a^*), a^*), \; a^*
\in\CM^{3n})\}$
is parametrized by $a^*$ and transversal to the zero section, which
permits the application of Proposition~\ref{prop:fio}.

\paragraph{The operator $\gamma^*$.} Recall
from~\eqref{equ:defi-gamma} that
$\gamma^*u(x,\theta) = u(x,x,\theta)$, so
\[
\gamma^*u(b) = \frac{1}{(2\pi\h)^{3n}} \intint e^{\frac{i}{\h}
  [(\gamma(b)-c)\cdot c^*)]} u(c) \DD c \DD c^*\,,
\]
with $c=(x,y,\theta)\in\CM^{3n}, c^*\in(\CM^{3n})^*$. Again, we see
that $\gamma^*$ is an elliptic, analytic Fourier integral operator,
with non-degenerate phase $\phy(b,c,c^*)$, and since $\gamma^*$
is a pull-back, we obtain, as for $\pi^*$ and $W^*$,
\[
K_{\gamma^*} = \{\left( (b, \trsp\DD\gamma_b c^*),\; (\gamma(b), c^*)
\right) ; \quad b\in \CM^{2n}, c^*\in (\CM^{3n})^*\}\,.
\]
Again, it maps the zero section $\{c^*=0\}\subset T^*\CM^{3n}$ to the
zero section in $T^*\CM^{2n}$, and
$K_{\gamma^*}^{-1}(T^*_b\CM^{2n})= T_{\gamma(b)}\CM^{3n}$, which is
transversal to the zero section $\{c^*=0\}$. Hence
Proposition~\ref{prop:fio} can be applied.

To conclude, we have shown that all the compositions involved in the
operator $A$ are transverse and with non-vanishing symbol, making it
an analytic Fourier integral operator
$\widetilde H_{0,(x_0,\bar x_0)} \to \widetilde H_{0,(x_0,\theta_0)}$,
which is elliptic and whose associated canonical transformation
$T^*\CM^{2n}\to T^*\CM^{2n}$ send the zero section to itself, which
proves Proposition~\ref{prop:S}. Choosing a Fourier integral operator
$B$ associated with the inverse canonical transformation, and applying
analytic ellipticity to the pseudo-differential operators $AB$ and
$BA$, we construct in the usual way a local inverse to $A$, sending
$\widetilde H_{0,(x_0,\theta_0)}$ to
$\widetilde H_{0,(x_0,\bar x_0)}$, thus proving Theorem~\ref{theom}.

\subsection{Local Toeplitz operators}

An as application of the results in this section, one can see that the
calculus of local analytic Toeplitz operators is equivalent to the one
of analytic $\Brg$ operators (and hence to the calculus of complex
pseudo-differential operators).

Choose $b_\hbar=1$ in Theorem~\ref{theom} (case~\ref{item:2}), and let
$a^0_\hbar$ be the corresponding symbol $a_\hbar$ given there. We then
get a local candidate $\Pi =\mathrm{Op}^{\mathrm{Brg}}(a_\hbar^0)$ for
the Bergman projection in $H_{\Phi ,x_0}$. Let
$q_\h=q(x,\overline{x};\hbar )$ be a classical analytic symbol of
order $0$, defined near $(x_0,\overline{x}_0)$. The corresponding
local Toeplitz operator is then by definition
\begin{equation}\label{lto.1}
  \Top(q_\h)=\Pi \circ q_\h,\hbox{ i.e.\ }\Top(q_\h)u=\Pi
  (q_\h u),\ u\in L^2_{\Phi ,x_0}.
\end{equation}

In the spirit of Theorem 3.1, we obtain the following result.
\begin{prop} Given such a $q_\h$, there exists a unique classical
  analytic symbol $a_\hbar$ of order 0, defined near
  $(x_0,\overline{x}_0)$, such that
  \[
    \mathrm{Top\,}(q_\h)\hnegl \Opbrg(a_\hbar ).
  \]
  Moreover, the map $q_\h\mapsto a_\h$ is a bijection on the set of
  (germs of) formal classical analytic symbol, defined near
  $(x_0,\overline{x_0})$.
\end{prop}

\begin{proof}
  We have
\[\begin{split}
    \Top(q_\h)u(x)&=\frac{1}{(2\pi \hbar )^n}\int e^{\frac{2}{\hbar
      }\psi (x,\overline{y})}a^0(x,\overline{y};\hbar
    )q(y,\overline{y};\hbar )u(y)e^{-\frac{2}{\hbar }\Phi (y)}L(dy)\,,
\end{split}\]
in other words, $\Top(q_\h) = \Opbrg(\tilde a_\h)$, with
\[
\tilde a (x,y,w;\h) := a^0(x,w;\h) q(y,w;\h).
\]
By Proposition~\ref{prop_lien_op_brg_op_weyl}, we obtain a classical
analytic symbol $b_\h$ defined near $(x_0,\theta_0)$ such that
$\Top(q_\h) \hnegl \Op^{\w}_\h(b_{ \hbar})$, and it is explicitly
given by the formula
\[
b_{ \hbar} = \gamma^* U_{1/2} \tilde{J} W^*
  a^0_{ \hbar} q_\h\,.
\]
Since $a^0_\h$ was obtained from the constant symbol $1$ by an
elliptic Fourier integral operator (Proposition~\ref{prop:S}), it is
elliptic near $(x_0,\overline{x_0})$. We may now argue as in
Proposition~\ref{prop:S}, except that the operator $\pi$ used there
must be replaced by $\pi_y := (x,y,w) \mapsto (y,w)$. Clearly, the
permutation $y\leftrightarrow x$ does not change the fact that we have
an elliptic Fourier integral operator, associated to a canonical
transformation which maps the zero section to the zero section.

Therefore, it follows from Section~\ref{sec:transco} that the
composition
$\gamma^* \circ U_{1/2} \circ \tilde{J} \circ W^* \circ a^0_{ \hbar}
\circ \pi_y^*$ is an elliptic analytic Fourier integral operator, and
hence the local map $ q_\h \mapsto b_\h$ is formally invertible in the
analytic sense. To conclude, it suffices to compose this bijection
with the one of Theorem~\ref{theom}, which gives $a_\h$ such that
\[
\Opbrg(a_\h) \hnegl \Op^{\w}_\h(b_{ \hbar}) \hnegl \Top(q_\h).
\]
\end{proof}

\section{The approximate Bergman projection}
\label{sec:appr-bergm-proj}

\subsection{Functional analysis of \texorpdfstring{$\lphi$}{Lphi}
  spaces}

\begin{defi}
  \label{defi:lphi}
  Let $ \Omega$ be an open subset of $ \CM^n$. Let
  $ \Phi \in \mathscr{C}^0( \Omega; \mathbb{R})$. For any $\h>0$, we
  define the following spaces.
  \begin{enumerate}%[label=\alph*)]
  \item $L^2_\Phi(\Omega)$ is the $L^2$-space with weight
    $e^{-2\Phi/\h}$ on $\Omega$; it is a Hilbert space with the norm
    \[
    \norm{u}_{L^2_\Phi(\Omega)} := \|u e^{- \Phi/ \hbar}
    \|_{L^2(\Omega)} = \int \abs{u(x)}^2 e^{-2\Phi(x)/\h}L(\DD x).
    \]

  \item $L_{ \Phi,\loc}^{2}(\Omega)$ is the Fréchet space
    $L^2_\loc(\Omega)$ equipped with the set of seminorms
    $\norm{u}_{L^2_\Phi(\tilde\Omega)}$, where
    $\tilde\Omega\Subset\Omega$ is an arbitrary open set with compact
    closure in $\Omega$.

  \item $L^2_{\Phi,\textup{comp}}(\Omega)$ is the space of compactly
    supported functions in $L^2_\Phi(\Omega)$.
  \end{enumerate}
\end{defi}
Since $\Phi$ is continuous, for any fixed $\h$ we have the set
equality
$L^2_{\Phi,\textup{comp}}(\Omega)= L^2_{\textup{comp}}(\Omega)$.
Similarly to the space $\Cinf_0(\Omega)$ in distribution theory,
$L^2_{\Phi,\textup{comp}}$ is a projective limit of Fréchet spaces,
and following the tradition we will only use the convergence of
sequences: for a fixed $\h$, the sequence $(u_j)_{j\in\NM}$ converges
to $u$ in $L^2_{\Phi,\textup{comp}}$ if the support of all $u_j$ is
contained in a fixed subset $\tilde\Omega\Subset \Omega$ and
$\norm{u_j - u}_{\lphi(\tilde\Omega)}\to 0$. Thus, the injection
$L^2_{\Phi,\textup{comp}}\subset L^2_{ \Phi,\loc}$ is sequentially
continuous; and moreover we have a well-defined pairing on
$L_{ \Phi,\loc}^{2} \times L^2_{\Phi,\textup{comp}}$ given by
\[
\pscal{u}{v}_{\lphi} = \int u(x) \overline{v(x)} e^{-2\Phi(x)/\h}
L(\DD x),
\]
which is continuous in the first factor and sequentially continuous in
the second one.

From the Fréchet topology of $L_{ \Phi,\loc}^{2}(\Omega)$ we obtain
that, for a fixed $\h>0$, a linear operator
$A:L_{ \Phi_1,\loc}^{2}(\Omega) \to L_{ \Phi_2,\loc}^{2}(\Omega)$ is
continuous if and only if for every $\Omega_2\Subset\Omega$, there
exist $\Omega_1\Subset\Omega$ and a constant $C>0$ such that
\begin{equation}
  \norm{A u}_{L^2_{\Phi_2}(\Omega_2)} \leq C \norm{u}_{L^2_{\Phi_1}(\Omega_1)}.
  \label{equ:cont-lphi}
\end{equation}
Notice that \eqref{equ:cont-lphi} implies that if $u$ vanishes on
$\Omega_1$, then $Au$ vanishes on $\Omega_2$; in other words, if we
regard the support of the distribution kernel $K_A$ of $A$ as a
relation from $\Omega$ to itself, we have
\begin{equation}
  (\spt K_A)^{-1}(\Omega_2):= \{y\in\Omega; \quad \exists (x,y)\in
  (\spt K_A)\cap \Omega_2\times \Omega\} \subset \Omega_1.
  \label{equ:proper2}
\end{equation}
On the other hand, an operator
$A:L_{ \Phi_1,\textup{comp}}^{2}(\Omega) \to L_{
  \Phi_2,\textup{comp}}^{2}(\Omega)$
is continuous if and only if for every $\Omega_1\Subset\Omega$, there
exists $\Omega_2\Subset\Omega$ such that
\begin{equation}
  \spt u \subset\Omega_1 \implies \spt Au \subset \Omega_2
  \label{equ:proper1}
\end{equation}
and $A$ is continuous as an operator from $\lphi(\Omega_1)$ to
$\lphi(\Omega)$. Then $A$ sends convergent sequences in
$ L_{ \Phi_1,\textup{comp}}^{2}(\Omega)$ to convergent sequences in
$ L_{ \Phi_2,\textup{comp}}^{2}(\Omega)$.

An operator that satisfies both~\eqref{equ:proper2}
and~\eqref{equ:proper1} is called \emph{properly supported}.  Notice
that the injections
$\Cinf_0(\Omega)\subset L^2_{\Phi,\textup{comp}}(\Omega)$ and
$L^2_{\Phi,\loc}(\Omega)\subset \mathscr{D}'(\Omega)$ are
(sequentially) continuous.  Hence any continuous operator
$A:L_{ \Phi_1,\textup{comp}}^{2}(\Omega_1) \to L_{
  \Phi_2,\loc}^{2}(\Omega_2)$
admits a Schwartz kernel
$K_A\in\mathscr{D}'(\Omega_2\times \Omega_1)$.

If $A$ has kernel $K_A$, its formal adjoint, denoted by $A^*$, is the
operator defined by the kernel $(x,y)\mapsto \overline{K_A(y,x)}$.  We
see that taking formal adjoint swaps properness
conditions~\eqref{equ:proper2} and~\eqref{equ:proper1}. Hence if
$A:L_{ \Phi_1,\loc}^{2}(\Omega_1) \to L_{ \Phi_2,\loc}^{2}(\Omega_2)$
is continuous, then
$A^*:L_{ \Phi_1,\textup{comp}}^{2}(\Omega_2) \to L_{
  \Phi_2,\textup{comp}}^{2}(\Omega_1)$ is continuous, and conversely.

We now introduce uniform versions of these remarks, as $\h\to 0$.
\begin{defi}\label{defi:unif-cont}
  A linear operator
  $A=A_\h : L_{ \Phi_1,\loc}^{2}(\Omega_1) \to L_{
    \Phi_2,\loc}^{2}(\Omega_2)$
  is \textbf{uniformly continuous}, and we write:
  \[
  A = \mathcal{O}(1): L_{ \Phi_1,\loc}^{2}(\Omega_1) \to L_{
    \Phi_2,\loc}^{2}(\Omega_2)
  \]
  if there exists $\h_0>0$ such that, for every
  $\tilde\Omega_2\Subset\Omega_2$, there exist
  $\tilde\Omega_1\Subset\Omega_1$, and a constant $C>0$, both
  independent of $\h$, such that, for all
  $u\in L^2_{\Phi_1,\loc}(\Omega_1)$,
  \begin{equation}
    \forall \h\in\;]0,\h_0], \qquad 
    \norm{A u}_{L^2_{\Phi_2}(\tilde \Omega_2)} \leq 
    C \norm{u}_{L^2_{\Phi_1}(\tilde \Omega_1)}.
    \label{equ:unif-cont-lphi}
  \end{equation}
\end{defi}
\begin{defi}
  A linear operator
  $A=A_\h : L_{ \Phi_1,\loc}^{2}(\Omega_1) \to L_{
    \Phi_2,\loc}^{2}(\Omega_2)$
  is \textbf{uniformly properly supported} if the projections from
  $\spt K_A$ to the factors $\Omega_1$ and $\Omega_2$ are uniformly
  proper, in the following sense: there exists $\h_0>0$ such that the
  following properties hold:
  \begin{enumerate}
  \item For all $\tilde\Omega_2\Subset\Omega_2$, there exists
    $\tilde \Omega_1\Subset \Omega_1$, independent of $\h$, such that
    \begin{equation}
      \forall \h\in\;]0,\h_0], \qquad 
      (\spt K_A)^{-1}(\tilde \Omega_2) \subset \tilde\Omega_1;
      \label{equ:unif-proper2}    
    \end{equation}

  \item For all $\tilde\Omega_1\Subset\Omega_1$, there exists
    $\tilde \Omega_2\Subset \Omega_2$, independent of $\h$, such that
    \begin{equation}
      \forall \h\in\;]0,\h_0], \qquad
      (\spt K_A) (\tilde\Omega_1) \subset \tilde\Omega_2.
      \label{equ:unif-proper1}
    \end{equation}
  \end{enumerate}
\end{defi}
\begin{prop}
  \label{prop:chi1-chi2}
  If
  $A:L_{ \Phi_1,\loc}^{2}(\Omega_1) \to L_{
    \Phi_2,\loc}^{2}(\Omega_2)$
  satisfies~\eqref{equ:unif-proper2}, then $A$ is uniformly continuous
  if and only if for all $\chi_j\in\Cinf_0(\Omega_j)$, $j=1,2$, the
  operator
  $\chi_2 A \chi_1: L^2_{\Phi_1}(\Omega_1)\to L^2_{\Phi_2}(\Omega_2)$
  is uniformly bounded, as $\h\to 0$.
\end{prop}
\begin{proof}
  If $A$ is uniformly continuous, and $\chi_j\in\Cinf_0(\Omega_j)$,
  $j=1,2$, are given, let $\tilde\Omega_2\Subset\Omega_2$ contain the
  support of $\chi_2$. Then there exists
  $\tilde\Omega_1\Subset\Omega_1$ such that,
  \begin{align}
    \norm{\chi_2 A \chi_1 u}_{L^2_{\Phi_2}(\Omega_2)} 
    & 
      \leq
      \norm{\chi_2}_{L^\infty}\norm{A \chi_1 u}_{L^2_{\Phi_2}(\tilde
      \Omega_2)} \leq C\norm{\chi_2}_{L^\infty}\norm{\chi_1
      u}_{L^2_{\Phi_1}(\tilde \Omega_1)} \\
    &
      \leq C\norm{\chi_2}_{L^\infty}\norm{\chi_1}_{L^\infty}\norm{
      u}_{L^2_{\Phi_1}(\tilde \Omega_1)} \\
    &
      \leq  C\norm{\chi_2}_{L^\infty}\norm{\chi_1}_{L^\infty}\norm{
      u}_{L^2_{\Phi_1}(\Omega_1)}.
  \end{align}
  Conversely, if $\tilde\Omega_2$ is given, let
  $\chi_2\in\Cinf(\Omega_2)$ be such that $\chi_2 = 1$ on a
  neighbourhood of $\overline{\tilde\Omega_2}$: we have
  \[
  \norm{A u}_{L^2_{\Phi_2}(\tilde \Omega_2)} = \norm{\chi_2 A
    u}_{L^2_{\Phi_2}(\tilde \Omega_2)} \leq \norm{\chi_2 A
    u}_{L^2_{\Phi_2}(\Omega_2)}.
  \]
  Now let $\chi_1\in\Cinf(\Omega_1)$ be such that $\chi_1 = 1$ on
  a neighbourhood of the compact set $(\spt K_A)^{-1}(\spt \chi_2)$
  that is independent of $\h$ (this is possible thanks
  to~\eqref{equ:unif-proper2}). Then
  \[
  \norm{\chi_2 A u}_{L^2_{\Phi_2}(\Omega_2)} = \norm{\chi_2 A \chi_1^2
    u}_{L^2_{\Phi_2}( \Omega_2)} \leq C \norm{\chi_1 u}_{L^2_{\Phi_1}(
    \Omega_1)}
  \]
  Finally we may choose $\tilde\Omega_1\Subset\Omega_1$ containing the
  support of $\chi_1$, and get
  $\norm{\chi_1 u}_{L^2_{\Phi_1}( \Omega_1)} \leq C_1
  \norm{u}_{L^2_{\Phi_1}( \tilde\Omega_1)} $,
  proving that $A$ is uniformly continuous.
\end{proof}
\begin{coro}
  If
  $A:L_{ \Phi_1,\loc}^{2}(\Omega_1) \to L_{
    \Phi_2,\loc}^{2}(\Omega_2)$
  is uniformly properly supported, then $A$ is uniformly continuous if
  and only if $A^*$ is uniformly continuous.
\end{coro}
\begin{proof}
  If
  $\chi_2 A \chi_1: L^2_{\Phi_1}(\Omega_1)\to L^2_{\Phi_2}(\Omega_2)$
  is uniformly bounded, then its adjoint
  $\overline{\chi_1}A^*\overline{\chi_2}$ is uniformly bounded.
\end{proof}

\begin{prop}
  \label{prop:phi1}
  Let $A:L_{ \Phi,\loc}^{2}(\Omega) \to L_{ \Phi,\loc}^{2}(\Omega)$ be
  uniformly continuous.  Let $\Phi_1\in \mathscr{C}(\Omega;\RM)$ be
  such that $\Phi_1<\Phi$. Then there exists
  $\Phi_2\in \mathscr{C}(\Omega;\RM)$, $\Phi_2<\Phi$, such that
  \[
  A = \mathcal{O}(1) : L_{ \Phi_1,\loc}^{2}(\Omega) \to L_{
    \Phi_2,\loc}^{2}(\Omega)
  \]
\end{prop}
\begin{proof}
  Let $U_m\Subset \Omega$, $m\in\NM$, be open subsets that form a
  locally finite covering of $\Omega$. For each $U_j$, by uniform
  continuity, there exists $V_j\Subset \Omega$ such that
  \begin{equation}
    \label{eq:3}
    \norm{A u}_{L^2_{\Phi}(U_j)} \leq C_j \norm{u}_{L^2_{\Phi}(V_j)},
  \end{equation}
  for some uniform constant $C_j$.  Let
  $\epsilon_j:= \inf_{V_j} (\Phi-\Phi_1)$ and
  \begin{equation}
    \label{equ:epsilon}
    \tilde\epsilon_m:=\min_{j; \,U_j\cap U_m \neq \varnothing} \epsilon_j.    
  \end{equation}
  We define $\Phi_2=\Phi-\sum_m \tilde\epsilon_m \chi_m$, where
  $\chi_m\in\Cinf_0(U_m; [0,1])$ form a partition of unity associated
  with the covering $(U_m)$. On each $U_j$ we have
  $\Phi_2\geq \Phi- \epsilon_j$, and on each $V_j$ we have
  $\Phi- \epsilon_j \geq \Phi - (\Phi-\Phi_1) = \Phi_1$.  Hence
  \begin{equation}
    \label{eq:3}
    \norm{A u}_{L^2_{\Phi_2}(U_j)} \leq \norm{A u}_{L^2_{\Phi-\epsilon_j}(U_j)}  
    \leq  C_j
    \norm{u}_{L^2_{\Phi-\epsilon_j}(V_j)} \leq  
    C_j \norm{u}_{L^2_{\Phi_1}(V_j)} .
  \end{equation}

\end{proof}

Notice that if $\Phi'<\Phi$, then we have a uniformly continuous
injection $L^2_{\Phi'}\subset L^2_\Phi$, which is exponentially small,
as $\h\to 0$, on every compact set. Thus, we introduce the next
definition.
\begin{defi}\label{defi:negligible}
  Let $A:L_{ \Phi,\loc}^{2}(\Omega) \to L_{ \Phi,\loc}^{2}(\Omega)$ be
  uniformly properly supported. We say that $A$ is \textbf{negligible}
  and write $A\equiv_\Phi 0$ if there exists a continuous function $\Phi_2$
  on $\Omega$ such that $\Phi_2<\Phi$ and
  \[
  A = \mathcal{O}(1) : L_{ \Phi,\loc}^{2}(\Omega) \to L_{
    \Phi_2,\loc}^{2}(\Omega)
  \]
\end{defi}
If $A$ is negligible, then it is also negligible in the sense of
Definition~\ref{defi:H-negligible}.
\begin{prop}
  \label{prop:negligible}
  If $A:L_{ \Phi,\loc}^{2}(\Omega) \to L_{ \Phi,\loc}^{2}(\Omega)$ is
  uniformly properly supported, then the following statements are
  equivalent:
  \begin{enumerate}%[label=(\roman*)]
  \item \label{item:negl-1} $A\equiv_\Phi 0$.
  \item \label{item:negl-2} for all $\chi_j\in\Cinf_0(\Omega)$,
    $j=1,2$, there exist $\h_0>0$, $C>0$ such that, for all
    $\h\in(0,\h_0]$,
    \[
    \norm{\chi_2 A \chi_1}_{\mathcal{L}(L^2_{\Phi}(\Omega),
      L^2_{\Phi}(\Omega))} \leq C e^{-\frac{1}{C\h}}.
    \]
  \item \label{item:negl-3} there exist
    $\Phi_j\in \mathscr{C}(\Omega;\RM)$, $j=1,2$ such that
    $\Phi_2<\Phi<\Phi_1$ and
    \[
    A = \mathcal{O}(1) : L_{ \Phi_1,\loc}^{2}(\Omega) \to L_{
      \Phi_2,\loc}^{2}(\Omega)
    \]
  \end{enumerate}
\end{prop}
Before the proof, we widen the scope and recall a notion of formal
exponential estimates of distribution kernels, introduced in Section
2.2 in \cite{HeSj85}: Let $\Omega _j\subset \CM^{n_j}$ be open for
$j=1,2$, let
$A=A_h:\, C_0^\infty (\Omega _1)\to \mathcal{ D}'(\Omega _2)$,
$0<h\le h_0$, be a linear operator with distribution kernel
$K_A\in \mathcal{ D}'(\Omega _2\times \Omega _1)$, and let
$F\in C(\Omega _2\times \Omega _1;\RM)$. Then we write
\begin{equation}\label{pf.1}
  K_A(x,y)=\widetilde{\mathcal{ O}}(1)e^{F(x,y)/h}=\widetilde{\mathcal{ O}}(e^{F(x,y)/h}),
\end{equation}
if for all $(x_0,y_0)\in \Omega _2\times \Omega _1$ and $\epsilon >0$,
there exist $C>0$ and open neighborhoods $V_{x_0}\subset \Omega _2$,
$V_{y_0}\subset \Omega _1$ of $x_0$, $y_0$ respectively, such that
$1_{V_{x_0}}\circ A$ is bounded:
$L^2_{\mathrm{comp}}(V_{y_0})\to L^2(V_{x_0}) $ with operator norm
$\le Ce^{\epsilon /h}e^{F(x_0,y_0)/h}$. Sometimes, we simply write
\begin{equation}\label{pf.2}
  A=\widetilde{\mathcal{ O}}(1)e^{F(x,y)/h}:\ L^2_{\mathrm{comp}}(\Omega
  _1)\to L^2_{\mathrm{loc}}(\Omega _2).
\end{equation}

We have (\ref{pf.1}) if and only if for all
$(x_0,y_0)\in \Omega _2\times \Omega _1$ and $\epsilon >0$, there
exist $C>0$ and
$\chi _{x_0}\subset C_0^\infty (\Omega _2;[0,+\infty [)$,
$\chi _{y_0}\subset C_0^\infty (\Omega _1;[0,+\infty [)$, equal to 1
near $x_0$ and $y_0$ respectively, such that
$\chi _{x_0}\circ A \circ \chi _{y_0}$ is bounded:
$L^2(\CM^{n_1})\to L^2(\CM^{n_2}) $ with operator norm
$\le Ce^{\epsilon /h}e^{F(x_0,y_0)/h}$.

When (\ref{pf.1}) holds it is then natural to write instead of
(\ref{pf.2}):
\begin{equation}\label{pf.3}
  A=\widetilde{\mathcal{ O}}(1)e^{F(x,y)/h}:\ L^2_{\mathrm{loc}}(\Omega
  _1)\to L^2_{\mathrm{loc}}(\Omega _2).
\end{equation}
Let $\Gamma =\overline{\bigcup_{]0,h_0]}\mathrm{supp\,}(K_{A_h}) }$,
so that the natural projections $\pi _j:\Gamma \to \Omega _j$ are
proper

\par With $F$ as above, let $\Phi _j\in C(\Omega _j;\RM)$, $j=1,2$. If
\begin{equation}\label{pf.4}
  F(x,y)+\Phi _1(y)\le \Phi _2(x) \hbox{ on }\Gamma 
\end{equation}
and (\ref{pf.1}) holds, then
\begin{equation}\label{pf.5}
  A=\widetilde{\mathcal{ O}}(1):\ L^2_{\Phi _1,\mathrm{loc}}(\Omega _1)\to L^2_{\Phi _2,\mathrm{loc}}(\Omega _2)
\end{equation}
in the following sense:
\begin{equation}\label{pf.6}\begin{split}
    &\hbox{For every open set }V_2\Subset \Omega _2, \ \exists \hbox{
      an
      open set }V_1\Subset \Omega _1,\\
    &\hbox{such that for every }\epsilon >0,\ \exists\, C>0\hbox{ such
      that}\\
    &\| Au\|_{L^2_{\Phi _2}(V_2)}\le Ce^{\epsilon /h} \|
    u\|_{L^2_{\Phi _1}(V_1)},\ \forall\ u\in L^2_{\Phi
      _1,\mathrm{loc}}(\Omega _1).
  \end{split}
\end{equation}
Conversely, if (\ref{pf.5}) holds, then
\begin{equation}\label{pf.7}
  K_A(x,y)=\widetilde{\mathcal{ O}} (1)e^{(\Phi _2(x)-\Phi _1(y))/h}.
\end{equation}

If we sharpen (\ref{pf.4}) by assuming strict inequality there, then
(\ref{pf.5}) can be replaced by the sharper statement that
\[
A=\mathcal{ O}(1):\ L^2_{\Phi _1,\mathrm{loc}}(\Omega _1)\to L^2_{\Phi
  _2,\mathrm{loc}}(\Omega _2),
\]
as defined earlier.
\begin{proof}[Proof of Proposition \ref{prop:negligible}] We
  show~\ref{item:negl-1} $\Rightarrow$~\ref{item:negl-2}
  $\Rightarrow$~\ref{item:negl-3} $\Rightarrow$~\ref{item:negl-1}.

  If~\ref{item:negl-1} holds, then there exists a continuous function
  $\Phi_2$ on $\Omega_2$ such that $\Phi_2<\Phi$ and
  \begin{equation}\label{pf.8}
    K_A(x,y)=\widetilde{\mathcal{ O}} (1)e^{(\Phi _2(x)-\Phi (y))/h}.
  \end{equation}
  It follows that if $\chi _j\in C_0^\infty (\Omega )$, $j=1,2$ and
  $\widetilde{C}>0$ and $1/\widetilde{C}$ is strictly smaller than
  \[
  \inf_{\mathrm{supp\,}\chi _2\times \mathrm{supp\,}\chi _1}(\Phi
  (x)-\Phi (y))-(\Phi _2(x)-\Phi (y)) = \inf_{\mathrm{supp\,}\chi
    _2\times \mathrm{supp\,}\chi _1}(\Phi (x)-\Phi _2(x)),
  \]
  then $\exists$ $\widehat{C}>0$ such that
  \[
  \| \chi _2A\chi _1\|_{\mathcal{ L}(L^2_\Phi (\Omega ), L^2_\Phi
    (\Omega ))}\le \widehat{C}e^{-1/(\widetilde{C}h)},
  \]
  and we get~\ref{item:negl-2} with
  $C=\max (\widetilde{C},\widehat{C})$.

  Assume~\ref{item:negl-2}. Let $U_j\Subset \Omega $, $j=1,2,\dots$ be
  open, forming a locally finite covering of $\Omega $. Let
  $1_{U_j}\le \chi _j\in C_0^\infty (\Omega ;[0,\infty [)$ so that
  \[
  \| \chi _j A\chi _k\|_{\mathcal{ L}(L^2_\Phi ,L^2_\Phi )}\le
  C_{j,k}e^{-1/(C_{j,k}h)},
  \]
  for some $C_{j,k}>0$.  Using a partition of unity
  $\psi _j\in C_0^\infty (U_j;[0,1])$ on $\Omega $ subordinated to the
  covering, and the fact that the norm of $\psi _jA\psi _k$ is bounded
  from above by that of $\chi _jA\chi _k$, we conclude that
  \begin{equation}\label{pf.9}
    A=\mathcal{ O}(1)e^{(-F(x,y)+\Phi (x)-\Phi
      (y))/h}:L^2_{\mathrm{loc}}(\Omega )\to L^2_{\mathrm{loc}}(\Omega ),
  \end{equation}
  if $0<F(x,y)\in C(\Omega \times \Omega )$ satisfies
  \begin{equation}\label{pf.10}
    F(x,y)\le \sup_{j,k}1_{U_j}(x)1_{U_k}(y)/C_{j,k}.
  \end{equation}
  An example of such a function is given by
  \begin{equation}\label{pf.11}
    F(x,y)=\sum_{j,k}\psi _j(x)\psi _k(y)/C_{j,k}.
  \end{equation}
  Sometimes, we write (\ref{pf.9}) as
  \begin{equation}\label{pf.12}
    A=\mathcal{ O}(1)e^{-F(x,y)/h}:\ L^2_{\Phi,\mathrm{loc}}(\Omega )\to L^2_{\Phi ,\mathrm{loc}}(\Omega ), 
  \end{equation}
  which also has a direct meaning similar to (\ref{pf.5}),
  (\ref{pf.6}).

  \par It follows that
  \begin{equation}\label{pf.13}
    A=\mathcal{ O}(1):\ L^2_{\Phi _1,\mathrm{loc}}(\Omega )\to L^2_{\Phi _2,\mathrm{loc}}(\Omega ) 
  \end{equation} 
  and hence that~\ref{item:negl-3} holds, provided that we can find
  continuous functions $\Phi _1$, $\Phi _2$ on $\Omega $ such that
  $\Phi _2<\Phi <\Phi _1$ and
  \begin{equation}\label{pf.14}
    \Phi _2(x)>-F(x,y)+\Phi (x)-\Phi (y)+\Phi _1(y),\ (x,y)\in \Gamma ,
  \end{equation}
  where the last estimate is equivalent to
  \begin{equation}\label{pf.15}
    (\Phi (x)-\Phi _2(x))+(\Phi _1(y)-\Phi (y))<F(x,y)\hbox{ on }\Gamma .
  \end{equation}

  \par It suffices to find continuous functions $\Phi _j$ such that
  \begin{equation}\label{pf.16}
    0<(\Phi -\Phi _2)(x)<\frac{1}{2}\inf_{y\in \Gamma ^{-1}(x)} F(x,y),
  \end{equation}
  and
  \begin{equation}\label{pf.17}
    0<(\Phi_1 -\Phi)(y)<\frac{1}{2}\inf_{x\in \Gamma (y)} F(x,y),
  \end{equation}
  where $\Gamma $ is viewed as a relation $\Omega _1\to \Omega _2$. By
  the properness of the projections
  $\Gamma \ni (x,y)\mapsto x\in \Omega $,
  $\Gamma \ni (x,y)\mapsto y\in \Omega $ and the continuity of $F>0$
  the right hand sides of (\ref{pf.16}), (\ref{pf.17}) are locally
  bounded from below by constants $>0$. We can then construct
  $\Phi -\Phi _2>0$, $\Phi _1-\Phi >0 $ as in (\ref{pf.11}) with the
  difference that we use a partition of unity in $y$ or in $x$
  only. We have shown the implication~\ref{item:negl-2}
  $\Rightarrow$~\ref{item:negl-3}.

  \par Finally, if~\ref{item:negl-3} holds, then
  \[
  A=\mathcal{ O}(1):\, L^2_{\Phi ,\mathrm{loc}}(\Omega )\to
  L^2_{\Phi_2 ,\mathrm{loc}}(\Omega ),
  \]
  which implies~\ref{item:negl-1}.
\end{proof}
In particular, item~\ref{item:negl-2} above gives the
\begin{coro}
  \label{coro:adjoint}
  If $A$ is uniformly continuous and uniformly properly supported,
  then $A\equiv_\Phi 0$ if and only if $A^*\equiv_\Phi 0$.
\end{coro}

\subsection{The approximate Bergman projection}

Let $\Omega\subset\CM^n$ be an open subset, and assume that $\Phi$ is
real analytic and pointwise strictly plurisubharmonic on $\Omega$:
\begin{equation}
  \forall x_0\in\Omega, \exists m_0>0, \qquad 
  m_0\textup{Id} \leq ( \partial^2_{x_i,
    \bar{x}_j} \Phi(x_0) )_{i, j=1}^n.
\end{equation}
In this section we use the $\Brg$ quantization
(Definition~\ref{defi:brg-quantization}) to construct an
\emph{approximate Bergman projection} on $\Omega$, in the sense of
Proposition~\ref{prop:approx_proj} below.

We first work on germs of functions near a point $x_0\in\CM^n$; in
this case, the result directly follows from Theorem
\ref{theom}. Indeed, we apply the second assertion of Theorem
\ref{theom} to the operator $ \Op^{\w}_\h(b_{ \hbar}^{\w})$ whose
classical analytic symbol $b_{ \hbar}^{\w}$ is chosen equal to $1$. We
obtain a classical analytic symbol $a_\h$ near $(x_0,\bar x_0)$ such
that the following holds. \begin{prop}
  \label{prop:pi_x0} There exists $r>0$ and a neighbourhood $\Omega_0$
  of $x_0$ such that the operator $\Pi_{x_0}:=\Opbrg_r(a_\h)$ has the
  following properties.
  \begin{enumerate}
  \item \label{item:pi_x0:selfadjoint}
    $\Pi_{x_0}\equiv_\Phi \Pi_{x_0}^*$ on $\lphi(\Omega_0)$.
  \item \label{item:pi_x0:holomo} For all $u\in \lphi(\Omega_0)$,
    $\Pi_{x_0}u \in H_\Phi(\Omega_0):=\Hol(\Omega_0)\cap
    \lphi(\Omega_0)$.
  \item \label{item:pi_x0:repro} There exists
    $\Phi_2\in\Cinf(\Omega_0;\RM)$ with $\Phi_2<\Phi$ and such that
    for all $\Omega_2\Subset \Omega_0$, there exists $C>0$,
    independent of $\h$, with
    \begin{equation}
      \forall u\in H_\Phi(\Omega_0), \quad 
      \norm{\Pi_{x_0} u - u}_{L^2_{\Phi_2}(\Omega_2)} \leq C \norm{u}_{\lphi(\Omega_0)}.
    \end{equation}
  \end{enumerate}
\end{prop}
\begin{proof}
  Using the notation of Definition~\ref{defi:brg-quantization}, let
  $\tilde\Omega_0 := B((x_0,\bar x_0), \tilde r)$ and
  $\Omega_0:=B(x_0,r)$. We can write~\eqref{equ:kernel-brg} as in
  integral on $\CM^n$ by replacing the distribution kernel $k_\h(x,y)$
  of $\Opbrg(a_\h)$ by
  $1_{\abs{x-y}<r} 1_{(x,\bar y)\in \tilde \Omega_0} k_\h(x,y)$.
  Since $a_\h(x,\bar y)$ is holomorphic in $x$, this distribution
  kernel is locally holomorphic in $x$ for almost all $y$, which gives
  Item~\ref{item:pi_x0:holomo}.

  The symbol $a_\h$ given by Theorem~\ref{theom} is such that
  $\Opbrg(a_\h)\hnegl \Op^{\w}_\h(1)\hnegl\textup{Id}$, acting on
  $H_{\Phi,x_0}$.  Thus, $\Pi_{x_0}-\textup{Id} \hnegl 0$, which gives
  Item~\ref{item:pi_x0:repro} (up to choosing a smaller neighbourhood
  $\Omega_0$, if necessary).  Let $\Pi_{x_0}^*$ be the adjoint of
  $\Pi_{x_0}$, viewed as an operator on $\lphi(\CM^n)$.  For any
  $u\in \lphi(\Omega_0)$, we have
  $\Pi_{x_0}^* u = \Opbrg(\tilde a_\h)u$, where
  $\tilde{a}_\h(x,y) = \overline{a(y,x)}$. Hence
  Item~\ref{item:pi_x0:holomo} holds for $\Pi_{x_0}^*$ as
  well. Therefore, Item~\ref{item:pi_x0:repro} implies
  \begin{equation}
    \norm{\Pi_{x_0} \Pi_{x_0}^*u - \Pi_{x_0}^*u}_{L^2_{\Phi_2}(\Omega_2)}
    \leq C \norm{\Pi_{x_0}^*u}_{\lphi(\Omega_0)} \leq \tilde C
    \norm{u}_{\lphi(\Omega_0)},
  \end{equation}
  where the last inequality follows from the fact that all $\Brg$
  operators with bounded symbols are uniformly bounded in
  $\lphi(\Omega_0)$; this is a consequence
  of~\eqref{equ:brg-bon-contour} and the Schur test. In other words,
  if $\chi_0\in\Cinf_0(\CM^n)$ is equal to $1$ on $\Omega_0$ and
  $\chi_2\in\Cinf_0(\Omega_2)$, we have
  \begin{equation}
    \label{equ:brg-selfadjoint1}
    \chi_2 (\Pi_{x_0} - 1)\Pi_{x_0}^* \chi_0 \equiv_\Phi 0.
  \end{equation}
  The operator $\chi_2 (\Pi_{x_0} - 1)\Pi_{x_0}^* \chi_0$ is uniformly
  properly supported and uniformly continuous on $\lphi(\CM^n)$.
  Hence, by Corollary~\ref{coro:adjoint}, we may take the adjoint:
  \begin{equation}
    \label{equ:brg-selfadjoint2}
    \chi_0 \Pi_{x_0}(\Pi_{x_0}^* - 1) \chi_2\equiv_\Phi 0\,.
  \end{equation}
  Assume that $\chi_2=1$ on an open neighbourhood $\Omega_3$ of $x_0$
  and let $\chi_3$ be a bounded function with compact support in
  $\Omega_3$. Multiplying on both sides~\eqref{equ:brg-selfadjoint1}
  and~\eqref{equ:brg-selfadjoint2} by $\chi_3$, we get
  \[
  \chi_3 \Pi_{x_0}^* \chi_3 \equiv_\Phi \chi_3 \Pi_{x_0}\Pi_{x_0}^* \chi_3
  \]
  and
  \[
  \chi_3 \Pi_{x_0} \chi_3 \equiv_\Phi \chi_3 \Pi_{x_0}\Pi_{x_0}^* \chi_3\,,
  \]
  and hence
  \[
  \chi_3\Pi_{x_0}^* \chi_3 \equiv_\Phi \chi_3\Pi_{x_0} \chi_3\,.
  \]
  Up to replacing $\Omega_0$ by a slightly smaller open set
  $\Omega_0'\Subset\Omega_3$ (which does not impact
  Items~\ref{item:pi_x0:holomo} and~\ref{item:pi_x0:repro}), and
  letting and $\chi_3=1_{\Omega_0'}$, we get
  Item~\ref{item:pi_x0:selfadjoint}.

\end{proof}

Next we globalize the operator $\Pi_{x_0}$ observing that, because of
the uniqueness in Theorem~\ref{theom}, the formal analytic symbol
$\hat a_\h(x,y)\sim\sum_j a_j(x,y) \h^j$ associated with $a_\h$ is in
fact well defined in a neighbourhood
$\Omega^{(2)}\subset \Omega\times\overline{\Omega}$ of the
antidiagonal
\[
\adiag(\Omega\times\overline{\Omega}) := \{(x, \bar x); \quad
x\in\Omega\}.
\]
For each $x_0\in\Omega$, there exists a small ball $\Omega_0$ around
$x_0$ and a constant $C_{x_0}>0$ such that
\begin{equation}
  \sup_{\Omega_0\times\overline{\Omega_0}} \abs{a_j} \leq C_{x_0}^{j+1}j^j.
\end{equation}
Thus $\hat a_\h\in\hat S^0(\Omega^{(2)})$ in the sense of
Definition~\ref{defi:formal}. Moreover, by
Item~\ref{item:pi_x0:selfadjoint} of Proposition~\ref{prop:pi_x0}, we
have $a_j(x,y) = \overline{a_j(y,x)}$ for all $j$. Using a covering of
$\Omega$ by such balls, one can construct a smooth function
$\mathcal{C}=\mathcal{C}(x,y)\in\Cinf(\Omega^{(2)}; \RM^*_+)$ such
that $\mathcal{C}(x,y) = \mathcal{C}(y,x)$ and
\begin{equation}
  \label{equ:C(x,y)}
  \forall (x,y)\in\Omega^{(2)}; \quad \abs{a_j(x,y)} \leq
  \mathcal{C}(x,y)^{j+1}j^j.
\end{equation}
Now put
\[
a_\mathcal{C}(x,y;\h):= \sum_{j\geq 0} \theta(j \h
\mathcal{C}(x,y))a_j(x,y)\h^j,
\]
where $\theta\in\Cinf_0([0,1[\,; [0,1])$ is equal to $1$ on
$[0,\frac{1}{2}]$. Then $a_\mathcal{C}\in\Cinf(\Omega^{(2)})$ and
\begin{equation}
  \label{equ:ac-a}
  a_\mathcal{C} - a_\h = \O(e^{-1/\hat C\h}) \quad \text{ in }
  \Omega_0\times\overline{\Omega_0},
\end{equation}
where $\hat C>0$ depends on $x_0$, $\mathcal{C}$ and
$\theta$. Moreover, $a_\mathcal{C}$ is `exponentially close' to a good
classical analytic symbol, in that there exists a smooth function
$\mathcal{C}_1(x,y)>0$ such that
\begin{equation}
  a_\mathcal{C}(x,y)  - \sum_{0\leq j \leq \frac{1}{2\h \mathcal{C}(x,y)}} a_j(x,y)\h^j = \O(e^{{-1}/{\mathcal{C}_1(x,y)\h}})
\end{equation}
and
\begin{equation}
  \label{equ:dbar-a_C}
  \dbar_{x,y} a_\mathcal{C}(x,y) = \O(e^{{-1}/{\mathcal{C}_1(x,y)\h}}).
\end{equation}
Let $\chi\in\Cinf(\Omega\times\overline{\Omega}; \RM)$ satisfy
$\chi(x,\bar y) = \chi(y, \bar x)$, be supported in $\Omega^{(2)}$,
and equal to $1$ near $\adiag(\Omega\times\overline{\Omega})$. We
extend the $\Brg$ quantization by putting
\begin{equation}
  \label{equ:Pi_C}
  (\Pi_{\mathcal{C}, \chi} u) (x) = \frac{2^n}{(\pi\h)^n} \int_{\Omega}
  e^{\frac{2}{\h} (\psi(x, \bar{y})- \Phi(y))}
  a_{\mathcal{C}}(x,\bar{y})  u(y) \chi(x,\bar y)
  \det(\partial^2_{x,w} \psi)(x, \bar{y}) L(\DD y).
\end{equation}
\begin{prop}
  \label{prop:approx_proj}
  The operator $\Pi_{\mathcal{C}, \chi}$ has the following properties.
  \begin{enumerate}
  \item \emph{Continuity:}
    \label{item:pi_chi-contintuity}
    $\Pi_{\mathcal{C}, \chi}: L_{ \Phi,\loc}^{2}(\Omega) \to L_{
      \Phi,\loc}^{2}(\Omega)$
    is uniformly properly supported and uniformly continuous.
  \item \emph{Self-adjointness:}
    \label{item:pi_chi-selfadjoint}
    $\Pi_{\mathcal{C}, \chi} = \Pi_{\mathcal{C}, \chi}^*$.
  \item \label{item:localization}\emph{Exponential localization:} If
    $K\subset\Omega$ is closed, there exists
    $\Phi_2\in\mathscr{C}^0(\Omega;\RM)$, $\Phi_2\leq \Phi$ with
    \[
    \Phi_2 < \Phi \quad \text{ on } \Omega\setminus K
    \]
    such that
    \[
    \Pi_{\mathcal{C}, \chi} = \mathcal{O}(1): L_{ \Phi,\loc}^{2}(K)
    \to L_{ \Phi_2,\loc}^{2}(\Omega),
    \]
    where
    $L_{ \Phi,\loc}^{2}(K):=\{u\in L_{ \Phi,\loc}^{2}(\Omega); \quad
    \spt u \subset K\}$.
  \end{enumerate}
\end{prop}
\begin{proof}
  The fact that $\Pi_{\mathcal{C}, \chi}$ is properly supported holds
  if $\Omega^{(2)}$ is chosen close enough to the antidiagonal. In
  this case the projections on $x$ or $y$ of any closed subset of
  $\Omega^{(2)}$ will be proper.

  The uniform continuity of $\Pi_{\mathcal{C}, \chi}$, as in the proof
  of Proposition~\ref{prop:pi_x0} above, follows
  from~\eqref{equ:brg-bon-contour}. This gives
  Item~\ref{item:pi_chi-contintuity}.

  Item~\ref{item:pi_chi-selfadjoint} (selfadjointness) is deduced from
  the fact that
  $\overline{a_{\mathcal{C}}(x,y)} = a_{\mathcal{C}}(y,x)$.

  In order to prove Item~\ref{item:localization}, we remark that if
  $K_2\subset\Omega\setminus K$ is compact, then the distance $\delta$
  between $K$ and $K_2$ is positive. Hence, the distribution kernel of
  the restriction of
  $\Pi_{\mathcal{C}, \chi}: L_{ \Phi,\loc}^{2}(K) \to L_{
    \Phi,\loc}^{2}(K_2)$
  is of the form $1_{\abs{x-y}>\delta} 1_{K_2}(x)k_\h(x,y)1_{K}(y)$,
  where $k_\h$ is the original kernel of~\eqref{equ:Pi_C}. In view
  of~\eqref{equ:brg-bon-contour}, the norm of this restriction is
  $\mathcal{O}(e^{-c(K_2)/\h})$ for some $c(K_2)>0$. Using a partition
  of unity of $\Omega\setminus K$, we construct a function $\Phi_2$ as
  in the proof of Proposition~\ref{prop:phi1}, and we obtain
  Item~\ref{item:localization}.
\end{proof}
Note that $\Pi_{\mathcal{C}, \chi} u $ is no longer holomorphic since
the presence of $\mathcal{C}$ and $\chi$ destroys holomorphy, but we
see that $\h\dbar\Pi_{\mathcal{C}, \chi} u$ is `exponentially
small'. To formulate this, we define appropriate spaces.

\begin{defi}
  Let $\Phi_1\in\mathscr{C}^0(\Omega;\RM)$. We define
  \begin{equation}
    H^\loc_{\Phi,\Phi_1}(\Omega) := \{u\in L^2_{\Phi,\loc}(\Omega); 
    \quad \h\partial u \in L^ 2_{\Phi_1,\loc}(\Omega)\}\cdot     
  \end{equation}
\end{defi}
$H^\loc_{\Phi,\Phi_1}$ is a Fréchet space when equipped with the
natural semi-norms, and this space injects uniformly continuously into
$L^2_{\Phi,\loc}(\Omega)$.
\begin{prop}
  \label{prop:pi_chi-holom}
  There exists $\Phi_1\in\mathscr{C}^0(\Omega;\RM)$ with $\Phi_1<\Phi$
  such that
  \[
  \Pi_{\mathcal{C}, \chi} = \mathcal{O}(1): L_{ \Phi,\loc}^{2}(\Omega)
  \to H^\loc_{\Phi,\Phi_1}(\Omega).
  \]
\end{prop}
\begin{proof}
  Here the notation $\O(1)$ is used similarly to
  Definition~\ref{defi:unif-cont}.  Let $\Omega_2\Subset\Omega$. Since
  $\Pi_{\mathcal{C}, \chi}$ is properly supported, there exists
  $\Omega_1\Subset\Omega$ such that the support of the distribution
  kernel of $\Pi_{\mathcal{C}, \chi} 1_{\Omega_2}$ is contained in
  $\Omega_1\times \Omega_2$. Applying the $\dbar$ operator
  on~\eqref{equ:Pi_C}, we get the sum of two terms: one involving
  $\dbar_x a_{\mathcal{C}}(x,\bar y)$, which we estimate uniformly on
  $\Omega_1\times \Omega_2$ by~\eqref{equ:dbar-a_C}, and another term
  involving $\dbar_x \chi(x,\bar y)$. Since $\dbar_x \chi(x,\bar y)$
  is supported away from the anti-diagonal, this last term can be
  uniformly estimated as well by the good contour
  property~\eqref{equ:brg-bon-contour}. This finally gives
  \[
  \norm{\h\dbar \Pi_{\mathcal{C}, \chi} u}_{\lphi(\Omega_2)} \leq
  Ce^{-1/C\h} \norm{u}_{\lphi(\Omega_1)}.
  \]
  In other words, $\h\dbar \Pi_{\mathcal{C}, \chi} \equiv_\Phi 0$
  (Proposition~\ref{prop:negligible}). Hence there exists
  $\Phi_1<\Phi$ such that
  \[
  \h\dbar \Pi_{\mathcal{C}, \chi} = \mathcal{O}(1): L_{
    \Phi,\loc}^{2}(\Omega) \to L^2_{\Phi_1}(\Omega)\,.
  \]
  By Proposition~\ref{prop:approx_proj}, the operator
  $\Pi_{\mathcal{C}, \chi}$ is uniformly continuous:
  $\lphi(\Omega)\to \lphi(\Omega)$, which finishes the proof.
\end{proof}

We next turn to the reproducing property: if $u$ is holomorphic, or
exponentially close to holomorphic, then $\Pi_{\mathcal{C}, \chi} u$
must be exponentially close to $u$. We first deal with the case of a
holomorphic $u$.
\begin{lemm}
  \label{lemm:pi_C-reproduisant}
  There exists $\Phi_2\in\mathscr{C}^\infty(\Omega;\RM)$ with
  $\Phi_2<\Phi$ such that for all $\Omega_2\Subset \Omega$, there
  exists $\Omega_1\Subset\Omega$ and $C>0$, independent of $\h$, with
  \begin{equation}
    \forall u\in H_\Phi^\loc(\Omega), \quad 
    \norm{\Pi_{\mathcal{C}, \chi} u - u}_{L^2_{\Phi_2}(\Omega_2)} 
    \leq C \norm{u}_{L^2_\Phi(\Omega_1)}.
  \end{equation}
\end{lemm}
\begin{proof}
  Around any $x_0$ there is a ball $\Omega_0$ such that
  $1_{\Omega_0}\Pi_{\mathcal{C}, \chi}1_{\Omega_0} \equiv_\Phi
  1_{\Omega_0}\Pi_{x_0}1_{\Omega_0}$
  (see~\eqref{equ:ac-a}) where $\Pi_{x_0}$ is as in
  Proposition~\ref{prop:pi_x0}. By Item~\ref{item:pi_x0:repro} of that
  proposition, we have, for any $\Omega_2\Subset \Omega_0$,
  \begin{equation}
    \forall u\in H_\Phi(\Omega_0), \quad 
    \norm{ \Pi_{\mathcal{C}, \chi} u - u}_{L^2_{\Phi_2}(\Omega_2)} \leq C \norm{u}_{\lphi(\Omega_0)}\,.
  \end{equation}
  We may conclude by a partition of unity argument, as in the proof of
  Proposition~\ref{prop:phi1}.
\end{proof}
\begin{prop}
  If $\Phi_1\in\mathscr{C}^0(\Omega;\RM)$ satisfies $\Phi_1<\Phi$,
  then there exists $\Phi_2\in\mathscr{C}^0(\Omega;\RM)$ with
  $\Phi_2<\Phi$ such that
  \[
  \Pi_{\mathcal{C}, \chi} - 1 = \mathcal{O}(1):
  H^\loc_{\Phi,\Phi_1}(\Omega) \to L_{ \Phi_2,\loc}^{2}(\Omega).
  \]
\end{prop}
\begin{proof}
                                  
  Let $z_0\in \Omega $.
  \begin{lemm}
    \label{lemm:pf1}
    $\exists$ an open neighborhood $V\Subset \Omega $ of $z_0$ and
    $\Phi _0\in C^\infty (\CM^n;\RM)$ such that:
    \begin{equation}\label{pf-4.17.1}
      \Phi _0=\Phi \hbox{ in }V,
    \end{equation}
    \begin{equation}\label{pf-4.17.2}
      \nabla ^\alpha \Phi _0={\cal O}(1)\hbox{ on }\CM^n\hbox{ when
      }|\alpha |\ge 2,
    \end{equation}
    \begin{equation}\label{pf-4.17.3}
      \exists\, C>0\hbox{ such that }\partial _{\bar z}\partial _{z}\Phi
      _0\ge 1/C.
    \end{equation}
  \end{lemm}
  \begin{proof}
    Let $\Phi ^{(2)}$ be the Taylor polynomial of order 2 of $\Phi $
    at $z_0$, so that
    \begin{equation}\label{pf-4.17.4}
      \nabla ^\alpha (\Phi -\Phi ^{(2)})={\cal O}(|z-z_0|^{3-|\alpha |}),\
      |z-z_0|\hbox{ small },
    \end{equation}
    for $0\le |\alpha |\le 2$. Let
    $\chi \in C_0^\infty (B_{\CM^n}(0,1);[0,1])$ be equal to 1 on
    $B_{\CM^n}(0,1/2)$ and consider for $0<\epsilon \ll 1$:
    \begin{equation}\label{pf-4.17.5}
      \Phi _0(z)=\Phi _{0,\epsilon }(z)=\Phi ^{(2)}(z)+\chi (|z-z_0|/\epsilon
      )(\Phi -\Phi ^{(2)})(z).
    \end{equation}
    Then,
    \begin{equation}\label{pf-4.17.6}
      \Phi _0(z)=\begin{cases}
        \Phi (z)\hbox{ in }B(z_0,\epsilon /2),\\
        \Phi ^{(2)}(z)\hbox{ in }\CM^n\setminus B(z_0,\epsilon ),
      \end{cases}
    \end{equation}
    \begin{equation}\label{pf-4.17.7}
      \nabla ^\alpha \Phi _0=\nabla ^\alpha \Phi ^{(2)}+{\cal
        O}(|z-z_0|^{3-|\alpha |}),\ |\alpha |\le 2,
    \end{equation}
    and in particular
    \begin{equation}\label{pf-4.17.8}
      \partial _z\partial _{\bar{z}}\Phi _0=\partial _z\partial
      _{\bar{z}}\Phi ^{(2)}+{\cal O}(\epsilon )\ge 1/{\cal O}(1),
    \end{equation}
    when $\epsilon >0$ is small enough. The lemma follows with
    $V=B(x_0,\epsilon /2)$ for some $0<\epsilon \ll 1$.
  \end{proof}
  Let $W\Subset V$ be an open neighborhood of $z_0$ with smooth
  boundary and let $\chi _W\in C_0^\infty (V;[0,1])$ satisfy:
  \begin{equation}\label{pf-4.17.9}
    \chi _W>0 \hbox{ in }W,\ \mathrm{supp\,}\chi _{W}\subset \overline{W}.
  \end{equation}
  Then for $\delta >0$ small enough, the function
  $\Phi _\delta =\Phi _0-\delta \chi _W$ satisfies (\ref{pf-4.17.2}),
  (\ref{pf-4.17.3}) and
  \begin{equation}\label{pf-4.17.10}
    \Phi _\delta =\Phi \hbox{ in }V\setminus W,
  \end{equation}
  \begin{equation}\label{pf-4.17.11}
    \Phi -\delta \le \Phi _\delta <\Phi \hbox{ in }W.
  \end{equation}
  We choose $\delta >0$ small enough so that
  \begin{equation}\label{pf-4.17.12}
    \Phi _\delta >\Phi _1 \hbox{ in }\overline{V}.
  \end{equation}

  \par Let $\widetilde{\chi }\in C_0^\infty (V;[0,1])$ be equal to 1
  on $W$ and write
  \begin{equation}\label{pf-4.17.13}
    u=\widetilde{\chi }u+(1-\widetilde{\chi })u,\ u\in
    H^{\mathrm{loc}}_{\Phi ,\Phi _1}(\Omega ).
  \end{equation}
  Then
  \begin{equation}\label{pf-4.17.14}
    \h \overline{\partial }(\widetilde{\chi }u)=u\h \overline{\partial
    }\widetilde{\chi }+\widetilde{\chi }\h \overline{\partial }u\in
    L^2_{\Phi _\delta },
  \end{equation}
  and
  \begin{equation}\label{pf-4.17.15}
    \| \h \overline{\partial } (\widetilde{\chi }u)\|_{L^2_{\Phi _\delta
      }(\CM^n)}\le {\cal O}(1)\left( \h \| u\|_{L^2_\Phi (V)}+\|
      \h \overline{\partial }u \|_{L^2_{\Phi _1}(V)} \right).
  \end{equation}
  Moreover $\h \overline{\partial }(\widetilde{\chi }u)$ is
  $\overline{\partial }$-closed, so we can apply
  Appendix~\ref{app:dbar} (cf.\ (\ref{dbar.12}), (\ref{dbar.13})), to
  find $w\in L^2_{\Phi _\delta }$ such that
  \begin{equation}\label{pf-4.17.16}
    \h \overline{\partial }w=\h \overline{\partial }(\widetilde{\chi }u),
  \end{equation}
  \begin{equation}\label{pf-4.17.17}
    \| w\|_{L^2_{\Phi _\delta }(\CM^n)}\le {\cal O}(\h ^{-1/2})\left( \h \| u\|_{L^2_\Phi (V)}+\|
      \h \overline{\partial }u \|_{L^2_{\Phi _1}(V)} \right)
  \end{equation}

  \par Write
  \begin{equation}\label{pf-4.17.18}
    u=(\widetilde{\chi }u-w)+(1-\widetilde{\chi })u+w.
  \end{equation}
  Here $\widetilde{\chi }u-w\in H_\Phi (V)$,
  \[
  \| \widetilde{\chi }u-w\|_{H_\Phi (V)}\le {\cal O}(\h ^{-1/2})\left(
    \h ^{1/2}\| u\|_{L^2_\Phi (V)}+\| \h \overline{\partial }u
    \|_{L^2_{\Phi _1}(V)} \right) .
  \]
  We may assume without loss of generality (see also a comment below)
  that $\mathrm{supp\,}\chi $ is contained in a sufficiently small
  neighborhood of the diagonal, so that the restriction of
  $\Pi _{C,\chi }u$ to $W$ only depends on ${{u}_\vert}_{V}$. By
  Lemma~\ref{lemm:pi_C-reproduisant}, we get with $\delta >0$ small
  enough,
  \begin{equation}\label{pf-4.17.19}
    \| (\Pi _{C,\chi }-1)(\widetilde{\chi }u-w)\|_{L^2_{\Phi
        _\delta     }(W)}
    \le {\cal O}(\h ^{-1/2})\left( \h ^{1/2}\|
      u\|_{L^2_\Phi (V)}+\| \h \overline{\partial }u \|_{L^2_{\Phi _1}(V)}
    \right) .
  \end{equation}

  Let $\widetilde{W}\Subset W$ be another neighborhood of $z_0$ with
  smooth boundary and let $\widetilde{\Phi }_\delta \ge \Phi _\delta $
  be a new function with the same properties as $\Phi _\delta $ after
  replacing $W$ with $\widetilde{W}$. Then (\ref{pf-4.17.19}) still
  holds after replacing $L^2_{\Phi _\delta }(W)$ with
  $L^2_{\widetilde{\Phi } _\delta }(\widetilde{W})$. From
  (\ref{pf-4.17.17}) we get
  \begin{equation}\label{pf-4.17.20}
    \| (\Pi _{C,\chi }-1)w\|_{L^2_{\widetilde{\Phi }_\delta
      }(\widetilde{W})}\le {\cal O}(\h ^{-1/2})\left( \h \|
      u\|_{L^2_\Phi (V)}+\| \h \overline{\partial }u \|_{L^2_{\Phi _1}(V)}
    \right) 
  \end{equation}
  when $\widetilde{\Phi }_\delta $ is close enough to $\Phi $ but
  still $<\Phi $ in $\widetilde{W}$.

  Since $\Pi _{C,\chi }$ enjoys the pseudolocal property
  (item~\ref{item:localization} of Proposition~\ref{prop:approx_proj})
  we get the same estimate for
  $(\Pi _{C,\chi }-1)(1-\widetilde{\chi })u$.

  \par Thus we have found a continuous function
  $\widetilde{\Phi }_\delta \le \Phi $ in $V$ with
  $\widetilde{\Phi }_\delta <\Phi $ in $\widetilde{W}$, such that
  \begin{equation}\label{pf-4.17.21}
    \| (\Pi _{C,\chi }-1)u\|_{L^2_{\widetilde{\Phi }_\delta
      }(\widetilde{W})}\le {\cal O}(\h ^{-1/2})\left( \h ^{1/2}\|
      u\|_{L^2_\Phi (V)}+\| \h \overline{\partial }u \|_{L^2_{\Phi _1}(V)}
    \right) .
  \end{equation}
  After a slight shrinking of $\widetilde{W}$ and increase of
  $\widetilde{\Phi }_\delta $ we can eliminate the factor
  ${\cal O}(\h ^{-1/2})$.

  Without the shrinking of the support of $\chi $, we get the same
  estimate after replacing $V$ with some larger domain
  $\Subset \Omega $. Varying $z_0$, we get the proposition by means of
  a partition of unity.
\end{proof}

\subsection{Uniqueness of the approximate Bergman projection}
\label{sec:uniq-appr-bergm}

In the previous paragraphs, we have constructed an operator
$\Pi_0=\Pi_{\mathcal{C}, \chi}$ with the following properties:

\begin{enumerate}%[label=(P\arabic*)]
\item \label{item:prop-1} $\Pi_0$ is uniformly continuous:
  $ L_{ \Phi,\loc}^{2}(\Omega) \to H^\loc_{\Phi,\Phi_1}(\Omega)$ for
  some $\Phi_1\in\mathscr{C}^0(\Omega;\RM)$ with $\Phi_1<\Phi$.
\item $\Pi_0$ is uniformly properly supported.
\item \label{item:prop-3} $\Pi_0\equiv_\Phi \Pi_0^*$ (see
  Definition~\ref{defi:negligible}).
\item \label{item:prop-4} If $\Phi_1\in\mathscr{C}^0(\Omega;\RM)$
  satisfies $\Phi_1<\Phi$, then there exists
  $\Phi_2\in\mathscr{C}^0(\Omega;\RM)$ with $\Phi_2<\Phi$ such that
  \[
  \Pi_0 - 1 = \mathcal{O}(1): H^\loc_{\Phi,\Phi_1}(\Omega) \to L_{
    \Phi_2,\loc}^{2}(\Omega).
  \]
\end{enumerate}

\begin{prop}
  \label{prop:uniqueness}
  Assume that $\Pi_0$ and $\tilde\Pi$ satisfy
  \ref{item:prop-1}--\ref{item:prop-4}. Then $\tilde\Pi\equiv_\Phi\Pi_0$.
\end{prop}
\begin{proof}
  Using~\ref{item:prop-3} and~\ref{item:prop-4} for $\tilde{\Pi}$, we
  see that $\tilde\Pi^*$ satisfies~\ref{item:prop-4}. Since $\Pi_0$
  satisfies~\ref{item:prop-1}, we get $(\tilde\Pi^*-1)\Pi_0\equiv_\Phi 0$,
  \emph{i.e.}
  \begin{equation}
    \label{equ:pi-pi}
    \tilde\Pi^* \Pi_0 \equiv_\Phi \Pi_0.
  \end{equation}
  By Corollary~\ref{coro:adjoint}, we get
  $\Pi_0\equiv_\Phi \Pi_0^* \equiv_\Phi \Pi_0\tilde\Pi$. By~\eqref{equ:pi-pi}
  with $\Pi_0$ and $\tilde\Pi$ exchanged, we get
  $\Pi_0^* \tilde \Pi \equiv_\Phi \tilde\Pi$ and hence
  $\Pi_0\equiv_\Phi\tilde\Pi$ as claimed.
\end{proof}

\section{The Bergman projection on \texorpdfstring{$\CM^n$}{Cn}}
\label{sec:bcn}

Let $\Phi :\CM^n\to \RM$ satisfy
\begin{equation}\label{bcn.1}\begin{split} &\Phi \hbox{ has a
      holomorphic extension to a tubular}\\ &\hbox{neighborhood
    }T\hbox{ in } \widetilde\CM^n, \hbox{the complexification of }\CM^n \,.
  \end{split}
\end{equation}
We use `$\Phi $' also to denote the extension.  Also assume that
\begin{equation}\label{bcn.2}
  \nabla ^2\Phi \hbox{ is bounded in }T,
\end{equation}
\begin{equation}\label{bcn.3}
  \partial_{\bar{z}}\partial_z\Phi \geq 1/C\hbox{ on }\CM^n\hbox{,
    for some constant }C>0.
\end{equation}

\par Examining the proofs, we see that the formal analytic symbol
$\widehat{a}_\hbar (x,y)\sim \sum_j a_j(x,y)\hbar ^j$ in the proof of
Proposition~\ref{prop:pi_x0} is well defined in a tubular neighborhood
$\Omega _1$ of the antidiagonal, $\mathrm{adiag\,}(\CM^n)$ and
satisfies the estimates on~\eqref{equ:C(x,y)} with $C(x,y)=C$
independent of $(x,y)$. Correspondingly, we define $a_C$ simply by
\begin{equation}\label{bcn.4}
  a_C(x,y;\h)=\sum_{0\le j\le 1/(2C\h)}a_j(x,y)\h^j
\end{equation} 
in $\Omega _1$ and $a_C$ is holomorphic. We can define
$\Pi _{c,\chi }u$ as in~\eqref{equ:Pi_C} with $\chi $ of the form
$\chi (x-y)$, where $\chi \in C_0^\infty (\CM^n)$ is equal to 1 near 0
and with support in a small neighborhood of 0. Choosing $\chi $ real
and even; $\chi (-y)=\chi (y) $, we get
\begin{prop}
  \label{prop:bcn1}
  $\exists$ $C_1>0$ such that with $\Phi _1=\Phi -1/C_1$,
  \begin{enumerate}[label=\roman*)]
  \item \label{item:piC-1}
    $\Pi _{C,\chi } ={\cal O}(1):\, L^2_{\Phi }(\CM^n)\to L^2_{\Phi
    }(\CM^n)$ is selfadjoint,
  \item \label{item:piC-2}
    $\Pi _{C,\chi }={\cal O}(1):\, L^2_\Phi (\CM^n)\to H_{\Phi ,\Phi
      _1}(\CM^n)$,
  \item \label{item:piC-3}
    $\Pi _{C,\chi }-1={\cal O}(1):\, H_{\Phi }(\CM^n)\to L^2_{\Phi
      _1}(\CM^n)$.
  \end{enumerate}
\end{prop}
Notice that~\ref{item:piC-2} amounts to~\ref{item:piC-1} and
\[
\h\overline{\partial }\Pi _{c,\chi }={\cal O}(e^{-1/(C_1\h)}):\,
L^2_\Phi \to L^2_\Phi ,
\]
where we omit to write out `$\CM^n$' when there is no risk of
confusion.

\par Since $\h\overline{\partial }\Pi _{C,\chi }u$ is
$\overline{\partial }$-closed for every $u\in L^2_\Phi $, we can
decompose:
\begin{equation}\label{bcn.5}
  \Pi _{C,\chi }=(\Pi _{C,\chi }-R)+R=:\widetilde{\Pi }+R,
\end{equation}
where
\begin{equation}\label{bcn.6}\begin{split}
    R=(\h\overline{\partial })^{\Phi ,*} (\Box^{(1)}_\Phi
    )^{-1}\h\overline{\partial }\Pi _{C,\chi }&={\cal
      O}(\h^{-1/2})e^{-1/(C_1\h)}:\, L^2_\Phi \to
    L^2_\Phi ,\\
    \widetilde{\Pi }&={\cal O}(1):\, L^2_\Phi \to H_\Phi .
  \end{split}
\end{equation}
Here the box operator is defined in Section \ref{app:dbar} and as
there we let the exponent $(\Phi ,*)$ indicate that we take adjoints
in the $L_\Phi ^2$-spaces of scalar or form-valued functions.

\par Let $\Pi $ be the orthogonal projection in $L^2_\Phi (\CM^n) $
onto $H_\Phi (\CM^n)$.
\begin{theo}
  \label{theo:bcn2}
  We have
  \begin{equation}\label{bcn.7}
    \Pi -\Pi _{C,\chi }={\cal O}(\h^{-1/2})e^{-1/(C_1\h)}):\, L^2_\Phi \to
    L^2_\Phi .
  \end{equation}
\end{theo}
\begin{proof}
  We have
  \begin{multline}
    \Pi \, \Pi _{C,\chi }=\Pi\, \widetilde{\Pi }+\Pi R =
    \widetilde{\Pi }+\Pi\, R\\
    =\Pi _{C,\chi }-(1-\Pi )R=\Pi _{C,\chi }+{\cal
      O}(\h^{-1/2})e^{-1/(C_1\h)}:\, L^2_\Phi \to L^2_\Phi .
  \end{multline}
  Taking the adjoints of this relation and using that
  $\Pi ^*_{C,\chi }=\Pi _{C,\chi }$, $\Pi ^*=\Pi $, we get
  \begin{equation}\label{bcn.8}
    \Pi _{C,\chi} =\Pi _{C,\chi }\Pi +{\cal O}(\h^{-1/2})e^{-1/C_1\h}.
  \end{equation}
  By~\ref{item:piC-3} in Proposition \ref{prop:bcn1}, we have
  \begin{equation}\label{bcn.9}
    \Pi _{C,\chi }\Pi =\Pi +{\cal O}(1)e^{-1/(C_1\h)}.
  \end{equation}
  (\ref{bcn.7}) follows from (\ref{bcn.8}) and (\ref{bcn.9}).
\end{proof}
\par We next prove a corresponding result on the level of distribution
kernels. Let $\widetilde{k}(x,y)e^{-2\Phi (y)/\h}$ denote the
distribution kernel of $\Pi $. For any $1\leq \nu \leq n$, since
$\h\partial_{\bar z_\nu} \Pi =0$ we know that
$\partial_{\bar z_\nu} k=0$. Taking the adjoint of this relation, we
get $\Pi (\h\partial_{\bar z_\nu})^*=0 $ as an operator
$C_0^\infty (\CM^n)\to {\cal D}'(\CM^n)$. Here
$(\h\partial_{\bar z_\nu})^*=-\h\partial_{z_\nu} +
2\partial_{z_\nu}\Phi $
is the adjoint of $\h\partial_{\bar z_\nu} $ in for the inner product
of $L^2_\Phi $, so we get for every $u\in C_0^\infty (\CM^n)$:
\[
\begin{split}
  0&=\int \widetilde{k}(x,y)e^{-2\Phi (y)/\h}(-\h\partial
  _{y_\nu}+2\partial
  _{y_\nu}\Phi (y))u(y)L(\DD y)\\
  &= \int \widetilde{k}(x,y)(-\h\partial _{y_\nu})(e^{-2\Phi (y)/\h}u(y))L(\DD y)\\
  &=\int \h\partial _{y_\nu}(\widetilde{k}(x,y)) e^{-2\Phi
    (y)/\h}u(y)L(\DD y).
\end{split}
\]
It follows that $\partial _{y_\nu}\widetilde{k}(x,y)=0$, so we have
the elliptic 1st order system for $\widetilde{k}$:
\[
\overline{\partial }_{x_\nu} \widetilde{k}(x,y)=0,\ \partial
_{y_\nu}\widetilde{k}(x,y)=0.
\]
From the ellipticity, we conclude that $\widetilde{k}(x,y)$ is a
smooth function, holomorphic in $x$ and anti-holomorphic in $y$. Hence
$\widetilde{k}(x,y)=k(x,\bar{y})$ where $k(x,y)$ if holomorphic on
$\CM^{2n}$. For more details, see~\cite{co-hi-sjo-18}.

\par Recall that
\begin{equation}\label{bcn.10}
  \begin{split}
    \Phi (y)&=\Phi (y_0)+2\Re (\partial _y\Phi (y_0)\cdot
    (y-y_0))+{\cal
      O}(|y-y_0|^2)\\
    &=\Phi (y_0)+2\Re (\partial _{\bar{y}}\Phi (y_0)\cdot
    (\overline{y-y_0}))+{\cal O}(|y-y_0|^2).
  \end{split}
\end{equation}
Let $f\in C_0^\infty (\CM^n)$ be a radial function with
$\int f(y)L(\DD y)=1$ and put
\begin{equation}\label{bcn.11}
  e_{x_0}(x)=\h^{-n}f\left(\frac{x-x_0}{\h^{1/2}} \right)
  e^{\frac{1}{\h}(2\Phi (x)-\Phi (x_0)-2\partial _{\bar{x}}\Phi
    (x_0)\cdot \overline{(x-x_0)})}.
\end{equation}
Then
\[
|e_{x_0}(x)|=\h^{-n}\left| f\left(\frac{x-x_0}{\h^{1/2}} \right)
\right| e^{\frac{1}{\h}(\Phi (x)+{\cal O}(|x-x_0|^2))},
\]
so
\begin{equation}\label{bcn.12}
  \| e_{x_0}\|^2_{L^2_\Phi }\asymp \h^{-n}.
\end{equation}
\begin{lemm}
  \label{bcn3}
  For $x_0,\, y_0\in \CM^n$, we have
  \begin{equation}\label{bcn.13}
    \pscal{\Pi e_{y_0}}{e_{x_0}}_{L^2_\Phi }= k(x_0,\bar{y}_0)e^{-\frac{1}{\h}(\Phi
      (x_0)+\Phi (y_0))}.
  \end{equation}
\end{lemm}
\begin{proof}
  \begin{multline}
    e^{\frac{1}{\h}(\Phi (x_0)+\Phi (y_0))}\pscal{\Pi e_{y_0}}{e_{x_0}} _{L^2_\Phi }=\\
    \iint \h^{-n}\overline{f\left(\frac{x-x_0}{\h^{1/2}}
      \right)}e^{-\frac{2}{\h}\partial _x\Phi (x_0)\cdot
      (x-x_0)} k(x,\bar{y})\times \\
    \h^{-n}f\left(\frac{y-y_0}{\h^{1/2}}
    \right)e^{-\frac{2}{\h}\partial _{\bar{y}}\Phi (y_0)\cdot
      \overline{(y-y_0) }} L(\DD x)L(\DD y).
  \end{multline}
  Applying the spherical mean-value property for holomorphic and
  anti-holo\-morphic functions to the $x$-integral and $y$-integral
  respectively, this boils down to $\widetilde{k}(x_0,y_0)$.
\end{proof}
\begin{remark}
  In \cite[Section 3]{co-hi-sjo-18} a somewhat similar argument is
  given to estimate a distribution kernel in the metaplectic framework
  and with the spherical mean-value property replaced by the use of
  the reproducing kernel (known exactly in that case).
\end{remark}

\begin{theo}\label{bcn4}
  Let $\Omega _1$ and $a_C$ be as in and around (\ref{bcn.4}). Let
  $k(x,\bar{y};\h)e^{-2\Phi (y)/\h}$ be the distribution kernel of
  $\Pi $. There exists a constant $C_2>0$ such that
  \begin{equation}\label{bcn.14}
    \left| 
      e^{-(\Phi (x)+\Phi (y))/\h}\left(
        k(x,\bar{y};\h)-(1_{\Omega _1}a_C)(x,\bar{y};\h) 
        e^{\psi (x,\bar{y})/\h} \right)
    \right| 
    \le {\cal O}(1)e^{-\frac{1}{C_2\h}},
  \end{equation}
  uniformly on $\CM^n\times \CM^n$.
\end{theo}
\begin{proof}
  For $(x_0,y_0)$ in a small tubular neighborhood of the diagonal, we
  have $\chi (x-y)=1$ in a small ball of fixed radius around
  $(x_0,y_0)$ and by the proof of Lemma \ref{bcn3}, we get
  \begin{equation}\label{bcn.16}
    \pscal{\Pi _{C,\chi }e_{y_0}}{e_{x_0}}_{L^2_{\Phi}} = 
    a_C(x_0,\bar{y}_0;\h)e^{(2\Re \psi (x_0,\bar{y}_0) - 
      (\Phi(x_0)+\Phi (y_0)))/\h}.
  \end{equation}

  Recall here that
  \[
  2 \Re \psi (x,\bar{y})-\Phi (x)-\Phi (y)\asymp -|x-y|^2
  \]
  so the right hand side of (\ref{bcn.16}) is exponentially decreasing
  outside any tubular neighborhood of
  $\mathrm{diag\,}(\CM^n\times \CM^n)$. Using this fact, we get by
  direct estimates that
  \begin{equation}\label{bcn.17}
    \pscal{\Pi _{C,\chi }e_{y_0}}{e_{x_0}}
    ={\cal O}(1)e^{-\frac{1}{C_1\h}},
  \end{equation}
  for $(x_0,y_0)$ outside any fixed tubular neighborhood of
  $\mathrm{diag\,}(\CM^n\times \CM^n)$. Thus,
  \begin{equation}\label{bcn.18}
    \pscal{\Pi _{C,\chi }e_{y_0}}{e_{x_0}}_{L^2_{\Phi
      }}=1_{\Omega _1}(x_0,\bar{y}_0)a_C(x_0,\bar{y}_0;\h)
    e^{(2\Re \psi (x_0,\bar{y}_0)-(\Phi
      (x_0)+\Phi (y_0)))/\h}+{\cal O}(1)e^{-\frac{1}{C_1\h}},
  \end{equation}
  where $\Omega _1$ is any small tubular neighborhood of the diagonal
  and $C_1=C_1(\Omega )>0$.

  The theorem, now follows from (\ref{bcn.18}), (\ref{bcn.13}) and the
  fact that (\ref{bcn.7}) provides us with the estimate,
  \[
  \pscal{ (\Pi -\Pi _{C,\chi })e_{y_0}}{e_{x_0}}_{L^2_\Phi }={\cal
    O}(\h^{-1/2})e^{-\frac{1}{C_1\h}}\| e_{x_0}\|_{L^2_\Phi } \|
  e_{y_0}\|_{L^2_\Phi } = {\cal O}(\h^{-2n-1/2})e^{-\frac{1}{C_1\h}},
  \]
  with some new constant, that can be further increased to absorb the
  power of $\h$.
\end{proof}

\section{The Bergman projection for line bundles}
\label{sec:line-bundle}

In this section we consider a compact complex manifold $X$, of complex
dimension $n$, and two holomorphic line bundles $L$ and $E$ over
$X$. Both $L$ and $E$ are equipped with Hermitian metrics, denoted
respectively by $g_L$ and $g_E$, giving rise to a metric
$g_L^k\otimes g_E$ on the tensor product $F_k:=L^k \otimes E$,
$k\in\NM^*$. We assume that $g_L$ has strictly positive
curvature. Then $i/2$ times the curvature of $L$, which is a closed
2-form whose cohomology class is the Chern class $2\pi c_1(g_L)$, is a
Kähler form, and therefore induces a volume form $\omega_n$ on $X$,
and hence a scalar product $\pscal{\cdot}{\cdot}_k$ on the space of
sections of $F_k$. Notice that $F_k$ is positive if $k$ is large
enough. The orthogonal projection $\Pi_k$ from $L^2(X,F_k)$ onto
$\mathcal{H}^0(X,F_k)$, the subspace of holomorphic sections, is
called the associated \emph{Bergman projection}. Its distribution
kernel is a smooth section $K(\cdot,\cdot;k)$ of the external tensor
product $F_k \boxtimes F_k^*$ over $X\times X$ defined by
\[
\Pi_k u (x) = \int_X K(x,y;k) u(y) \omega_n(\DD y).
\]
(Recall that
$F_k \boxtimes F_k^* = \pi_1^* (F_k) \otimes \pi_2^*(F_k^*)$, where
$\pi_j$, $j=1,2$ are the coordinate projection maps $X\times X \to X$.
Thus $F_k \boxtimes F_k^*$ is the line bundle over $X\times X$ whose
fiber over $(x,y)$ is the space of linear maps from $F_k(y)$ to
$F_k(x)$.)  If $f_1,\dots,f_{N_k}$ is an orthonormal basis of
$\mathcal{H}^0(X,F_k)$ then the formula
$\Pi_k u = \sum_{j=1}^{N_k} \pscal{u}{f_j}f_j$ gives
\begin{equation}
  \label{equ:K_basis}
  K(x,y;k) = \sum_{j=1}^{N_k} f_j(x;k)\pscal{\cdot}{f_j(y;k)}_{F_k(y)}.
\end{equation}

We now fix a point $x_0\in X$ and use a trivializing holomorphic
section $s_L$ of $L$ above a neighborhood $\Omega_0$ of $x_0$ (which
me may identify with an open ball around $0\in\CM^n$) to define the
local real-valued analytic function $ \Phi_{x_0}$ such that the
Hermitian norm of $s_L$ is given, for $x\in \Omega_0$, by
\begin{equation}
  \label{equ:Phi0}
  \abs{s_L(x)}_L = e^{-\Phi_{x_0}(x)}.
\end{equation}
The corresponding Kähler form is
$\omega_L=i\partial\dbar\Phi_{x_0}$. We define similarly the section
$s_E$, and
\begin{equation}
  \label{equ:trivialization-sk}
  s_k:=s_L^k\otimes s_E \in \mathcal{H}^0(X,F_k).
\end{equation}
Notice that $\abs{s_k(x)}_{F_k} = e^{-k\Phi_{x_0}(x)}G_{x_0}(x)$ for
some non-vanishing analytic function $G_{x_0}=\abs{s_E}_E$, and hence,
if a local section of $F_k$ has the form $\tilde u = u s_k$, then
\[
\norm{\tilde u}^2_{k} = \int_{\Omega_0}\abs{u}^2
e^{-2k\Phi_{x_0}}G_{x_0}^2 \omega_n.
\]
Similarly, if $y_0\in X$, we construct a trivializing section $t_k$ of
$F_k$ on a neighborhood $V_0$ of $y_0$, with
$\abs{t_k(y)}_{F_k} = e^{-k\Phi_{y_0}(y)}G_{y_0}(y)$, and we can write
\begin{equation}
  \label{equ:kernel-fibre}
  K(x,y;k) = b(x,\bar y;k) s_k(x)\otimes t_k(y)^*,
\end{equation}
for $(x,y)\in \Omega_0\times V_0$, where $t_k(y)^*$ denotes the
adjoint map $\pscal{\cdot}{t_k(y)}$. Then from~\eqref{equ:K_basis} we
see that $b(x,y;k)$ is holomorphic both in $x$ and $y$, and
\begin{equation}
  \label{equ:norm-K}
  \abs{K(x,y;k)}_{F_{k,x}\otimes F_{k,y}^*}  
  = e^{-k(\Phi_{x_0}(x) + \Phi_{y_0}(y))} G_{x_0}(x) G_{y_0}(y)\abs{b(x,\bar y;k)}.
\end{equation}
On $V_0$, we define the `local Bergman projection'
$\tilde \Pi_k=\tilde\Pi_{k,x_0,y_0}$ by
\[
\forall u \in \lphi(V_0,G_{y_0}^2\omega_n), \qquad \Pi_k (u t_k) =
(\tilde \Pi_k u) s_k,
\]
which means that
\begin{equation}
  \label{equ:Pi-tilde}
  \tilde \Pi_k u (x) = \int_{\Omega_0} b(x,\bar y) u(y) 
  e^{-2k\Phi(y)} G^2(y) \omega_n(\DD y) = 
  \pscal{u G_{y_0}^2}{B_x}_{\lphi(\Omega_0,\omega_n)}
\end{equation}
with $B_x(y;k):=\overline{b(x, \bar y;k)}$.

\begin{theo}
  \label{theo:line-bundle}
  Assume that $g_L$ and $g_E$ are real-analytic (and $g_L$ has
  strictly positive curvature). Then the following estimates hold:
  \begin{enumerate}
  \item \label{item:off-diag} If $x_0\neq y_0$ then there exists $C>0$
    such that, uniformly in a neighborhood $\Omega_0\times V_0$ of
    $(x_0,y_0)$,
    \[
    \abs{K(x_,y;k)}_{F_{k,x}\otimes F_{k,y}^*} \leq Ce^{-\frac{k}{C}}
    .
    \]
    Equivalently,
    \[
    e^{-k(\Phi_{x_0}(x) + \Phi_{y_0}(y))}\abs{b(x,\bar y;k)} =
    \mathcal{O}(e^{-\frac{k}{C}}) .
    \]
  \item \label{item:near-diag} There exists $C>0$ such that, for any
    $x_0\in X$, there exists a neighborhood $\Omega_0$ of $x_0$, and a
    classical analytic symbol $a$ on
    $\Omega_0\times \overline{\Omega_0}$, such that, for all
    $(x,y)\in \Omega_0\times \Omega_0$,
    \[
    e^{-k(\Phi(x)+\Phi(y))}\abs{b_k(x,\bar y) - \frac{(2k)^n}{\pi^n}
      {a(x,\bar y;k^{-1})}e^{2k\psi(x,\bar y)}} \leq C
    e^{-\frac{k}{C}},
    \]
    where $\Phi=\Phi_{x_0}$ is defined in~\eqref{equ:Phi0}, and $\psi$
    is its polarized form~\eqref{equ:bar-psi}.
%% TODO  in the second case, b_k is not explicitly defined, but ok...
  \end{enumerate}
\end{theo}
\begin{proof}
  We first treat the case where the bundle $E$ is trivial.  The
  strategy is the same as in Section~\ref{sec:bcn}.  Consider a
  trivialization of $F_k$ in a neighborhood $\Omega_0$ of $x_0$, as
  above.  As in Proposition~\ref{prop:approx_proj}, we construct a
  classical analytic symbol $a_\h$ and the approximate Bergman
  projection $\Pi_{0}=\Pi_{\mathcal{C},\chi}$ obtained by smooth
  cut-off of $\Opbrg_r(a_\h)$ (see~\eqref{equ:Pi_C}), acting on
  $\lphi(\Omega_0)$. We now let $\h=1/k$ and define
  $\hat\Pi_{0}:L^2_{\textup{comp}}(\Omega_0;F_k)\to L^2(\Omega_0;F_k)$
  by
  \[
  \hat\Pi_{0} (u s_k) = (\Pi_{0} u) s_k.
  \]
  One can find a finite cover of $X$ by open sets $\Omega_j$, with
  $x_j\in \Omega_j$, on which the corresponding operator
  $\hat \Pi_{j}$ is defined as above. Let
  $\chi_j\in\Cinf_0(\Omega_j;\RM^+)$ be such that $\sum_j \chi_j = 1$
  on $X$, let $\tilde\chi_j\in\Cinf_0(\Omega_j)$ be equal to $1$ on a
  neighborhood of the support of $\chi_j$, and define
  \begin{equation}
    \label{equ:Pi-fibre}
    \hat\Pi := \sum_{j} \tilde\chi_j \hat \Pi_{j} \chi_j \quad
    = \O(1) : L^2(X;F_k) \to L^2(X; F_k).
  \end{equation}
  Let $\Lambda^{(p,q)}\to X$ be the vector bundle of $(p,q)$-forms on
  the tangent space of $X$, equipped with the metric induced from
  $\omega_n$ on $X$. Let $\dbar_k$ be the usual Dolbeault operator,
  mapping sections of $\Lambda^{(0,q)} \otimes F_k$ to sections of
  $\Lambda^{(0,q+1)} \otimes F_k$.  The following analogue of
  Proposition~\ref{prop:bcn1} holds.
  \begin{prop}
    \label{prop:fibre2}
    $\exists$ $C_1>0$ such that
    \begin{enumerate}[label=\roman*)]
    \item \label{item:propfibre-autoadj}
      $\hat\Pi^* - \hat\Pi = {\cal O}(e^{-kC_1}):\, L^2(X;F_k)\to
      L^2(X;F_k)$,
    \item \label{item:propfibre-holo}
      $k^{-1}\dbar_k\hat\Pi={\cal O}(e^{-kC_1}):\, L^2(X;F_k)\to
      L^2(X;\Lambda^{(0,1)} \otimes F_k)$ ,
    \item \label{item:propfibre-repro}
      $\hat\Pi-1={\cal O}(e^{-kC_1}):\, \mathcal{H}^0(X;F_k)\to
      L^2(X;F_k)$.
    \end{enumerate}
  \end{prop}
  \begin{proof} Let us denote by $\equiv$ the equality of operators
    modulo an exponentially small term ${\cal O}(e^{-kC})$, for some
    $C>0$, in the norm of $L^2(X;F_k)$.  We use the exponential
    locality property of Proposition~\ref{prop:approx_proj} which
    implies that
    \begin{equation}
      \label{equ:Pi_localized}
      \forall \chi_{1}, \chi_2 \in \Cinf_0(\Omega_j) \text{ with disjoint supports, } \chi_1 \hat\Pi_j \chi_2 \equiv 0.
    \end{equation}
    This gives, for all $m$,
    \begin{equation}
      \label{equ:chi-j-m}
      \chi_m \tilde\chi_j  \hat \Pi_{j} \chi_j 
      \equiv \chi_m \tilde\chi_j  \hat \Pi_{j} \chi_j\tilde\chi_m.
    \end{equation}
    Let $\Omega_{j,m}\Subset \Omega_j\cap \Omega_m$; the restricted
    operator on $L^2_\Phi(\Omega_{i,j})$,
    $\textup{1}_{\Omega_{j,m}}\hat\Pi_j\textup{1}_{\Omega_{j,m}}$ is
    an approximate Bergman projection in the sense of
    Section~\ref{sec:uniq-appr-bergm}, and so is
    $\textup{1}_{\Omega_{j,m}}\hat\Pi_m\textup{1}_{\Omega_{j,m}}$.
    Hence, by uniqueness (Proposition~\ref{prop:uniqueness}), we have
    \[
    \textup{1}_{\Omega_{j,m}}\hat\Pi_j\textup{1}_{\Omega_{j,m}} \equiv
    \textup{1}_{\Omega_{j,m}}\hat\Pi_m\textup{1}_{\Omega_{j,m}}.
    \]
    Hence we have
    $\tilde\chi_j \chi_m \hat \Pi_{j} \chi_j\tilde\chi_m \equiv
    \tilde\chi_j \chi_m \hat \Pi_{m} \chi_j\tilde\chi_m$
    which, in view of~\eqref{equ:chi-j-m}, gives
    \[
    \chi_m \tilde\chi_j \hat \Pi_{j} \chi_j \equiv \chi_m \hat \Pi_{m}
    \chi_j \tilde\chi_m.
    \]
    Thus from item~\ref{item:pi_chi-selfadjoint} of
    Proposition~\ref{prop:approx_proj} we have
    $\hat\Pi_m^* \equiv \hat\Pi_m$, and hence
    \[
    \hat\Pi = \sum_{j,m} \chi_m \tilde\chi_j \hat \Pi_{j} \chi_j
    \equiv \sum_{j,m} \chi_m \hat \Pi_{m} \tilde\chi_m \chi_j = \sum_m
    \chi_m \hat \Pi_{m} \tilde\chi_m \equiv \hat\Pi^*,
    \]
    which shows item~\ref{item:propfibre-autoadj}.
    
    Applying the Dolbeault operator $\dbar_{k}$, we obtain
    \[
    \h\dbar_k \hat\Pi \equiv \sum_j (\h\dbar\tilde\chi_j) \hat \Pi_{j}
    \chi_j + \tilde\chi_j \h\dbar_k\hat \Pi_{j} \chi_j.
    \]
    From Proposition~\ref{prop:pi_chi-holom} and its proof, we have
    $\h\dbar_k\hat \Pi_{j}\equiv 0$, and from~\eqref{equ:Pi_localized}
    we get $(\h\dbar\tilde\chi_j) \hat \Pi_{j} \chi_j \equiv 0$.  This
    proves item~\ref{item:propfibre-holo}.

    Finally, let $u\in \mathcal{H}^0(X;F_k)$. Restricting to
    $\Omega_j$, Lemma~\ref{lemm:pi_C-reproduisant} gives
    \[
    \chi_j(1-\hat\Pi_j)u \sim_a 0.
    \]
    Hence
    $\chi_j(1-\hat\Pi_j)\tilde\chi_j u \sim_a
    \chi_j(1-\hat\Pi_j)(\tilde\chi_j - 1)u$.
    By exponential localization,
    $\chi_j (1-\hat\Pi_j)(1-\tilde\chi_j)u \sim_a 0$. Hence
    $\chi_j (1-\hat\Pi_j)\tilde\chi_ju \sim_a 0$. By summing and using
    the selfadjointness of $\hat\Pi$, we get
    \[
    \hat \Pi u \sim_a \sum_j \chi_j\tilde\chi_j u = \sum_j \chi_j u = u.
    \]
    In addition, we see from Lemma~\ref{lemm:pi_C-reproduisant} that
    the estimates are actually uniform in $u$ when $\norm{u}_k=1$,
    since in fact $\chi_j(1-\hat\Pi_j) \hnegl 0$ on $H_\Phi(\Omega_j)$
    (Definition~\ref{defi:H-negligible}). Thus, we obtain
    item~\ref{item:propfibre-repro}.
  \end{proof}

  We now use some basic facts from the Hodge-Kodaira theory in order
  to prove that $\hat\Pi$ is close to $\Pi_k$.  The Hodge-Kodaira
  Laplacian, acting on sections of $\Lambda^{(0,1)} \otimes F_k$, is
  \begin{equation}
    \label{equ:laplacian}
    \Box_k := \dbar_k^* \dbar_k + \dbar_k \dbar_k^*,
  \end{equation}
  where the adjoint is taken with respect to the scalar product
  $\pscal{\cdot}{\cdot}_k$ defined above, extended to differential
  forms thanks to the metric $\omega_n$ on $X$. In the semiclassical
  setting $k\to\infty$, it is natural to consider the renormalized
  operator $\frac{1}{k^2}\Box_k$. The following well-known estimate
  can be found in~\cite[Section 7.3]{demailly-l2-notes}
  or~\cite[Corollary 4.3]{berndtsson10}.
  \begin{lemm}[Bochner-Kodaira-Nakano inequality]
    \label{lemm:nakano}
    For all $\tilde u\in \Cinf(X;\Lambda^{(0,1)} \otimes F_k)$,
    \[ \pscal{\frac{1}{k^2}\Box_k\tilde u}{\tilde u}_k \geq
    \frac{2}{k} \norm{\tilde u}_k^2.
    \]
  \end{lemm}
  Indeed, recall that $F_k=L^k$, and hence the curvature of $F_k$ is
  $k$ times the curvature of $L$; from~\cite[(7.7)]{demailly-l2-notes}
  or~\cite[Corollary 4.3]{berndtsson10}, one can check that the
  constant $2$ is due to the fact that the curvature form of $L$ is by
  definition $2\partial\dbar \Phi(x) = \frac{2}{i}\omega_n$.

  \begin{prop}
    \label{prop:fibre3}
    \begin{equation}
      \label{equ:approximate-fibre}
      \Pi_k -\hat\Pi={\cal O}(e^{-k/C_1}):\,  L^2(X;F_k)\to
      L^2(X;F_k) .
    \end{equation}
  \end{prop}
  \begin{proof}
    The argument is the same as for the proof of
    Theorem~\ref{theo:bcn2}. For any smooth section $\tilde u$ of
    $\Lambda^{(0,1)}\otimes F_k$, we get from~\eqref{equ:laplacian}
    that
    \begin{equation}
      \label{equ:fibre-norme-box}
      \pscal{k^{-2}\Box_k \tilde u}{\tilde u}_k
      = \tnorm{k^{-1}\dbar_k \tilde u}_k^2
      + \tnorm{k^{-1}\dbar^*_k \tilde u}_k^2.
    \end{equation}
    This, together with Lemma~\ref{lemm:nakano}, gives
    \begin{equation}
      \label{equ:minoration-kodaira}
      \pscal{k^{-2}\Box_k \tilde u}{\tilde u}_k \geq 
      \tilde c  \norm{\tilde u}^2_{H^1_k},
    \end{equation}
    where
    \begin{equation}
      \label{equ:norm-H1-fibre}
      \norm{\tilde u}^2_{H^1_k} := k^{-2}\tnorm{\dbar_k \tilde u}_k^2
      + k^{-2}\tnorm{\dbar^*_k \tilde u}_k^2 + k^{-1}\norm{\tilde
        u}_k^2.
    \end{equation}
    (This is analogous to ~\eqref{dbar.6}.)  Since $\Box_k$ is
    selfadjoint, Lemma~\ref{lemm:nakano} implies that we can define a
    bounded operator $(k^{-2}\Box_k)^{-1}$, acting on $(0,1)$-forms,
    with norm
    \begin{equation}
      (k^{-2}\Box_k)^{-1} 
      = \O(k) : L^2(X;\Lambda^{(0,1)} \otimes F_k) \to
      L^2(X;\Lambda^{(0,1)} \otimes F_k).
      \label{equ:fibre-inverse-l2}\,,
    \end{equation}
    and~\eqref{equ:minoration-kodaira} implies that
    $(k^{-2}\Box_k)^{-1}$ can be extended to
    \begin{equation}
      (k^{-2}\Box_k)^{-1}   = \O(1) : H^{-1}_k \to H^1_k\,,
      \label{equ:fibre-inverse-h1}
    \end{equation}
    where $H^{-1}_k$ is the dual to $H^1_k$, and the latter is the
    completion of the space of smooth sections for the
    norm~\eqref{equ:norm-H1-fibre}.  Consider the bounded selfadjoint
    operator $P$ on $L^2(X;F_k)$ given by the formula:
    \begin{align}
      P & = 1 - (k^{-1}\dbar_k^*) (k^{-2}\Box_k)^{-1} (k^{-1}\dbar_k) \\
        & =
          1 - \dbar_k^* \Box_k^{-1} \dbar_k \quad : L^2(X;F_k) \to L^2(X;F_k)\,.
          \label{eq:1}
    \end{align}
    First, we remark that if $\dbar_k u = 0$ then $Pu = u$.  Next, we
    have
    \begin{equation}
      \label{equ:dbarP}
      \dbar_k P = \dbar_k - \dbar_k \dbar_k^* \Box_k^{-1} \dbar_k 
      =  \dbar_k - (\Box_k - \dbar_k^*\dbar_k) \Box_k^{-1} \dbar_k.
    \end{equation}
    Since $\dbar_k^*\dbar_k$ commutes with $\Box_k$, it also commutes
    with $\Box_k^{-1}$, and the right-hand side of~\eqref{equ:dbarP}
    vanishes.  Therefore the range of $P$ is contained in
    $\mathcal{H}^0(X;F_k)$.  This entails that $P$ is the orthogonal
    projection onto $\mathcal{H}^0(X;F_k)$, \emph{i.e.} $P=\Pi_k$.

    Let $v\in L^2(X;F_k)$. In order to measure the lack of holomorphy
    of $\hat\Pi v$ we define
    \[
    u = Rv := (1-\Pi_k)\hat\Pi v = \dbar_k^* \Box_k^{-1} \dbar_k
    \hat\Pi v.
    \]
    From~\eqref{equ:fibre-inverse-h1}, we get
    \begin{equation}
      \label{equ:R-fibre}
      \norm{R v}_k \leq \frac{C}{\sqrt k} \tnorm{\dbar_k \hat\Pi v}_k.
    \end{equation}
    Since
    $\dbar_k (\hat\Pi - R ) = \dbar_k \hat\Pi - \dbar_k(1-\Pi_k)\hat
    \Pi = 0$, we have
    \[
    \Pi_k \hat\Pi = \Pi_k(\hat{\Pi}-R) + \Pi_k R = \hat{\Pi}-R + \Pi_k
    R = \hat\Pi - (1-\Pi_k)R.
    \]
    Using~\eqref{equ:R-fibre} with item (\emph{ii}) of
    Proposition~\ref{prop:fibre2}, we get
    \[
    \Pi_k \hat\Pi = \hat \Pi + \O(k^{1/2}e^{-k/C_1}) = \hat \Pi + \O(e^{-k/\tilde C_1}).
    \]
    Passing to the adjoints we get, using item \emph{(i)} of
    Proposition~\ref{prop:fibre2}:
    \begin{equation}
      \label{equ:Pi-Pik}
      \hat\Pi\Pi_k = \hat\Pi +\O(e^{-k/\tilde C_1}).
    \end{equation}
    On the other hand, using item \emph{(iii)} of
    Proposition~\ref{prop:fibre2} we have
    $\hat\Pi\Pi_k = \Pi_k +\O(e^{-k/C_1})$; in view
    of~\eqref{equ:Pi-Pik}, this gives the result.
  \end{proof}

  Finally, let $(x_0,y_0)\in X\times X$, and let $s_k$, $t_k$ be
  trivializing sections of $F_k$ near $x_0$ and $y_0$, respectively,
  as discussed above, see~\eqref{equ:kernel-fibre}.  Near $x_0$ we may
  define the compactly supported section
  $\tilde e_{x_0} = e_{x_0} s_{x_0}$ where $e_{x_0}$ is given
  by~\eqref{bcn.11}, in which $\Phi=\Phi_{x_0}$ is now defined by
  $\abs{s_{x_0}(x)}_{L}= e^{-\Phi_{x_0}(x)}$.  Similarly, we define
  $\tilde e_{y_0}$, using a function $\Phi_{y_0}$ defined near $y_0$.
  Lemma~\ref{bcn3} gives
  \begin{equation}
    \label{equ:peak-fibre}
    \pscal{\Pi_k \tilde e_{y_0}}{\tilde e_{x_0}}_k = b(x_0,\bar y_0)
    e^{-k(\Phi_{x_0}(x_0) + \Phi_{y_0}(y_0))}.
  \end{equation}
  If $x_0 \neq y_0$, one can find smooth cut-off functions
  $\chi_{x_0}$ and $\chi_{y_0}$, with disjoint supports, such that
  $\chi_{x_0} e_{x}= e_{x}$ and $\chi_{y_0} e_{y}= e_{y}$, for $(x,y)$
  close to $(x_0,y_0)$.  By~\eqref{equ:Pi_localized} we see that
  $\chi_{x_0}\hat\Pi \chi_{y_0} \equiv 0$. Hence
  \[
  \abs{\pscal{\hat \Pi \tilde e_{y}}{\tilde e_{x}}_k} \leq C e^{-k/C}
  \]
  for some $C>0$, and in this case item \ref{item:off-diag} of the
  theorem is a consequence of~\eqref{equ:peak-fibre}
  and~\eqref{equ:approximate-fibre} (and, of course,~\eqref{bcn.12}).

  Now let us assume that $x_0$ and $y_0$ belong to the same
  trivializing open set $\Omega_j$, and take
  $\Phi_{x_0} = \Phi_{y_0}$.  Thus, as in the proof of
  Theorem~\ref{bcn4}, we have
  \[
  \pscal{\hat\Pi_j e_{y_0}}{e_{x_0}}_{L^2_{\Phi}} =
  \frac{(2k)^n}{\pi^n} a(x_0,\bar{y}_0;k^{-1})e^{2k(\psi
    (x_0,\bar{y}_0) - (\Phi(x_0)+\Phi (y_0)))}.
  \]
  Using again~\eqref{equ:peak-fibre},~\eqref{equ:approximate-fibre}
  and~\eqref{bcn.12}, we obtain item~\ref{item:near-diag} of the
  theorem, finishing the proof in the case of a trivial factor $E$.

  If the bundle $E$ is not trivial, we need to replace the local
  weight $k\Phi(x)$ by $k\Phi(x) + \Phi_G(x)$, where
  $\Phi_G(x) := - \ln G(x)$.  This amounts to replacing the symbol
  $a_\h(x,\bar y)$ by
  $a_\h^G(x,\bar y) := a_\h(x,\bar y) e^{2\psi_G(x,\bar y)}$, where
  $\psi_G(x, y)$ is the holomorphic function defined for $x$ close to
  $y$ by $\psi_G(x, \bar x) = \Phi_G(x)$, as in~\eqref{equ:bar-psi}.
  Similarly, $b_k(x,\bar y)$ should be replaced by
  $b_k(x,\bar y)e^{2\psi_G(x,\bar y)}$, see~\eqref{equ:kernel-fibre}.
  Since $\Phi_G$ does not depend on $k$, the ellipticity estimates of
  Lemma~\ref{lemm:nakano} hold with $ck$ replaced by $ck-C$ for some
  $C>0$. Hence if $k$ is large enough, items~\ref{item:off-diag}
  and~\ref{item:near-diag} of the theorem still hold true, with a
  possibly different constant $C$.
\end{proof}

\begin{remark}
  Using the auxiliary bundle $E$, one obtains that
  Theorem~\ref{theo:line-bundle} holds for an arbitrary analytic
  volume form instead of the natural Kähler one. Indeed, the analytic
  factor in front of $\omega_n$ can be incorporated in the Hermitian
  metric of $E$.
\end{remark}
\begin{remark}
  The case of compact Riemann surfaces with constant curvature was
  already obtained by Berman~\cite{berman12}.  After this article was
  written, we have received preliminary works by Deleporte, where the
  author first studies the particular case of Kähler manifolds with
  constant sectional curvature, for which explicit computations can be
  done that don't require microlocal
  analysis~\cite{deleporte-const-curv18}, and then obtains a new proof
  of Theorem~\ref{theo:line-bundle}, see~\cite{deleporte18}. Even more
  recently, the submission of our work was the incentive for other
  authors to find a more elementary proof of
  Theorem~\ref{theo:line-bundle}: see~\cite{charles2019analytic,
    hezari2019property}.
\end{remark}

\appendix

\section{Quick review of \texorpdfstring{$\overline{\partial }$}{dbar}
  on \texorpdfstring{$L^2_\Phi (\CM^n)$}{L2phi(CN)}.}
\label{app:dbar}

We review Hörmander's approach~\cite{hormander65} to the
$\overline{\partial }$-problem in the simple case of functions on
$\CM^n$ and with more explicit reference to the Hodge Laplacian. A
more detailed presentation can be found in the appendix
of~\cite{sj-96} and here we only give a short r\'esum\'e.

Let $\Phi :\, \CM^n\to \RM$ be a function of class $ C^2$
(in~\cite{sj-96} we treat the slightly more general case of $C^{1,1}$)
such that
\begin{equation}\label{dbar.1}
  \nabla  ^\alpha \Phi \in L^\infty (\CM^n),\hbox{ for }|\alpha |=2,
\end{equation}
\begin{equation}\label{dbar.2}
  \Phi ''_{\bar{x},x}\ge 1/C, \hbox{ for some constant } C>0.
\end{equation}

The problem $\h\overline{\partial }u=v$ in the spaces $L^2_\Phi $ is
equivalent to
\begin{equation}\label{dbar.3}
  \overline{\partial }_\Phi (e^{-\Phi/\h }u) = e^{-\Phi /\h} v
\end{equation}
in the usual (unweighted) $L^2$-spaces, where
\[
\overline{\partial }_\Phi =e^{-\Phi /\h }\circ \h
\overline{\partial}\circ e^{\Phi /\h }= \h \overline{\partial } +
(\overline{\partial }\Phi)^\wedge ,
\]
and $\omega ^\wedge$ indicates left exterior multiplication with the
$(0,1)$-form $\omega $. The corresponding real adjoint operator will
be denoted by $\omega ^\rfloor$ (contraction with $\omega $). Here we
use the standard point-wise real scalar product on real $p$-forms,
extended bilinearly to the complexified space. Recall that
\[
\langle \DD x_j|\DD x_k\rangle = \langle
\DD\bar{x}_j|\DD\bar{x}_k\rangle=0,\ \langle \DD
x_j|\DD\bar{x}_k\rangle=\delta _{j,k},\ \ \CM^n=\CM^n_{x_1,\dots,x_n}.
\]
Write
\[
\overline{\partial }_\Phi = \sum Z_j\otimes\DD \bar{x}_j^\wedge, \ \
\overline{\partial }^*_\Phi =\sum Z_j^*\otimes \DD x_j^\rfloor,
\]
\[
Z_j=\h \partial _{\bar{x}_j}+\partial _{\bar{x}_j}\Phi ,\ \
Z_j^*=-\h \partial _{x_j}+\partial _{x_j}\Phi .
\]
Recall that $\overline{\partial }$ and $\overline{\partial }_\Phi $
take $(0,q)$-forms to $(0,q+1)$ forms and define complexes:
$\overline{\partial }^2=0$, $\overline{\partial }_\Phi ^2=0$. The
Hodge Laplacian is
\[
\widetilde{\Box}_\Phi =\overline{\partial }_\Phi \overline{\partial
}^*_\Phi + \overline{\partial }^*_\Phi \overline{\partial }_\Phi.
\]
It preserves $(0,q)$-forms and a standard calculation gives
\begin{equation}\label{dbar.4}
  \widetilde{\Box}_\Phi  =\left( \sum_1^n Z_j^*Z_j \right)\otimes
  1+\sum_{j,k}[Z_j,Z_k^*]\DD\bar{x}_j^\wedge \DD x_k^\rfloor,\
  [Z_j,Z_k^*]=2\h \partial _{\bar{x}_j}\partial _{x_k}\Phi .
\end{equation}
Identifying the $(0,1)$-form $\sum u_j\DD\bar{x}_j$ with the
$\CM^n$-valued function $(u_1,\dots,u_n)^\mathrm{t}$, we get for the
restriction $\widetilde{\Box}^{(1)}_\Phi $ of $\widetilde{\Box}_\Phi $
to $(0,1)$-forms:
\begin{equation}\label{dbar.5}
  \widetilde{\Box}^{(1)}_\Phi =\left( \sum_1^n Z_j^*Z_j \right)\otimes
  1+2\h \Phi ''_{\bar{x},x}.
\end{equation} It follows that 
\begin{equation}\label{dbar.6}
  \| u\|_{H^1}^2\le {\cal O}(1) \pscal{\widetilde{\Box}^{(1)}_\Phi u}{u},\ u\in C_0^\infty
  (\CM^n;\wedge ^{0,1}\CM^{n}),
\end{equation}
where
\begin{equation}\label{dbar.7}
  \| u \|_{H^1}:=\left( \sum\left( \| Z_ju\| ^2 +\| Z^*_ju\|
      ^2\right)+\h \| u\|^2 \right)^{\frac{1}{2}}.
\end{equation}
Let $H^1\subset L^2(\CM^n;\wedge^{0,1}\CM^n)$ be the Hilbert space
obtained as the completion of
$ C_0^\infty (\CM^n;\wedge ^{0,1}\CM^{n})$ for the
$H^1$-norm. Sometimes we drop the notation $\wedge^{0,1}\CM^n$, when
it is clear that we work with $(0,1)$-forms. The inclusion map
$H^1\to H^0:=L^2$ is of norm ${\cal O}(\h^{-1/2})$ and the same holds
for the dual inclusion $H^0\to H^{-1}$, where $H^{-1}$ denotes the
dual of $H^1$ for the $L^2$-inner product.

\par From (\ref{dbar.6}) we get with standard variational arguments
that
\begin{equation}\label{dbar.8}
  \widetilde{\Box}_\Phi ^{(1)}:H^1\to H^{-1}
\end{equation}
is bijective with with inverse satisfying
\begin{equation}\label{dbar.9}
  (\widetilde{\Box}_\Phi ^{(1)})^{-1}=\begin{cases}
    {\cal O}(1):\, H^{-1}\to H^1,\\
    {\cal O}(\h ^{-1/2}):\, H^0\to H^1,\ H^{-1}\to H^0,\\
    {\cal O}(\h ^{-1}):\, H^0\to H^0.
  \end{cases}
\end{equation}

\par We saw in the appendix of~\cite{sj-96} that if
$v\in L^2(\CM^n;\wedge^{0,1}\CM^n)$ satisfies
$\overline{\partial }_\Phi v=0$, then
$u=\overline{\partial }^*_{\Phi }(\widetilde{\Box}_\Phi ^{(1)})^{-1}v$
solves
\begin{equation}\label{dbar.10}
  \overline{\partial }_\Phi u=v,
\end{equation}
and
\begin{equation}\label{dbar.11}
  \| u\|_{H^0}\le {\cal O}(1)\| v\|_{H^{-1}}\le {\cal O}(\h ^{-1/2})\| v\|_{H^0}.
\end{equation}

\par If $v\in L^2_\Phi (\CM^n;\wedge^{0,1}\CM^n)$ and
$\overline{\partial }v=0$, then
\[
u=e^{\Phi /\h }\overline{\partial }_\Phi ^*(\widetilde{\Box}_{\Phi
}^{(1)})^{-1}e^{-\Phi /\h }v
\]
solves,
\begin{equation}\label{dbar.12}
  \h \overline{\partial }u=v,
\end{equation}
and
\begin{equation}\label{dbar.13}
  \| u\|_{L^2_\Phi }\le {\cal O}(\h ^{-1/2})\| v\|_{L^2_\Phi }.
\end{equation}

\par The orthogonal projection
$\widetilde{\Pi }:\, L^2(\CM^n)\to L^2\cap{\cal N}(\overline{\partial
}_\Phi )$ on the level of 0-forms, is given by
\begin{equation}\label{dbar.14}
  \widetilde{\Pi }=1-\overline{\partial }^*_\Phi (\widetilde{\Box}_\Phi
  ^{(1)})^{-1}\overline{\partial }_\Phi .
\end{equation}
See (A.14) in~\cite{sj-96}.

\par
We finally translate the results to the setting of $L^2_\Phi $, noting
that $L^2\ni u\mapsto e^{\Phi /\h}u\in L^2_\Phi $ is unitary and maps
${\cal N}(\overline{\partial }_\Phi )$ to
${\cal N}(\h \overline{\partial })$. Correspondingly we have the
unitary conjugations \begin{equation}\label{dbar.15} \begin{split}
    \overline{\partial }_\Phi &=e^{-\Phi /\h} \h\overline{\partial
    }e^{\Phi /\h},\\
    \overline{\partial }_\Phi^{*} &=e^{-\Phi /\h}
    (\h\overline{\partial })^{\Phi ,*}e^{\Phi
      /\h}, \end{split} \end{equation} where the exponent $(\Phi ,*)$
indicates the adjoint for the $L^2_\Phi $ norms.  Note that the last
relation gives,
\[
(\h \overline{\partial })^{\Phi ,*}=e^{\Phi /\h}\overline{\partial
}_\Phi ^* e^{-\Phi /\h}=e^{2\Phi /\h}(\h \overline{\partial })^{\Phi
  ,*} e^{-2\Phi /\h},
\]
which is easy to show directly.  \par Also by unitarity,
\[
\Box_{\Phi }:=(\h\overline{\partial })^{\Phi ,*}\h \overline{\partial
}+\h \overline{\partial }(\h\overline{\partial })^{\Phi ,*}
\]
fulfills
\begin{equation}\label{dbar.16}\Box_\Phi =e^{\Phi
    /\h}\widetilde{\Box}_\Phi e^{-\Phi /h}, 
\end{equation} hence
\[
(\Box^{(1)}_\Phi)^{-1} =e^{\Phi /\h}(\widetilde{\Box}_\Phi^{(1)})^{-1}
e^{-\Phi /h}.
\]
By unitarity and (\ref{dbar.14}) the orthogonal projection
\[
\Pi :\, L^2_\Phi (\CM^n)\to L^2_\Phi (\CM^n)\cap {\cal N}(\h
\overline{\partial })
\]
is given by
\begin{equation}\label{dbar.17}
  \Pi =e^{\Phi /\h }\widetilde{\Pi }e^{-\Phi /\h }=1-
  (\h\overline{\partial })^{\Phi,*} (\Box_\Phi
  ^{(1)})^{-1}\h\overline{\partial } .
\end{equation}

\par In line with the unitary relations (\ref{dbar.15}),
(\ref{dbar.16}) we have
\begin{equation}\label{dbar.18}
  \begin{split}
    Y_j:&=e^{\Phi /\h}Z_je^{-\Phi /\h}=h\partial _{\bar{x}_j},\\
    Y_j^{\Phi ,*}&=e^{\Phi /\h}Z_j^*e^{-\Phi /\h}=-h\partial
    _{x_j}+2\partial _x\Phi
  \end{split}
\end{equation} 
and the continuity statements (\ref{dbar.6}), (\ref{dbar.8}),
(\ref{dbar.9}) remain valid for $\Box_\Phi $ if we redefine the spaces
$H^k$ by replacing the unweighted $L^2$ norms with $L^2_\Phi $ norms
and replace $Z_j$, $Z_j^*$ with $Y_j$, $Y_j^{\Phi ,*}$.

\par If $v\in L^2_\Phi (\CM^n;\wedge^{0,1}\CM^n)$ and
$\overline{\partial }v=0$, then from (\ref{dbar.10}), (\ref{dbar.11})
and the unitary conjugations above, we see that
\[
u= (\h\overline{\partial })^{\Phi ,*}(\Box_\Phi ^{(1)})^{-1}v
\]
solves,
\begin{equation}\label{dbar.12}
  \h \overline{\partial }u=v,
\end{equation}
and
\begin{equation}\label{dbar.13}
  \| u\|_{L^2_\Phi }\le {\cal O}(\h ^{-1/2})\| v\|_{L^2_\Phi }.
\end{equation}
  
\section{Direct study of \texorpdfstring{$A$}{A} in
  \texorpdfstring{\eqref{equ:defA}}{equ:defA}}
\label{dsa}

Here we perform a more direct study of the operator $A$ in
Subsection~\ref{sec:transco}. Recall that
\[
U_{1/2}=\exp \left(\frac{i}{2\h}\h D_\theta \cdot (\h D_y-\h D_x)
\right) ,
\]
so that
\[
U_{1/2}=\mathcal{F}^{-1}\circ \exp \left(\frac{i}{2\h}\theta^* \cdot
  (y^*-x^*) \right)\circ \mathcal{ F} ,
\]
where $\mathcal{F} = \mathcal{F}_\h$ denotes the usual semiclassical
Fourier transform on $\RM^{3n}$. Hence
\[
U_{1/2}u=K*u,
\]
where
\begin{equation}\label{dsa.1}
  K=\mathcal{ F}^{-1}\left( \exp \left(\frac{i}{2\h}\theta^* \cdot 
      (y^*-x^*) \right) \right) .
\end{equation}
To compute $K$ we first diagonalize the quadratic form
$q=\theta ^*\cdot (y^*-x^*)$ by means of a real orthogonal change of
variables.  Writing $q$ as a difference of two squares and adjusting a
parameter, we find
\begin{equation}\label{dsa.2}
  \theta ^*\cdot (y^*-x^*)=\frac{1}{\sqrt{2}}(\xi _1^2-\xi _2^2),
\end{equation}
where
\[
\begin{split}
  \xi _1&=\frac{x^*}{2}-\frac{y^*}{2}-\frac{\theta ^*}{\sqrt{2}},\\
  \xi _2&=-\frac{x^*}{2}+\frac{y^*}{2}-\frac{\theta ^*}{\sqrt{2}}.
\end{split}
\]
Adding the third coordinate
\[
\xi _3=\frac{x^*}{\sqrt{2}}+\frac{y^*}{\sqrt{2}},
\]
we get
\[
\begin{pmatrix}\xi _1\\\xi _2\\\xi _3\end{pmatrix}=V\begin{pmatrix}
  x^*\\ y^*\\ \theta ^*\end{pmatrix},
\]
where
\begin{equation}\label{dsa.3}
  V=\begin{pmatrix}\frac{1}{2} & -\frac{1}{2} & -\frac{1}{\sqrt{2}}\\
    -\frac{1}{2} & \frac{1}{2} & -\frac{1}{\sqrt{2}}\\
    \frac{1}{\sqrt{2}} & \frac{1}{\sqrt{2}} & 0
  \end{pmatrix}
\end{equation}
is orthogonal with determinant 1. Thus,
\begin{equation}\label{dsa.4}
  q\circ V^{-1}(\xi )=\frac{1}{\sqrt{2}}(\xi _1^2-\xi _2^2).
\end{equation}
Let $x_1,x_2,x_3$ be the coordinates on
$\RM^n_x\times \RM^n_y\times \RM^n_\theta $ that are dual to
$\xi _1,\xi _2,\xi _3$.  In these coordinates,
\[
K=\mathcal{F}^{-1}\left( \exp \frac{i}{2\h}\frac{1}{\sqrt{2}}\left(
    \xi _1^2-\xi _2^2 \right) \right) = \frac{1}{(\pi \h)^n}\exp
\left( -\frac{i}{\h\sqrt{2}} \left(x_1^2-x_2^2 \right) \right)\delta
(x_3).
\]
In order to get $K$ in the coordinates $(x,y,\theta )$ we perform the
dual change of variables,
\[
\begin{pmatrix} x \\ y\\
  \theta \end{pmatrix}=V^\mathrm{t}\begin{pmatrix}x_1 \\ x_2 \\
  x_3\end{pmatrix},
\]
or equivalently,
\[
\begin{pmatrix}x_1 \\ x_2 \\ x_3\end{pmatrix}=V
\begin{pmatrix} x \\ y\\
  \theta \end{pmatrix},
\]
since $V$ is orthogonal; $(V^\mathrm{t})^{-1}=V$. More explicitly,
\[
\begin{split}
  x_1&=\frac{x}{2}-\frac{y}{2}-\frac{\theta }{\sqrt{2}},\\
  x_2&=-\frac{x}{2}+\frac{y}{2}-\frac{\theta }{\sqrt{2}},\\
  x_3&=\frac{x}{\sqrt{2}}+\frac{y}{\sqrt{2}}.
\end{split}
\]
Thus,
\begin{equation}\label{dsa.5}
  \begin{split}
    K=&\frac{1}{(\pi \h)^n}\times\\ &\exp \left(-\frac{i}{\h\sqrt{2}}
      \left( \left( \frac{x}{2}-\frac{y}{2}-\frac{\theta
          }{\sqrt{2}}\right)^2-\left(-\frac{x}{2}+\frac{y}{2}-\frac{\theta
          }{\sqrt{2}} \right)^2\right) \right) \delta
    \left(\frac{x+y}{\sqrt{2}}
    \right)\\
    =&\frac{1}{(\pi \h)^n} \exp \left( \frac{i}{\h}(x-y)\cdot \theta
    \right) \delta \left(\frac{x+y}{\sqrt{2}} \right).
  \end{split}
\end{equation}
Noticing that $\delta (t/\sqrt{2})=\sqrt{2}^n\delta (t)$ on $\RM^n$,
we get for $U_{1/2}u=K*u$:
\[
\begin{split}
  U_{1/2}u(x,y,\theta ) & =\frac{1}{(\pi \h)^n}\iiint
  e^{\frac{i}{\h}(x-\widetilde{x}-y+\widetilde{y})\cdot (\theta
    -\widetilde{\theta })}\delta
  \left(\frac{x-\widetilde{x}+y-\widetilde{y}}{\sqrt{2}}
  \right)u(\widetilde{x},\widetilde{y},\widetilde{\theta
  })\DD\widetilde{x}\DD\widetilde{y}\DD\widetilde{\theta }\\
  & = \left( \frac{\sqrt{2}}{\pi \h} \right)^n \iint
  e^{\frac{i}{\h}(x-\widetilde{x}-y+x-\widetilde{x}+y)\cdot (\theta
    -\widetilde{\theta })}
  u(\widetilde{x},x+y-\widetilde{x},\widetilde{\theta})
  \DD\widetilde{x}\DD\widetilde{\theta },
\end{split}
\]
\begin{equation}\label{dsa.6}
  U_{1/2}u(x,y,\theta )=\left(\frac{\sqrt{2}}{\pi \h} \right)^n\iint
  e^{\frac{2i}{\h}(x-\widetilde{x})\cdot (\theta
    -\widetilde{\theta })}
  u(\widetilde{x},x+y-\widetilde{x},\widetilde{\theta
  })\DD\widetilde{x}\DD\widetilde{\theta }.
\end{equation}

\par Recall from Proposition \ref{prop_lien_op_brg_op_weyl} that
\begin{equation}\label{dsa.7}
  \begin{split}
    \left(W ^*u \right)(x,y,\theta )&=u(x,y,w(x,y,\theta )),\\
    \left( \gamma ^*u \right)(x,\theta )&=u(x,x,\theta ),
  \end{split}
\end{equation}
and from the beginning of the proof of Proposition \ref{prop:S}, that
\begin{equation}\label{dsa.8}
  \left(\pi ^*v \right)(x,y,w)=v(x,w).
\end{equation}
This gives first that $W ^*\pi ^*u(x,y,\theta )=u(x,w(x,y,\theta ))$
and then with (\ref{dsa.6}) that
\begin{multline}
  U_{1/2}\widetilde{J}W ^*\pi ^*u(x,y,\theta )\\
  =\left(\frac{\sqrt{2}}{\pi \h} \right)^n\iint
  e^{\frac{2i}{\h}(x-\widetilde{x})\cdot (\theta -\widetilde{\theta
    })}\widetilde{J}(\widetilde{x},x+y-\widetilde{x},\widetilde{\theta
  })u(\widetilde{x},w(\widetilde{x},x+y-\widetilde{x},\widetilde{\theta
  }))\DD\widetilde{x}\DD\widetilde{\theta },
\end{multline}
and hence,
\begin{multline}\label{dsa.9}
  Au(x,\theta )=\gamma ^*U_{1/2}\widetilde{J}W ^*\pi
  ^*u(x,\theta )\\
  =\left(\frac{\sqrt{2}}{\pi \h} \right)^n\iint
  e^{\frac{2i}{\h}(x-\widetilde{x})\cdot (\theta -\widetilde{\theta
    })}\widetilde{J}(\widetilde{x},2x-\widetilde{x},\widetilde{\theta
  })u(\widetilde{x},w(\widetilde{x},2x-\widetilde{x},\widetilde{\theta
  }))\DD\widetilde{x}\DD\widetilde{\theta}
\end{multline}
In this integral, we replace the integration variable
$\widetilde{\theta }$ with
$\widetilde{w}:=w(\widetilde{x},2x-\widetilde{x},\widetilde{\theta})$,
so that
\[
\widetilde{\theta }=\theta
(\widetilde{x},2x-\widetilde{x},\widetilde{w}),\ \
\DD\widetilde{\theta} = \det \left(\frac{\partial \theta }{\partial w}
\right)(\widetilde{x},2x-\widetilde{x},\widetilde{w})
\DD\widetilde{w},
\]
and get
\begin{multline}
  \label{dsa.10}
  Au(x,\theta )=\\
  \left(\frac{\sqrt{2}}{\pi \h} \right)^n \iint
  e^{\frac{i}{\h}F(x,\theta ;\widetilde{x},\widetilde{w})}
  \widetilde{J}(\widetilde{x},2x-\widetilde{x},\theta
  (\widetilde{x},2x-\widetilde{x},\widetilde{w}))\times \\
  \det \left(\frac{\partial \theta }{\partial w} \right)
  (\widetilde{x},2x-\widetilde{x},\widetilde{w})
  u(\widetilde{x},\widetilde{w}) \DD\widetilde{x}\DD\widetilde{w},
\end{multline}
where
\begin{equation}\label{dsa.11}
  F(x,\theta ;\widetilde{x},\widetilde{w})=2(\theta -\theta
  (\widetilde{x},2x-\widetilde{x},\widetilde{w}))\cdot (x-\widetilde{x}).
\end{equation}

There are no fiber variables present in the
representation~\eqref{dsa.10} of the Fourier integral operator $A$, so
the phase generates a canonical relation
\begin{equation}\label{dsa.12}
  C_A:\ (\widetilde{x},\widetilde{w};-\partial
  _{\widetilde{x}}F,-\partial _{\widetilde{w}}F)\mapsto (x,\theta
  ;\partial _x F,\partial _\theta F).
\end{equation}
Recall from the identity after \eqref{eq_Kuranishi_trick} that
\[
\theta (x,y,\theta) =
\frac{2}{i}\psi'_x((x+y)/2,\widetilde{w})+\mathcal{O}((x-y)^2),
\]
hence
\[
\theta (\widetilde{x},2x-\widetilde{x},\widetilde{w}) =
\frac{2}{i}\psi'_x(x,\widetilde{w})+\mathcal{O}((x-\widetilde{x})^2),
\]
\begin{equation}\label{dsa.13}
  F(x,\theta ;\widetilde{x},\widetilde{w})=2\left(\theta
    -\frac{2}{i}\psi '_x(x,\widetilde{w}) \right)\cdot
  (x-\widetilde{x})+\mathcal{O}((x-\widetilde{x})^3),
\end{equation}
\begin{equation}\label{dsa.14}
  \begin{split}
    -\partial_{\widetilde{x}}F&= 2\left(\theta -\frac{2}{i}\psi
      '_x(x,\widetilde{w}) \right)+\mathcal{ O}((x-\widetilde{x})^2),\\
    -\partial _{\widetilde{w}}F&=\frac{4}{i}\psi
    ''_{\widetilde{w},x}(x,\widetilde{w})(x-\widetilde{x})+\mathcal{
      O}((x-\widetilde{x})^3)\\
    \partial _xF&=2\left(\theta -\frac{2}{i}\psi
      '_x(x,\widetilde{w})\right) -\frac{4}{i}\psi
    ''_{x,x}(x,\widetilde{w})(x-\widetilde{x})
    +\mathcal{ O}((x-\widetilde{x})^2)\\
    \partial _\theta F&=2(x-\widetilde{x})+\mathcal{
      O}((x-\widetilde{x})^3),
  \end{split}
\end{equation}
\begin{equation}\label{dsa.15}
  F''_{x,\theta
    ;\widetilde{x},\widetilde{w}}=\begin{pmatrix}F''_{x,\widetilde{x}}
    &F''_{x,\widetilde{w}}\\
    F''_{\theta ,\widetilde{w}} &F''_{\theta ,\widetilde{w}}\end{pmatrix}
  =\begin{pmatrix}
    F''_{x,\widetilde{x}} &-\frac{4}{i}\psi ''_{x,\widetilde{w}}+\mathcal{
      O}(x-\widetilde{x})\\
    -2+\mathcal{ O}((x-\widetilde{x})^2) &\mathcal{ O}((x-\widetilde{x})^3)
  \end{pmatrix} .
\end{equation}

\par Recall that $\psi ''_{x,\widetilde{w}}=\Phi ''_{x,\bar{x}}$ when
$\widetilde{w}=\bar{x}$, so
$F''_{x,\theta ;\widetilde{x},\widetilde{w}}$ is invertible when
$\widetilde{w}-\bar{x}$ and $x-\widetilde{x}$ are small. In this
region, $C_A$ is therefore equal to the graph of a canonical
transformation locally. Still when $|x-\widetilde{x}|$ is small, we
have the equivalences
\[
\begin{cases}
  \partial _{\widetilde{x}}F=0,\\ \partial
  _{\widetilde{w}}F=0 \end{cases} \Longleftrightarrow
\begin{cases}x-\widetilde{x}=0,\\ \theta =\frac{2}{i}\psi
  '_x(x,\widetilde{w})\end{cases} \Longleftrightarrow
\begin{cases}\partial _xF=0,\\ \partial _\theta F=0\end{cases} ,
\]
so $C_A$ maps the zero-section $\widetilde{x}^*=\widetilde{w}^*=0$ to
the zero-section $x^*=\theta ^*=0$.

\par In particular, if we restrict the attention to a neighborhood of
a point given by
\[
x=\widetilde{x}=x_0,\ \widetilde{w}=\bar{x}_0,\ \theta
=\frac{2}{i}\partial _x\Phi (x_0),\
\widetilde{x}^*=\widetilde{w}^*=x^*=\theta ^*=0,
\]
we see that
\begin{equation}\label{dsa.16}
  C_A:\ (x_0,\bar{x}_0;0,0)\mapsto (x_0,(2/i)\partial _x\Phi (x_0);0,0)
\end{equation}
and that in a neighborhood of this point $C_A$ coincides with the
graph of a canonical transformation which maps the zero section over a
neighborhood of $(x_0,\bar{x}_0)$ to the zero section over a
neighborhood of $(x_0,(2/i)\partial _x\Phi (x_0))$.

\bibliographystyle{abbrv} \bibliography{biblio}

\end{document}